\documentclass[11pt,a4paper, twoside, reqno]{amsart}
\usepackage{booktabs}
\usepackage[english]{babel}
\usepackage[T1]{fontenc}

\usepackage[usenames,dvipsnames]{xcolor}
\usepackage[colorlinks=true,linkcolor=NavyBlue,urlcolor=RoyalBlue,citecolor=PineGreen,%hidelinks,
	hypertexnames=false]{hyperref}
 \usepackage{upgreek}
\allowdisplaybreaks
\usepackage{amsfonts}
\usepackage{amsmath}
\usepackage{amssymb}
\usepackage{amsthm}
\usepackage{latexsym}
\usepackage{stmaryrd}
 \usepackage{tgpagella} % text only
 \usepackage{mathpazo}
 \usepackage[scaled]{beramono}
\usepackage{mathtools}
\usepackage{bm}
\usepackage{mathbbol}
\usepackage{shuffle}
\usepackage{rotating}
\usepackage{tikz}
\usepackage{tikz-cd} \usepackage{subfigure}
\usepackage{verbatim}
\usepackage[margin=0.7in]{geometry}
\usepackage[font=small]{caption}
\usepackage{adjustbox}
\usepackage{enumerate}
\usepackage{cleveref}
\usepackage{svg}
\usepackage{mathabx}
\usepackage{tcolorbox}

\renewcommand{\arraystretch}{1.5}

\newcommand{\R}{\mathbb{R}}
\newcommand{\rmd}{{\rm d}}
\newcommand{\C}{\mathcal{C}}

\theoremstyle{plain}
\newtheorem*{assumption*}{Assumption}

\DeclareFontFamily{U}{MnSymbolC}{}
\DeclareFontShape{U}{MnSymbolC}{m}{n}{
	<-5.5> MnSymbolC5
	<5.5-6.5> MnSymbolC6
	<6.5-7.5> MnSymbolC7
	<7.5-8.5> MnSymbolC8
	<8.5-9.5> MnSymbolC9
	<9.5-11.5> MnSymbolC10
	<11.5-> MnSymbolCb12
}{}

\usepackage[linecolor=white,backgroundcolor=white,bordercolor=white,textsize=tiny]{todonotes}

\makeatletter
\renewcommand{\tocsection}[3]{%
	\indentlabel{\@ifnotempty{#2}{\bfseries\ignorespaces#1 #2.\,\,}}\bfseries#3}
\renewcommand{\tocsubsection}[3]{%
	\indentlabel{\@ifnotempty{#2}{\ignorespaces#1 #2\quad}}#3}
\renewcommand{\tocsubsubsection}[3]{%
	\quad\quad\quad\indentlabel{\@ifnotempty{#2}{\ignorespaces#1 #2\quad}}#3}

\newcommand\@dotsep{4.5}
\def\@tocline#1#2#3#4#5#6#7{\relax
	\ifnum #1>\c@tocdepth % then omit
	\else
		\par \addpenalty\@secpenalty\addvspace{#2}%
		\begingroup \hyphenpenalty\@M
		\@ifempty{#4}{%
			\@tempdima\csname r@tocindent\number#1\endcsname\relax
		}{%
			\@tempdima#4\relax
		}%
		\parindent\z@ \leftskip#3\relax \advance\leftskip\@tempdima\relax
		\rightskip\@pnumwidth plus1em \parfillskip-\@pnumwidth
		#5\leavevmode\hskip-\@tempdima{#6}\nobreak
		\leaders\hbox{$\m@th\mkern \@dotsep mu\hbox{.}\mkern \@dotsep mu$}\hfill
		\nobreak
		\hbox to\@pnumwidth{\@tocpagenum{\ifnum#1=1\bfseries\fi#7}}\par% <-- \bfseries for \Szection page
		\nobreak
		\endgroup
	\fi}
\AtBeginDocument{%
	\expandafter\renewcommand\csname r@tocindent0\endcsname{0pt}
}

\def\l@subsection{\@tocline{2}{0pt}{2.5pc}{5pc}{}}
\makeatother

\makeatletter
\newcommand{\oset}[3][0.25ex]{%
  \mathrel{\mathop{#3}\limits^{
    \vbox to#1{\kern-2\ex@
    \hbox{$\scriptstyle#2$}\vss}}}}
\makeatother

\def\cA{\mathcal{A}}

\def\tr{\mathrm{tr}}

\def\1{\mathbf{1}}

\newtheoremstyle{bfnote}%
{}{}%
{\itshape}{}%
{\sf}{.}%
{ }%
{\thmname{#1}\thmnumber{ #2}\thmnote{ (#3)}}

\newtheorem{theorem}{Theorem}[section]

\newtheorem{proposition}[theorem]{Proposition}
\newtheorem*{proposition*}{Proposition}
\newtheorem{corollary}[theorem]{Corollary}
\newtheorem*{theorem*}{Theorem}
\newtheorem{lemma}[theorem]{Lemma}
\newtheorem{remark}[theorem]{Remark}

\def\Psi{{\color{blue} \psi}}

\def\trphi{{\rm tr}_2\varphi^{(2)}}
\def\pingx{{\pi^{g,x}_n}}

\title{Rates of convergence in the free multiplicative Central Limit Theorem}
\author[M. Banna]{Marwa Banna}
\address{New York University Abu Dhabi, Division of Science, Mathematics, Abu Dhabi, UAE}
\email{marwa.banna@nyu.edu}

\author[N. Gilliers]{Nicolas Gilliers}
\address{New York University Abu Dhabi, Division of Science, Mathematics, Abu Dhabi, UAE}
\email{nag9000@nyu.edu}

\author[P.-L. Tseng]{Pei-Lun Tseng}
\address{New York University Abu Dhabi, Division of Science, Mathematics, Abu Dhabi, UAE}
\email{pt2270@nyu.edu}

\date{\today}

\thanks{...}

\begin{document}
\maketitle
\begin{abstract}
We provide the first quantitative estimates for the rate of convergence in the free multiplicative  central limit theorem (CLT),  in terms of the Kolmogorov and $r$-Wasserstein distances for $r \geq 1$. While the free additive CLT has been thoroughly studied, including convergence rates, the multiplicative setting remained open in this regard. We consider products of the form
\[
\pi_n^{g,n^{-1/2}x} := g\left(\frac{x_1}{\sqrt{n}}\right) \cdots g\left(\frac{x_n}{\sqrt{n}}\right),
\]
where $x_1, \dots, x_n$ are freely independent self-adjoint centered operators with common variance $\sigma^2$ and $g \colon \mathbb{R} \to \mathbb{C}$ satisfies certain regularity and integrability conditions. We quantify the deviation of the singular value distribution of $\pi_n^{g,x}$ from the free positive semicircular law, with bounds depending only on the moments of the underlying variables. Additionally, we present a combinatorial proof of the free multiplicative CLT that extends to the unbounded setting.
\end{abstract}

\section{Introduction and main results}

This work investigates the rate of convergence in the free multiplicative  central limit theorem (CLT) in terms of the Kolmogorov and $r$-Wasserstein distances, for $r\ge1$.  While the rate of convergence for the central limit theorem in the context of free additive convolution has been studied by \cite{kargin07free-additive,speicher07rate}, with numerous generalizations considered, see \cite{BannaMaiBerry,ChGo-8,ChGo-13,MaSakuma,neufeld24}, our work presents the first quantitative estimate for the multiplicative case. To be more precise, we consider a sequence of free self-adjoint operators $x_1,\dots,x_n$ in a tracial probability space $(\mathcal{A},\varphi)$ and let $g\colon \mathbb{R}\to \mathbb{C}$ be a function satisfying certain regularity and integrability conditions. Our main goal is to quantify the deviations of the singular values of the product
\begin{equation}\label{eqn:product}
\pi_{n}^{g,n^{-1/2}x}:=g\left(\frac{x_1}{\sqrt{n}}\right) \cdots g\left(\frac{x_n}{\sqrt{n}}\right)  
\end{equation}
from the free multiplicative semicircular law when $n$ goes to infinity. The free multiplicative  CLT has been investigated over the years. The work of %Bercovici and Voiculescu 
\cite{Ber92} paved the way to limit theorems for products of free random variables.   Subsequently, in \cite{ber2000}, Bercovici and Pata established \emph{necessary and sufficient conditions} for the weak convergence of products of identically distributed positive random variables that are not necessarily compactly supported. These results were generalized in the papers \cite{ber08,chistyakov08} to more general triangular arrays of positive random variables that are not necessarily identically distributed. The behavior of the support of measures under free multiplicative convolution is an important question that has also been studied in \cite{huang14,kargin2007norm,kargin2008asymptotic}. 

The explicit description of the limiting distribution, known as the free multiplicative semicircular law, for products of positive freely independent and identically distributed random variables, was provided by \cite{Ho2011} who proved that the logarithm of the limit law is the free additive convolution of the semicircular law with one-half the uniform distribution on $[-1,1]$%, see Section \ref{Subsection:positive-semicircular} for more details
. This limiting distribution also appears in the context of multiplicative Brownian motions on the general linear group, see \cite{AuerVoigt}. Zhong later studied the support of this limiting distribution in \cite{zhong14, zhong15}, while its free cumulants were explicitly computed by Arizmendi, Fujie, and Ueda in \cite{Arizmendi24}.

This paper aims to provide the first quantitative estimates for the free multiplicative CLT in the general setting described in \eqref{eqn:product}, where the free variables may follow different distributions. The estimates are expressed in terms of the $r$-Wasserstein distance for any $r\geq 1$ which ensures not only convergence in distribution but also convergence of moments up to order $r$. Moreover, our bounds depend solely on the moments of the underlying variables. It is also worth mentioning that we provide a combinatorial proof of the free multiplicative central limit theorem for the general product in \eqref{eqn:product} in Appendix~\ref{Appendix:combinatorialProof}. In contrast to Ho’s approach, our method, when combined with the results of \cite{ber08, chistyakov08}, extends to proving convergence in the
$r$-Wasserstein distance in the unbounded setting, see Remark~\ref{remark:unbddCLT}.

The main results are presented in Section~\ref{Subsection:main results}, followed by a discussion of the key ideas behind the proofs in Section \ref{Section: outline of key results}. Two preliminary sections, Sections~\ref{Subsect: distance on probability} and \ref{Subsection:positive-semicircular}, are also included, introducing the necessary background and concepts used throughout the paper. The proofs of the main results are presented in Section~\ref{sec:invarianceppl}, followed by an appendix containing additional material essential to the proof. 

\subsection{Main results}\label{Subsection:main results}
For any given $0<\gamma \leq 1$ and  $I\subseteq \mathbb{R}$, we denote by $ \C^{\gamma}(I)$ the class of complex-valued functions that are $\gamma$-H\"older continuous on $I$, that is,
\begin{align*}
    \C^{\gamma}(I) = \Big\{ h:I\rightarrow \mathbb{C} : \sup_{x\neq y \in I}\frac{|h(x)-h(y)|}{|x-y|^{\gamma}} <+\infty\Big\},
\end{align*}
and  denote by $\|h\|_{\gamma,I}$ the H\"older semi-norm on the interval $I$ with exponent $\gamma$, that is, the largest constant $C >0$ such that 
$$
| h(x)-h(y)| \leq C |x-y|^{\gamma},\quad x,y \in I.
$$
Finally, we denote by $\C^{k,\gamma}(I)$ the class of complex-valued functions $h$ on $I$ that are $k$-times continuously differentiable, with the $k$-th derivative $h^{(k)}$ belonging to $\C^{\gamma}(I)$.  

\begin{theorem}
\label{thm:mainthm}
Let $x=(x_i)_{i\geq 1}$ be a sequence of free centered random variables in $(\mathcal{A},\varphi)$ with common variance $\sigma^{2}$. Let $g\colon \mathbb{R}\to \mathbb{C}$ be a  function satisfying  the following regularity and integrability conditions:\footnote{The notation $p+o$ refers to moments of order slightly greater than $p$, i.e. of the form $p+\varepsilon$ for some small $\varepsilon>0$.}
\begin{enumerate}[\indent 1.]
\item \label{assumptionone}$g \in \C^{2,\gamma}(\mathbb{R})$, for some $0< \gamma \leq 1$   with $g(0)=1$ and $g'(0)=g''(0)\in \mathbb{R}\setminus \{ 0\}$,
\item \label{assumptiontwo}there exists some $0<\alpha\leq 1$ such that $$\sup_i \| \|g\|_{1,[\alpha x_i^-, \alpha x_i^+ ]} \|_{L^{4+o}_\varphi}< \infty, \quad \text{and} \quad \sup_i \| \|g''\|_{\gamma,[\alpha x_i^-, \alpha x_i^+ ]} \|_{L^{2+o}_\varphi} < \infty,$$ where    $x_i^-:=\min(x_i,0)$, $x_i^+:=\max(x_i,0)$ and  $\|.\|_{L_\varphi^p}$ denotes the $L^p$ norm with respect to $\varphi$,
\item \label{assumptionthree}$\varphi(g(n^{-1/2}x_i)) \geq 1$ for all $ i\geq 1$,
\item  \label{assumptionfour}the sequence of the moments of order $6{+}o$ of $x$ is uniformly bounded.
\end{enumerate}
Let $r\geq 1$. Under the assumptions listed above, the $r$-Wasserstein distance $W_r$ 
between the spectral distribution of {$|\pi_n^{g,n^{-1/2}x}|$} and the square root of the free multiplicative semicircular distribution\footnote{The free multiplicative semicircular distribution is defined in Section \ref{Subsection:positive-semicircular}.} with parameter {$\frac{1}{2}|g^2|'(0)\sigma$} normalized by $ (1+ \| \pi_n^{g,n^{-1/2}x}\|_{L^{8\vee (1+r)}_{\varphi}})^{6/r} $, converges to $0$ as $n$ goes to infinity at rate $n^{-\beta'}\ln(n)^{\beta''}$ where $\beta'$ and $\beta''$ are given in Table~\ref{tb:rates} with constants depending merely on the moments and not  the operator norm.

\renewcommand{\arraystretch}{1.3}

\begin{figure}[htb!]
$$
\begin{array}{c c c c}
\toprule
r & \gamma_{\rm crit} & (\beta'_{-},\beta''_{-}) & (\beta'_{+},\beta''_{+}) \\
\midrule
r=1	         & 4/5				        & 		(\gamma/8,0)       &(1/10,0)\\
1 < r < 2    & (r^2 + r + 2)/(r^2 + r + 4)       & \left( \gamma/(2r^2 + 2r + 4), 0 \right) & \left( 1/(2r^2 + 2r + 8), 0 \right) \\
r = 2          & 2/3                                        & \left( \gamma/8, 1/2 \right) & \left( 1/12, 0 \right) \\
2 < r < 4    & (r^2 + 2r)/(r^2 + r + 4)           & \left( \gamma/(2r^2 + 4r), 0 \right) & \left( 1/(2r^2 + 2r + 8), 0 \right) \\
r = 4          & 1                                           & \left( \gamma/48, 0 \right)     & \left( 1/48, 1/4 \right) \\
r > 4          & 1                                           & \left( \gamma/12r, 0 \right)    & \left( 1/12r, 0 \right) \\
\bottomrule
\end{array}
$$
\caption{\label{tb:rates}Rates of convergences in the $r$-Wasserstein distance. We recall that $\gamma$ is the H\"older exponent of $g''$. The column of $(\beta'_{-},\beta"_{-})$ collects the rates when $\gamma < \gamma_{\rm crit}$ and the column $(\beta'_{+},\beta"_{+})$ collects the rates when $\gamma > \gamma_{\rm crit}$.}
\end{figure}

\end{theorem}
\begin{remark}
    Note that the conditions $g(0)=1$ and $g'(0)=g''(0)\in \mathbb{R}\setminus \{ 0\}$ in Theorem \ref{thm:mainthm} are only imposed for technical reasons to obtain our quantitative estimates. The free multiplicative CLT still holds without them, see Theorem \ref{thm: moment CLT}.
\end{remark}
\begin{corollary} Let $r\geq 1$ and  $x=(x_i)_{i\geq 1}$ be a sequence of free identically distributed random variables with variance $\sigma^2$. The singular values of the product 
\begin{equation}\label{cor:product}
   { \Big(1+\frac{x_1}{\sqrt{n}}+\frac{x^2_1}{2n}\Big)\cdots \Big(1+\frac{x_n}{\sqrt{n}}+\frac{x^2_n}{2n}\Big)}
\end{equation}
converge as $n$ goes to infinity in the $W_r$ distance to the square root of the free multiplicative semicircular distribution with parameter $\frac{1}{2}\sigma$, with a rate of convergence given by the last column in Table \ref{tb:rates}.
\end{corollary}
\begin{proof}
Let $g : \mathbb{R} \to \mathbb{R},\, x \mapsto 1+x+x^2/2$. It is easy to check that $g$, $x$ satisfy all the assumptions of Theorem \ref{thm:mainthm}. In addition, the moments at all orders of \eqref{cor:product} are uniformly bounded in $n$: this is a straightforward application of Theorem \ref{thm: moment CLT}.
\end{proof}
\begin{corollary}[Free Multiplicative CLT -- Wasserstein]
\label{thm: rate of CLT in Wasserstein} 
\label{thm:multfreecltexp}
Let $x = (x_i)_{i\geq 1}$ be a sequence of identically distributed free random variables with constant variance $\sigma^2 > 0$. We assume the existence of $\alpha_0 > 0$ such that 
\begin{equation}
\label{eqn:assumptionexp}
\sup_{n\geq 0} \sup_{|u| < \alpha_0} \varphi(e^{ux_i}) < +\infty.
\end{equation}
As $n$ goes to infinity, the singular values of $e^{n^{-1/2}x_1}\cdots  e^{n^{-1/2}x_n}$ converge to the square root of the free multiplicative semicircular distribution with parameter $\frac{1}{2} \sigma^2$ in the $W_r$ distance at rates given by the Table \ref{tb:rates}.
\end{corollary}
\begin{proof}
Set $g = \exp$. Since $\exp$ is a convex function, Jensen's inequality yields $\varphi(\exp(n^{-1/2}x_i)) \geq 1$. The conclusion follows from the application of Theorem \ref{thm:mainthm} and Theorem \ref{thm: moment CLT} with $k=8$ under the assumption \eqref{eqn:assumptionexp}.
\end{proof}
\begin{corollary}[Rates in the Kolmogorov distance]
With the notations of Theorem \ref{thm:mainthm}, let $1\geq \gamma > 0$, $r\geq 1$, $g\colon \mathbb{R}\to \mathbb{C}$ and $x = (x_i)_{i\geq 1}$ satisfying assumptions \ref{assumptionone} - \ref{assumptionfour}.
Assume in addition that the sequence of moments of order $8$ of $|\pi_n^{g,n^{-1/2}x}|$ is bounded uniformly in $n\geq 1$. 

Then, as $n$ goes to infinity, the Kolmogorov distance between the spectral distribution of $|\pi_n^{g,n^{-1/2}x}|$ and the square root of the free multiplicative semicircular distribution with parameter $\frac{1}{2}|g^2|'(0)\sigma$ converges to $0$ at rate $n^{-\beta'}\ln(n)^{\beta''}$ where $2\beta'$ and $2\beta''$ are tabulated in the first line of Table \ref{tb:rates}.
\end{corollary}
\begin{proof}
Recall that if $\nu$ is a probability measure with a Lebesgue density bounded by $C$, then the Kolmogorov distance can be bounded by the 1-Wasserstein distance as follows:
\begin{equation*}
\mathcal{K}(\mu,\nu) \leq \sqrt{2C W_1(\mu,\nu)}. 
\end{equation*}
Since the square root of the free multiplicative semicircular distribution is a density relative to the Lebesgue measure on $\mathbb{R}_+$. We use the inequality above to obtain the result.
\end{proof}

\tableofcontents

%%%%%%%%%%%%%%%%%%%%%%%%%%%%%%%
% Section: Idea of the proof and outline 
%%%%%%%%%%%%%%%%%%%%%%%%%%%%%%%

\subsection{Idea of the proof and outline of key results}
\label{Section: outline of key results}
We begin by outlining the key steps leading to the desired estimates. For the case of the $1$-Wasserstein distance, we make use of the Kantorovich-Rubinstein duality formula. Since no analogous duality holds for $r > 1$, we instead rely on Rio's inequality~\eqref{inequ:Rio}, which bounds the $r$-Wasserstein distance in terms of the $r$-th root of the $r$-Zolotarev distance. The definitions of these distances and the relationships between them are provided below in Section~\ref{Subsect: distance on probability}. 

While obtaining bounds on Wasserstein distances is generally challenging, working with the Zolotarev distance offers a useful alternative. By maximizing the relevant quantities over the class of Zolotarev functions, we gain access to a variety of analytic tools that enable us to derive the desired quantitative estimates. For a function $f$ belonging to the Zolotarev class, we provide a quantitative bound for
\[ \left|\varphi\big(f(|\pi_n^{g,n^{-1/2}x}|)\big) - 
    \varphi\big(f(e^{y})\big)\right|,
\] 
where $y$ is the $\log$ of a free multiplicative semicircular element. The above bound follows from the infinite divisibility of the free multiplicative semicircular distribution together with the following invariance principle:
\begin{align}\label{term:invarianceprinciple}
    \left|\varphi\big(f(|\pi_n^{g,n^{-1/2}x}|)\big) - 
    \varphi\big(f(|\pi_n^{h,n^{-1/2}y}|)\big)\right|,
\end{align}
where the function $h$ and the sequence $y = (y_i)_{i \leq n}$ satisfy the same assumptions in Theorem \ref{thm:mainthm}. 

Our approach, which we refer to as the Fourier–Lindeberg method, involves regularizing Zolotarev functions to make them amenable to a Lindeberg-type replacement scheme. However, the inherent nonlinearity of the multiplicative setting presents significant challenges. To address this, we rely on the Hermitization trick, which helps linearize the problem to some extent. The price to pay is to prove an invariance principle within the framework of the noncommutative probability space $(M_2(\mathcal{A}), \tr_2\varphi)$ of $2\times 2$ matrices over $\mathcal{A}$. This is carried out in Section~\ref{sec:invarianceppl}. The final step is to estimate the divergence of the moments of the Fourier transform of kernel regularization of Zolotarev functions as illustrated in Lemmas \ref{lem:unboundedlipschitz} and \ref{lemma:UnbddZo}.

The Lindeberg method has proven to be a powerful technique for obtaining quantitative bounds in limit theorems. It has already been successfully applied to establish the only known quantitative estimates for the operator-valued free, Boolean, and monotone additive central limit theorems in terms of the L\'evy distance, as shown in \cite{Ariz-Banna-Tseng,BannaMaiBerry}. However, the arguments required in the multiplicative setting cannot be reduced to adaptations of those used for the additive case.

The remaining part of this section aims to review essential concepts on distances between probability measures, and aspects of free probability theory that are relevant to this paper.

\subsection{Distances on probability distributions}\label{Subsect: distance on probability}
On the set of all Borel probability measures on the real line $\mathbb{R}$, there are several well-established and useful notions of distance. In this paper, we consider the Kolmogorov,  Wasserstein, and Zolotarev distances, and discuss some of the relationships between them. For a Borel probability measure $\mu$ on $\mathbb{R}$, we denote by $\mathcal{F}_\mu$ its \emph{cumulative distribution function}, defined by
\(
  \mathcal{F}_\mu(t) := \mu((-\infty,t])
\), for any $t \in \R$. If  $\mu$ and $\nu$ are two Borel probability measures on $\mathbb{R}$, the \emph{Kolmogorov distance} by
\begin{align}
\label{eqn:kolmogorovdistance}
\tag{$K$}
\mathcal{K}(\mu,\nu) := \sup_{t\in \mathbb{R}} \big|F_{\mu}(t)-F_{\nu}(t)\big|.
\end{align}
For any $r\geq 1$ and any $\mu,\nu\in P_r(I)$, the set of probability measures on the interval $I\subseteq \mathbb{R}$ with moments up to the order $r$, the \emph{Wasserstein distance} of order $r$ or the $r$-\emph{Wasserstein distance} is defined by 
\begin{equation}
\label{eqn:rWassersteindistance}
\tag{$W_r$}
 W_r(\mu,\nu)=\inf \{\|X-Y\|_{L^r} \mid X\overset{d}{\sim} \mu, \text{ and }Y\overset{d}{\sim} \nu\}.   
\end{equation}
It is known that convergence in the $r$-Wasserstein distance is equivalent to convergence in distribution along with the convergence of moments up to order $r$ \cite[Theorem 6.9]{villani08}.  If we  assume that  $\nu$ has a Lebesgue density bounded by $C$, then by  \cite[Proposition 1.2]{ross11}, the Kolmogorov distance can be bounded by the Wasserstein distance as follows:
\begin{equation}\label{eqn: relation btw k-w distance} 
\mathcal{K}(\mu,\nu) \leq \sqrt{2C W_1(\mu,\nu)}. 
\end{equation} Computing the $r$-Wasserstein distance is generally not an easy task; however, in the particular case $r = 1$, and thanks to the Kantorovich-Rubinstein duality formula, the $1$-Wasserstein distance can be reformulated as the solution to a maximization problem over $1$-Lipschitz functions. Specifically, it can be written as
\begin{align}
\label{eqn:wasserstein}
\tag{$W_1$}
W_1(\mu,\nu) 
&= \sup \Big\{\int_{\mathbb{R}} f(x)\nu(\mathrm{d}x) - \int_{\mathbb{R}}f(y)\mu(\mathrm{d}y)\colon f\colon \mathbb{R}\to \mathbb{R}, \|f\|_{1,\mathbb{R}} \leq 1\Big\}. 
\end{align}
It is important to emphasize that this duality does not hold when $r > 1$. Nevertheless, the $r$-Wasserstein distances can still be bounded above in terms of the so-called Zolotarev distances, using Rio's inequality ~\eqref{inequ:Rio} below.

 In \cite{Zol-78,Zol-81}, Zolotarev introduced a metric for analyzing the class of probability measures that possess finite moments of order $r$ and have matching moments up to and including order $r-1$.  Here, we adopt an alternative definition of the Zolotarev metric, which extends its applicability to probability measures with moments of order $r$ that do not necessarily coincide. 

Let $r > 1$ and write $ r = \ell + \beta $ with $\beta \in (0,1]$ and $\ell$ an integer (in particular, if $r$ is an integer, $ \ell = r-1$ and $\beta = 1$). Let $I \subset \mathbb{R}$ be an interval containing $0$. We denote by $\Lambda^{0}_r(I)$ the set of all $\ell$ times differentiable functions $f\colon I \to \mathbb{R}$ satisfying 
$$
f^{(i)}(0)=0,~ 0 \leq i \leq \ell, \quad \text{and} \quad \sup_{x\neq y}\frac{|f^{(\ell)}(x)-f^{(\ell)}(y)|}{|x-y|^{\beta}} < +\infty.
$$ 
%If $I \subset \mathbb{R}$ is an interval of $\mathbb{R}$, we call $\|f^{(\ell)}\|_{\beta,I}$ the best constant $C > 0$ such that 
% $$
% \frac{|f^{(\ell)}(x)-f^{(\ell)}(y)|}{|x-y|^{\beta}} \leq C,\quad x \neq y \in I.
% $$
For $\mu,\nu \in \mathcal{P}_r(I)$, the \emph{Zolotarev distance} is defined by 
\begin{equation}
\label{def Zolotarev distance}
\tag{$Z_r$}
Z_r(\mu,\nu) : = \sup \{ |\mu(f) - \nu(f)|,\, f \in \Lambda^{0}_r(I), \|f^{(\ell)}\|_{\beta,I} \leq 1\},
\end{equation}
where we recall that $\|f^{(\ell)}\|_{\gamma,I}$ denotes the H\"older semi-norm on the interval $I$ with exponent $\gamma$. 
We extend the definition of $\Lambda^{0}_r(I)$ by setting $\Lambda^{0}_1(I)$ as the class of Lipschitz functions $f\colon\mathbb{R}\to\mathbb{R}$ with $f(0)=0$. For more details on ideal metrics and Zolotarev metrics, we refer to \cite{bogachev2017weighted, Bogachev-book}. 

Having recalled the Zolotarev distance, we illustrate Rio's inequality \cite{Rio-98} that provides an upper bound for the Wasserstein distances of order $r$ in terms of the Zolotarev metric. Whenever $r>1$, Rio proved that there exists a positive constant $c_r$ such that, for any pair $(\mu,\nu)\in \mathcal{P}_r(I)$, 
\begin{equation}
	\label{inequ:Rio}
W_r(\mu,\nu)^r \leq c_r Z_r(\mu,\nu).
\end{equation}

\subsection{Free multiplicative semicircular distribution}
\label{Subsection:positive-semicircular}
For an introduction to free probability theory, including noncommutative probability spaces, free independence, free cumulants, moment-cumulant formula, operator-valued free probability, and other fundamental results and tools, we refer the reader to the books \cite{mingospeicherbook,NicaSpeicher}. This section is devoted to the free multiplicative semicircular distribution, denoted by $\mu_{p.s.c}$. Bercovici and Voiculescu showed that $\mu_{\text{p.s.c}}$ is infinitely divisible with respect to free multiplicative convolution and has compact support. In particular, $\mu_{\text{p.s.c}}$ is absolutely continuous with respect to the Lebesgue measure. Moreover, by combining this result with \cite[Theorem 5.7]{zhong14}, it follows that the support of $\mu_{\text{p.s.c}}$ is a closed and bounded interval in $\mathbb{R}$. Additionally, as noted in \cite[Section 7]{Ber92} the $S$-transform of $\mu_{\text{p.s.c}}$ is given by
$$
S_{\mu_{p.s.c}}(z)=\exp\Big(-4\sigma^2\big(z+\frac{1}{2}\big)\Big).
$$
Ho \cite{Ho2011} explicitly determined the limiting distribution in the free multiplicative central limit theorem and showed that the free multiplicative semicircular law $\mu_{\text{p.s.c}}$ has a logarithm $y_{\text{l.s.c}}(\sigma)$, which is given by the free additive convolution of a semicircular distribution and one-half the uniform distribution on $[-1,1]$. More precisely,
\[
y_{\text{l.s.c}}(\sigma) = 2\sigma s + 2\sigma^2 w,
\]
where $s$ is a standard semicircular element, $w$ is uniformly distributed on $(-1,1)$, and $s$ and $w$ are freely independent.
Moreover, Ho gave the explicit expressions of the moments, see \cite[Eq.(41)]{Ho2011}. Furthermore, the free cumulants $\kappa_n^{\mu_{p.s.c}}$ were explicitly computed in \cite[Proposition A.1]{Arizmendi24}
\begin{equation}\label{eqn: free cumulant of p.s.c}
\kappa_k^{\mu_{p.s.c}}=\frac{k^{k-1}}{k!}(2\sigma)^{2(k-1)}e^{2k\sigma^2}\quad \text{ for each }\quad k\geq 1.    
\end{equation}
 See \cite{Arizmendi24,Ber92,Ho2011} for more details about the material presented in this section.

%%%%%%%%%%%%%%%%%%%%%%%%%%%%%%%%%
% convergence in the Wasserstein distance
%%%%%%%%%%%%%%%%%%%%%%%%%%%%%%%%%
\section{Asymptotic convergence rates}

\label{subsect: main part of proof C3-theorem}
In this section, we first present a key intermediate result, followed by the proof of our main result, Theorem \ref{thm:mainthm}.

\subsection{Invariance principle}\label{sec:invarianceppl}
 As outlined in Section \ref{Section: outline of key results}, a key  element of the proof is proving the following invariance principle: for a sufficiently smooth function $f\colon\mathbb{R}_+\to \mathbb{R}$ with bounded derivatives
\begin{align}\label{inv-prin-unnormalized}
    \left|\varphi\big(f(|\pi_n^{g,x}|)\big) - 
    \varphi\big(f(|\pi_n^{h,y}|)\big)\right|,
\end{align}
where the function $h$ and the sequence $y = (y_i)_{i \leq n}$ satisfy the same assumptions in Theorem \ref{thm:mainthm}. Note that the variables in \eqref{inv-prin-unnormalized} are not yet normalized. The bounds derived here serve as an intermediate result; normalization will be addressed in the proof of Theorem \ref{thm:mainthm} in Section \ref{sec:proofmainthm}.

 We extend $f\colon\mathbb{R}_+\to \mathbb{R}$ to a symmetric function on $\mathbb{R}$, denoting this extension by $\tilde{f}$. Since $\tilde{f}$ is an even function, it admits the representation
\begin{equation}\label{Fourier integral}
\tilde{f}(x) = \int_{u\in\mathbb{R}}\cos(u x)\mathcal{F}{\tilde{f}}(\rmd u),    
\end{equation}
where $\mathcal{F}\tilde{f}$ is the Fourier transform of $\tilde{f}$ : 
\begin{equation*}
\mathcal{F}\tilde{f}(x) = \int_{\mathbb{R}} \tilde{f}(x)e^{{i} u x}{\rmd u} = 2\int_{\mathbb{R}_+}{f}(x)\cos(u x){\rmd u} ,\quad u \in \mathbb{R}.
\end{equation*}
Let $(x_i)_{1 \leq  i \leq n}, (y_{i})_{1 \leq  i \leq n} \in \mathcal{A}^{\mathbb{N}}$ be two sequences of  self-adjoint elements  and $g,h\colon \mathbb{R}\to \mathbb{C}$ two smooth functions.  We first observe that, by Lemma \ref{Prop Fourier cosine}, that
\begin{align*}
\varphi\big(f(|\pi_n^{g,x}|)\big) - 
    \varphi\big(f(|\pi_n^{h,y}|)\big) &= \int_{u\in \mathbb{R}} \big(\varphi \big(\cos(u|\pi_n^{g,x}|)\big)\big)  -\varphi \big(\cos(u|\pi_n^{h,y}|)\big) \mathcal{F}{\tilde{f}}(\rmd u) 
    \\&= \int_{u\in \mathbb{R}} \big(\trphi \big(e^{{\sf i}u H({\pingx})}\big) - \trphi \big(e^{{\sf i}u H({\pi^{h,y}_n})}\big)\big) \mathcal{F}{\tilde{f}}(\rmd u) , 
\end{align*}
where $H({\pingx})$ and $H({\pi^{h,y}_n})$ on the right-hand side are the Hermitization matrices $$
H({\pingx})=\begin{bmatrix}
    0 & {\pingx} \\ ({\pingx})^{\star} & 0
\end{bmatrix} \quad \text{and} \quad H({\pi^{h,y}_n})=\begin{bmatrix}
    0 & \pi^{h,y}_n \\ (\pi^{h,y}_n)^{\star} & 0
\end{bmatrix}. 
$$ 
We proceed by applying the Lindeberg method in $(M_2(\cA), \tr_2\varphi)$ to estimate the term 
\[\big|\trphi \big(e^{{\sf i}u H({\pingx})}\big) - \trphi \big(e^{{\sf i}u H({\pi^{h,y}_n})}\big)\big|.
\]
To this end, we introduce the following notation. For any $k\leq \ell \in \{1, \ldots,n\}$, we  define the partial products
\begin{equation*}
    \pi_{[k,\ell]}^{g,x} = X_{k} \dots X_{\ell}
 \quad    \text{and} \quad \pi_{[k,\ell]}^{h,y} = Y_{k} \dots Y_{\ell}, 
\end{equation*}
where $X_j=g(x_j)$ and $Y_j=h(y_j)$ for each $k\leq j \leq l$.
For any integer $1\leq i\leq n$, we also define the interpolating products
$$
z_i =X_{1}\cdots X_{i-1} X_{i}Y_{i+1}\cdots Y_{n} \quad \text{and} \quad z_i^1=X_{1}\cdots X_{i-1}1Y_{i+1}\cdots Y_{n}.
$$
Finally, we define the following two matrices in $\mathcal{M}_{2}(\mathcal{A})$
$$
Z_i=\begin{bmatrix}
    0 & z_i \\ z_i^{\star} & 0
\end{bmatrix} \quad \text{and} \quad Z_i^1=\begin{bmatrix}
    0 & z_i^1 \\ (z_i^1)^{\star} & 0
\end{bmatrix},
$$
to which we shall drop the index $n$ to soothe the notation. The Lindeberg method expresses the above difference as a telescoping sum,  replacing each  $X_{i}$ with the corresponding $Y_{i}$, one at a time. Specifically, we write:
$$
\trphi \big(e^{{\sf i}u H({\pingx})}\big) - \trphi \big(e^{{\sf i
}u H({\pi^{h,y}_n})}\big)=\sum_{i=1}^n \big(\trphi(e^{{\sf i}u Z_i}-e^{{\sf i}u Z_i^1})-\trphi(e^{{\sf i}u Z_{i-1}}-e^{{\sf i}u Z_i^1})\big).
$$
To control $\trphi(e^{{\sf i}u Z_i}-e^{{\sf i}u Z_i^1})$, we apply Lemma \ref{lemma exponential identity} up to the third order and obtain
\begin{multline*}
 \trphi(e^{{\sf i}u Z_i}-e^{{\sf i}u Z_i^1})= 
     {\sf i}u\int_{\mathcal{T}_1} \trphi\big(e^{{\sf i}\alpha_0 u Z_i^1} (Z_i - Z^{1}_i)e^{{\sf i}\alpha_1 u Z_i^1}\big) \rmd \alpha \\-u^2\int_{\mathcal{T}_2} \trphi\big(e^{{\sf i}\alpha_0 u Z_i^1}(Z_i - Z^{1}_i)e^{{\sf i}\alpha_1 u Z_i^1}(Z_i - Z^{1}_i)e^{\alpha_2{\sf i} u Z_i^1}\big) \rmd \alpha 
     \\  -{\sf i}u^3 \int_{\mathcal{T}_{3}} \trphi\big(e^{{\sf i}\alpha_0 u Z_i}(Z_i - Z^{1}_i)e^{{\sf i}\alpha_1 u Z_i^1}(Z_i - Z^{1}_i)e^{{\sf i}\alpha_2 u Z_i^1}(Z_i - Z^{1}_i)e^{{\sf i}\alpha_{3} u Z_i^1}\big) \rmd \alpha,
\end{multline*}
where we recall that $\rmd \alpha = \rmd \alpha_0 \cdots \rmd \alpha_k$ and
$$
\mathcal{T}_k= \Big\{\alpha=(\alpha_0, \dots , \alpha_k) \ | \sum_{i=0}^k \alpha_i =1 \; \text{and} \; 0\leq \alpha_i \leq 1 ; \text{ for } i =0,\dots, k \Big\}. 
$$ 
As the state $\trphi$ is tracial, we have 
\begin{align*}
 &\trphi(e^{{\sf i}u Z_i}-e^{{\sf i}u Z_i^1})\\& = 
      {\sf i}u \ \trphi\big(e^{{\sf i}u Z_i^1} (Z_i - Z^{1}_i)\big) -u^2\int_{\mathcal{T}_1} \trphi\big(e^{{\sf i}\alpha_0 u Z_i^1}(Z_i - Z^{1}_i)e^{{\sf i}\alpha_1 u Z_i^1}(Z_i - Z^{1}_i)\big) \rmd \alpha 
     \\ & \quad- {\sf i}u^3 \int_{\mathcal{T}_{3}} \trphi\big(e^{{\sf i}\alpha_0 u Z_i}(Z_i - Z^{1}_i)e^{{\sf i}\alpha_1 u Z_i^1}(Z_i - Z^{1}_i)e^{{\sf i}\alpha_2 u Z_i^1}(Z_i - Z^{1}_i)e^{{\sf i}\alpha_{3} u Z_i^1}\big) \rmd \alpha.
\end{align*}
Noting that  $Z_i - Z^{1}_i = U_i \begin{bmatrix}0 &X_{i}-1 \\ X_{i}^{\star}-1 & 0\end{bmatrix} U^{\star}_i := U_i (\bm{X}_{i}-I_2) J U^{\star}_i$, where 
\begin{align*}
  \bm{X}_{i} = \begin{bmatrix} X_{i} & 0 \\ 0 & X_{i}^{\star} \end{bmatrix}, \quad U_i = \begin{bmatrix}
               X_{1}\cdots X_{i-1} & 0 \\
               0 & Y_{n}^{\star}\cdots Y_{i+1}^{\star}
              \end{bmatrix}, \quad \text{and}\quad J = \begin{bmatrix} 0 & 1 \\ 1 & 0\end{bmatrix},
\end{align*}
then we write next:
\begin{align*}
& \trphi (e^{{\sf i}u Z_i}-e^{{\sf i}u Z_i^1}) \\
      &={\sf i}u \ \trphi\big(e^{{\sf i}u Z_i^1} U_i (\bm{X}_{i}-I_2) J U^{\star}_i \big) \\ & -u^2\int_{\mathcal{T}_1} \trphi\big(e^{{\sf i}\alpha_0u Z_i^1} U_i (\bm{X}_{i}-I_2) J U^{\star}_i e^{{\sf i}\alpha_1uZ_i^1} U_i (\bm{X}_{i}-I_2) J U^{\star}_i\big) \rmd \alpha 
     \\ & - {\sf i}u^3 \int_{\mathcal{T}_{3}} \trphi\big(e^{{\sf i}\alpha_0uZ_i}U_i (\bm{X}_{i}-I_2) J U^{\star}_i e^{{\sf i}\alpha_1uZ_i^1}U_i (\bm{X}_{i}-I_2) J U^{\star}_i e^{{\sf i}\alpha_2uZ_i^1}U_i (\bm{X}_{i}-I_2) J U^{\star}_i e^{{\sf i}\alpha_{3} uZ_i^1}\big) \rmd \alpha .
    \end{align*}
By similarly treating $\trphi(e^{{\sf i}u Z_{i-1}}-e^{{\sf i}u Z_i^1})$ and recalling that $Z_{i-1} - Z^{1}_i := U_i (\bm{Y}_{i}-I_2) J U^{\star}_i$ where $\bm{Y}_{i}$  is defined analogously to $\bm{X}_{i}$, we obtain 
\begin{multline*}
    \trphi \big(e^{{\sf i}u H({\pingx})}\big) - \trphi \big(e^{{\sf i}uH({\pi^{h,y}_n})}\big) \\= 
    \sum_{i=1}^n \big( {\sf i}u \ \trphi  ( L_{i,1}) -u^2 \int_{\mathcal{T}_{1}}\ \trphi  ( L_{i,2} )\rmd \alpha- {\sf i}u^3 \int_{\mathcal{T}_{3}}\ \trphi( L_{i,3})\rmd \alpha\big),
\end{multline*}
where for any $i=1,\dots,n$,
\begin{align}\label{Li terms}
&L_{i,1} := e^{{\sf i}uZ_i^1}U_i (\bm{X}_{i}-I_2) J U^{\star}_i -e^{{\sf i}uZ_i^1}U_i (\bm{Y}_{i}-I_2) J U^{\star}_i, \notag \\
&L_{i,2}:= e^{{\sf i}u \alpha_0 Z^1_i}U_i (\bm{X}_{i}-I_2) J U^{\star}_ie^{{\sf i}u\alpha_1Z^1_i}U_i (\bm{X}_{i}-I_2) J U^{\star}_i - e^{{\sf i}u\alpha_0Z^1_i}U_i (\bm{Y}_{i}-I_2) J U^{\star}_ie^{{\sf i}u\alpha_1Z^1_i}U_i (\bm{Y}_{i}-I_2) J U^{\star}_i, \notag\\
& L_{i,3}:=  e^{{\sf i}u \alpha_0 Z_i}\prod_{k=1}^3 \big(U_i (\bm{X}_{i}-I_2) J U^{\star}_ie^{{\sf i}u \alpha_k Z^1_i}\big)  - e^{{\sf i}u \alpha_0 Z_{i-1}}\prod_{k=1}^3 \big(U_i (\bm{Y}_{i}-I_2) J U^{\star}_i e^{{\sf i}u \alpha_k Z^1_i}\big).\notag\\~
\end{align}
We shall denote, throughout, for any  $x\in \mathcal{A}_{sa}^{\mathbb{N}}$ and continuous function $g\colon\mathbb{R}\to \mathbb{C}$, the quantity
\begin{align}
\label{eqn:mngx}
M_{g,x}^n = \max_{1\leq i \leq n} \{ \max\{\varphi(g(x_1)\cdots g(x_i))^{-1},\varphi(g(x_1)\cdots g(x_i))^{-1}\},~1\leq i \leq n\}.
\end{align}
\subsubsection*{\bf First Order Term}
We assume that $g,h$ and $x,y$ satisfy the hypothesis of Theorem \ref{thm:mainthm}. 
\begin{proposition}
\label{prop:firstorderterm}
With the notations introduced at the beginning of Section \ref{sec:invarianceppl} and \eqref{eqn:mngx}, for any $1 \leq i \leq n$, for any $\zeta > 0$ 
\begin{align*}
|\trphi(L_{i,1}) |  &\lesssim   \big(M_{g,x}^n\|\pi_{[1,n]}^{g,x}\|_{L^2_\varphi}+ M_{h,y}^n
 \|\pi_{[1,n]}^{h,y}\|_{L^2_\varphi}\big)^2 \\&
 \times \Big(\big(\big\| \|  g^{\prime\prime} \|_{\gamma,[x^-,x^+]}\big\|_{L^{1+\zeta^{-1}}_\varphi} \vee \big\| \|  h^{\prime\prime} \|_{\gamma,[y^-,y^+]} \big\|_{L^{1+\zeta^{-1}}_\varphi} \big)\big(
  \|x\|_{L^{(2+\gamma)(1+\zeta)}_\varphi}^{{2+\gamma}}\vee \|y\|_{L^{(2+\gamma)(1+\zeta)}_\varphi}^{{2+\gamma}}\big)
 \\& \quad+ \big|g^{\prime\prime}(0) \| x_i\|_{L^2_\varphi}^2 - h^{\prime\prime}(0)\|y_i\|_{L^2_\varphi}^2\big|\Big).
\end{align*}
\end{proposition}

\begin{proof}[Proof of Proposition \ref{prop:firstorderterm}]	

By using $\star$-freeness of  $U_i,\bm{X}_{i}, \bm{Y}_{i}$ with respect to $\varphi^{(2)}$, and the fact that  $\trphi$ is tracial, we obtain 
\begin{align*}
\big|\trphi(L_{i,1}) \big|
&=\Big|\trphi\Big(e^{{\sf i}u Z^1_i} \big(\varphi^{(2)} (\bm{X}_{i}-I_2)
-\varphi^{(2)} (\bm{Y}_{i}-I_2)\big)U_iJU_i^{\star}\Big)\Big|
\\ & \leq \Big|\tr_2 \Big( \varphi^{(2)} (U_iJU_i^{\star}e^{{\sf i}u Z^1_i}) \big(\varphi^{(2)} (\bm{X}_{i}-I_2)
-\varphi^{(2)} (\bm{Y}_{i}-I_2)\big)\Big)\Big|
\\&\leq \big\| \varphi^{(2)}(|U_iJU_i^{\star}|)\big\|_{\mathcal{M}_{2}(\mathbb{C})} \big\|\varphi^{(2)} (\bm{X}_{i}-I_2)
-\varphi^{(2)} (\bm{Y}_{i}-I_2) \big\|_{\mathcal{M}_{2}(\mathbb{C})}.
\end{align*} 
Computing the coefficients of the matrix $U_iJU_i^{\star}$ and applying Cauchy-Schwartz inequality yield: 
\begin{align*}
\big\| \varphi^{(2)}(|U_iJU_i^{\star}|)\big\|_{\mathcal{M}_{2}(\mathbb{C})} &= \varphi\big(z_i^1(z_i^1)^{\star}\big)^{\frac{1}{2}} \leq \|\pi^{g,x}_{[1,\,i-1]} \|_{L^2_\varphi}\|\pi^{h,y}_{[i+1,\,n]} \|_{L^2_\varphi}. 
\end{align*}
Then using the inequality in Lemma \ref{thm:NC L^p-inequality}, we bound the above term as follows 
\begin{align*}
\|\pi^{g,x}_{[1,\,i-1]} \|_{L^2_\varphi}\|\pi^{h,y}_{[i+1,\,n]} \|_{L^2_\varphi} &\leq |\varphi(X_{i})\cdots\varphi(X_{n})|^{-1} \|\pi_{[1,n]}^{g,x}\|_{L^2_\varphi} |\varphi(Y_1)\cdots\varphi(Y_{i})|^{-1}\|\pi_{[1,n]}^{h,y}\|_{L^2_\varphi}
\\& \leq  M_{g,x}^nM_{h,y}^n
 \|\pi_{[1,n]}^{g,x}\|_{L^2_\varphi}\|\pi_{[1,n]}^{h,y}\|_{L^2_\varphi}.
 \end{align*}
By applying Lemma \ref{Lemma:moment-est} and putting the above bounds together, we infer 
\begin{align*}
|\trphi(L_{i,1}) |  &\lesssim    M_{g,x}^nM_{h,y}^n
\|\pi_{[1,n]}^{g,x}\|_{L^2_\varphi}\|\pi_{[1,n]}^{h,y}\|_{L^2_\varphi} 
 \\& \times \Big(\big(\big\| \|  g^{\prime\prime} \|_{\gamma,[x^-,x^+]}\big\|_{L^{1+\zeta^{-1}}_\varphi} \vee \big\| \|  h^{\prime\prime} \|_{\gamma,[y^-,y^+]} \big\|_{L^{1+\zeta^{-1}}_\varphi} \big)\big(
  \|x_i\|_{L^{(2+\gamma)(1+\zeta)}_\varphi}^{{2+\gamma}}\vee \|y_i\|_{L^{(2+\gamma)(1+\zeta)}_\varphi}^{{2+\gamma}}\big)
 \\& \quad+ \big|g^{\prime\prime}(0) \| x_i\|_{L^2_\varphi}^2 - h^{\prime\prime}(0)\|y_i\|_{L^2_\varphi}^2\big|\Big).
 \end{align*}
Finally, we conclude our estimate for $\trphi(L_{i,1})$ by observing
$$
 M_{g,x}^n M_{h,y}^n
 \|\pi_{[1,n]}^{g,x}\|_{L^2_\varphi}\|\pi_{[1,n]}^{h,y}\|_{L^2_\varphi} \leq \big(M_{g,x}^n\|\pi_{[1,n]}^{g,x}\|_{L^2_\varphi}+ M_{h,y}^n
 \|\pi_{[1,n]}^{h,y}\|_{L^2_\varphi}\big)^2.
$$
\end{proof}
%%%%%%%%%%%%%%%%%%%%%%%%%%%%%%%%
% Subsection: 2nd order term
%%%%%%%%%%%%%%%%%%%%%%%%%%%%%%%%

\subsubsection*{\bf Second Order Term}

\begin{proposition}\label{prop:secondorderterm}
Using the notations introduced in Section \ref{sec:invarianceppl} and \eqref{eqn:mngx}, we obtain for any $1\leq i\leq n$,
\begin{align*}
\Big|& \trphi \big(L_{i,2}\big) \Big| 
\\& \lesssim   \big(M_{g,x}^n\|\pi_{[1,n]}^{g,x}\|_{L^4_\varphi}+ M_{h,y}^n \|\pi_{[1,n]}^{h,y}\|_{L^4_\varphi}\big)^4 \\
& \hspace{0.5cm}\times  \Big({\big|g^{\prime\prime}(0) \| x_i\|_{L^2_\varphi}^2 - h^{\prime\prime}(0)\|y_i\|_{L^2_\varphi}^2\big| +\big|g^{\prime}(0)^2 \| x_i\|_{L^2_\varphi}^2 - h^{\prime}(0)^2\|y_i\|_{L^2_\varphi}^2\big|}\\ & \hspace{0.5cm}+{\big||g^{\prime}|^2(0) \| x_i\|_{L^2_\varphi}^2 - |h^{\prime}|^2(0)\|y_i\|_{L^2_\varphi}^2\big| }+\|x_i\|_{L^3_{\varphi}}^3 +\|y_i\|_{L^3_{\varphi}}^3 
+ \big(\|x_i\|_{L^{6(1+\zeta)}_{\varphi}}^{3+\gamma}\vee \|y_i\|_{L^{6(1+\zeta)}_{\varphi}}^{3+\gamma}\big)\Big)Q.
\end{align*}
where $Q$ is a polynomial expression of a finite number of positive fractional powers of 
$\|\| g\|_{1,[x_i^-,x_i^+]} \|_{L^{1+\zeta^{-1}}_{\varphi}}$, $\|\| h\|_{1,[y_i^-,y_i^+]} \|_{L^{1+\zeta^{-1}}_{\varphi}}$, 
$\| \|  g^{\prime\prime} \|_{\gamma,[x_i^-,x_i^+]} \|_{L^{2(1+\zeta^{-1})}_\varphi}$, and $\| \|  h^{\prime\prime} \|_{\gamma,[y_i^-,y_i^+]} \|_{L^{2(1+\zeta^{-1})}_\varphi}$.
\end{proposition}

\begin{proof}[Proof of Proposition \ref{prop:secondorderterm}]
We follow the same steps as before and write
\begin{align*}
\varphi^{(2)} \big(L_{i,2}^{x}\big) &:=
\varphi^{(2)} \big(e^{{\sf i}u \alpha_0 Z^1_i}U_i (\bm{X}_{i}-I_2) J U^{\star}_ie^{{\sf i}u\alpha_1Z^1_i}U_i (\bm{X}_{i}-I_2) J U^{\star}_i\big)
\end{align*}
By using the fact that $U_i$ and $\bm{X}_{i}$ are free with respect to $\varphi^{(2)}$, one obtains
\begin{multline*}
\varphi^{(2)} (L_{i,2}^{\bm{x}}) = \kappa^{\varphi^{(2)}}_2\big( ({\bm X}_{i}-I_2) J \varphi^{(2)}\big(U_i^{\star}e^{{\sf i}u \alpha_1 Z^1_i}U_i \big), ({\bm X}_{i}-I_2)  \big) \varphi^{(2)}\big( J U_i^{\star} e^{{\sf i}u \alpha_0 Z^1_i}U_i \big)\\
- {\varphi^{(2)}}\big(  ({\bm X}_{i}-I_2) J \big)\varphi^{(2)}\big(U_i^{\star}e^{{\sf i}u \alpha_1 Z^1_i}U_i  {\varphi^{(2)}}\big( {\bm X}_{i}-I_2\big)  J U_i^{\star} e^{{\sf i}u \alpha_0 Z^1_i}U_i \big) .
\end{multline*}
We use Lemma \ref{Prop Fourier cosine} to write $e^{{\sf i}uZ^1_i}= D_{i, u} + \widetilde{D}_{i,u} J$, where:
\begin{equation*}\label{eqn:D_{i,u}_and_tilde-D_{i,u}}
D_{i,u}=  \begin{bmatrix} \cos(u|z^1_i|) & 0 \\  0 &\cos(u|(z^1_i)^{\star}|) \end{bmatrix}\quad \text{and} \quad 
\widetilde{D}_{i,u}= \begin{bmatrix}  {\sf i}u (z^1_i)^{\star}{\rm sinc}(u|z^1_i|) &0  \\ 0&{\sf i}u z^1_i\ {\rm sinc}(u|(z^1_i)^{\star}|)  \end{bmatrix}.
\end{equation*}
Therefore, by noting that $\varphi^{(2)}\big(U_i^{\star}D_{i,u \alpha_1}U_i \big)$ and $\varphi^{(2)}\big(JU_i^{\star}\widetilde{D}_{i,u \alpha_1}JU_i \big)$ are diagonal matrices, then by Proposition \ref{prop:cumulantsampliation} and the notation introduced therein, we obtain
\begin{align*}
\varphi^{(2)} (L_{i,2}^{\bm{x}})=&\bm{\kappa}^{\varphi,\,{\bm X_{i}-I_2}}_2(1,0) J\varphi^{(2)}\big(U_i^{\star}D_{i,u \alpha_1}U_i \big)\varphi^{(2)}\big( J U_i^{\star} e^{{\sf i}u \alpha_0 Z^1_i}U_i \big)\\& 
+ \bm{\kappa}^{\varphi,\,{\bm X_{i}-I_2}}_2(0,0) \varphi^{(2)}\big(JU_i^{\star}\widetilde{D}_{i,y \alpha_1}JU_i \big)\varphi^{(2)}\big( J U_i^{\star} e^{{\sf i}u \alpha_0 Z^1_i}U_i \big) \\&
- {\varphi^{(2)}}\big(  {\bm X_{i}-I_2} \big) J{\varphi^{(2)}}\big(   {\bm X_{i}-I_2} \big) \varphi^{(2)}\big(U_i^{\star}D_{i,u \alpha_1}U_i  J U_i^{\star} e^{{\sf i}u \alpha_0 Z^1_i}U_i \big)   \\&
-{\varphi^{(2)}}\big(  {\bm X_{i}-I_2} \big){\varphi^{(2)}}\big(  {\bm X_{i}-I_2} \big) \varphi^{(2)}\big(JU_i^{\star}\widetilde{D}_{i,u \alpha_1}JU_i J U_i^{\star} e^{{\sf i}u \alpha_0 Z^1_i}U_i \big).
\end{align*}
We now also write $e^{{\sf i}u \alpha_0 Z^1_i} = D_{i, u\alpha_0} + \widetilde{D}_{i,u\alpha_0} J $ to have expressions with only diagonal matrices  and $J$'s. Keep in mind that the terms where the matrix $J$ appears an odd number of times have zero as diagonal entries, and hence their trace is zero.   We note that $L_{i,2}:=L_{i,2}^{\bm{x}} - L_{i,2}^{\bm{y}}$, with the term $L_{i,2}^{\bm{y}}$ being defined and treated in the same way as $L_{i,2}^{\bm{x}}$ but with  $\bm{Y}_i$  replacing  $\bm{X}_i$'s. Therefore, taking the trace, we obtain
\begin{align*}
 &   \trphi \big(L_{i,2}\big) = \trphi \Big( L_{i,2}^{(1)}+L_{i,2}^{(2)}-L_{i,2}^{(3)}-L_{i,2}^{(4)} \Big) , \end{align*}
where, with the notations introduced in  Appendix $A$:
\begin{align*}
&L_{i,2}^{(1)}= \big( \bm{\kappa}^{\varphi,\bm{X}_{i}-I_2}_2(1,0) - \bm{\kappa}^{\varphi,\bm{Y}_{i}-I_2}_2(1,0)\big) J U_i^{\star}D_{i,u \alpha_1}U_i 
\varphi^{(2)}\big( J U_i^{\star} D_{i,u \alpha_0} U_i \big),
\\& L_{i,2}^{(2)}= \big(\bm{\kappa}^{\varphi,\bm{X}_{i}-I_2}_2(0,0) -\bm{\kappa}^{\varphi,\bm{Y}_{i}-I_2}_2(0,0)\big) JU_i^{\star}\widetilde{D}_{i,u \alpha_1}JU_i  \varphi^{(2)}\big( J U_i^{\star} \widetilde{D}_{i,u \alpha_0}JU_i \big),
\\& L_{i,2}^{(3)}= \Big( {\varphi^{(2)}}\big( \bm{X}_{i}-I_2 \big) {\varphi^{(2)}}\big( \bm{X}^{\star}_{i}-I_2 \big) -  {\varphi^{(2)}}\big( {\bm Y}_{i}-I_2  \big) {\varphi^{(2)}}\big({\bm Y}_{i}^{\star}-I_2\big)\Big)J  U_i^{\star}D_{i,u \alpha_1}U_i  J U_i^{\star} D_{i,u \alpha_0} U_i ,
\\& L_{i,2}^{(4)}= \Big( {\varphi^{(2)}}\big(  \bm{X}_{i}-I_2 \big) {\varphi^{(2)}}\big(\bm{X}_{i}-I_2 \big) -  {\varphi^{(2)}}\big({\bm Y}_{i}-I_2 \big) {\varphi^{(2)}}\big(  {\bm Y}_{i}-I_2 \big)\Big) J U_i^{\star}\widetilde{D}_{i,u \alpha_1}JU_i  J U_i^{\star} \widetilde{D}_{i,u \alpha_0}JU_i.
\end{align*}
We shall now estimate the order of each of the above terms and start with $L_{i,2}^{(1)}$:
\begin{align*}
    \big|&\trphi \big(L_{i,2}^{(1)}\big) \big|\\&\leq \big\| \bm{\kappa}^{\varphi,\, {\bm X}_{i}-I_2}_2(1,0) - \bm{\kappa}^{\varphi, \,{\bm Y}_{i}-I_2}_2(1,0)\big\|_{\mathcal{M}_{2}(\mathbb{C})}\tr_2\big| \varphi^{(2)} (J U_i^{\star}D_{i,u \alpha_1}U_i) 
\varphi^{(2)}\big( J U_i^{\star} D_{i,u \alpha_0} U_i \big)\big|
\\&\leq \big\| \bm{\kappa}^{\varphi,\, {\bm X}_{i}-I_2}_2(1,0) - \bm{\kappa}^{\varphi, \,{\bm Y}_{i}-I_2}_2(1,0)\big\|_{\mathcal{M}_{2}(\mathbb{C})}\tr_2\big| \varphi^{(2)} (U_iU_i^{\star}D_{i,u \alpha_1}) 
J\varphi^{(2)}\big(U_iU_i^{\star} D_{i,u \alpha_0}  \big)\big|.
\end{align*}
Noting that $\|D_{i,u \alpha_0}\|=\|D_{i,u \alpha_1}\|\leq1$,
\begin{align*}
\tr_2\big| &\varphi^{(2)} (U_iU_i^{\star}D_{i,u \alpha_1}) 
J\varphi^{(2)}\big(U_iU_i^{\star} D_{i,u \alpha_0}  \big)\big| \\
&\leq\|\varphi^{(2)} (U_iU_i^{\star}D_{i,u \alpha_1}) \|_{\mathcal{M}_2(\mathbb{C})}
\|\varphi^{(2)}\big(U_iU_i^{\star} D_{i,u \alpha_0})\|_{\mathcal{M}_2(\mathbb{C})} \leq\|\varphi^{(2)} (U_iU_i^{\star}) \|^2_{\mathcal{M}_2(\mathbb{C})}  \\ 
&= \max\Big( \varphi\big(\pi_{[1,i-1]}^{g,x}(\pi_{[1,i-1]}^{g,x})^{\star}\big), \varphi\big(\pi_{[i+1,n]}^{h,y}(\pi_{[i+1,n]}^{h,y})^{\star}\big) \Big)^2 \\
&=\max\big( \|\pi_{[1,i-1]}^{g,x} \|^4_{L^2_\varphi}, \|\pi_{[i+1,n]}^{h,y} \|^4_{L^2_\varphi} \big). 
\end{align*}
By using Lemma \ref{thm:NC L^p-inequality}, we infer 
\begin{align*}
\tr_2\big| \varphi^{(2)} &(U_iU_i^{\star}D_{i,u \alpha_1}) 
J\varphi^{(2)}\big(U_iU_i^{\star} D_{i,u \alpha_0}  \big)\big| 
\\ &\leq  \max\big( |\varphi(X_{i})\cdots \varphi(X_n)|^{-4} \|\pi_{[1,n]}^{g,x} \|^4_{L^2_\varphi} ,
|\varphi(Y_{1})\cdots \varphi(Y_{i})|^{-4} \|\pi_{[1,n]}^{h,y} \|^4_{L^2_\varphi}  \big)
\\ 
& \leq  \max\big( (M_{g,x}^n)^4 \|\pi_{[1,n]}^{g,x} \|^4_{L^2_\varphi} ,
(M_{h,y}^n)^4 \|\pi_{[1,n]}^{h,y} \|^4_{L^2_\varphi}  \big) \\
&\leq  \big( M_{g,x}^n \|\pi_{[1,n]}^{g,x} \|_{L^2_\varphi} +
M_{h,y}^n\|\pi_{[1,n]}^{h,y} \|_{L^2_\varphi}  \big)^4,
\end{align*}
and  conclude the proof with the following inequality:
\begin{align}\label{eqn:secondorderfirst}
\big|\trphi \big(L_{i,2}^{(1)}\big) \big| \leq  \big( M_{g,x}^n \|\pi_{[1,n]}^{g,x} \|_{L^2_\varphi} +
M_{h,y}^n\|\pi_{[1,n]}^{h,y} \|_{L^2_\varphi}  \big)^4 \cdot \big\| \bm{\kappa}^{\varphi,\, {\bm X}_{i}-I_2}_2(1,0) - \bm{\kappa}^{\varphi, \,{\bm Y}_{i}-I_2}_2(1,0)\big\|_{\mathcal{M}_{2}(\mathbb{C})}
 . 
\end{align}
Since $\|\widetilde{D}\| \leq 1$ (since $\widetilde{D}$ appears as a coefficient of a unitary operator), following the same steps, we obtain the same bound for $\trphi \big(L_{i,2}^{(2)}\big)$. As for the remaining terms, we note that 
\begin{align*}
 \big|\trphi \big(L_{i,2}^{(3)}\big) \big|
 &\leq \big\| {\varphi^{(2)}}\big( {\bm X}_{i}-I_2  \big) {\varphi^{(2)}}\big(   {\bm X}_{i}^{\star} -I_2 \big)-  {\varphi^{(2)}}\big({\bm Y}_{i}-I_2   \big) {\varphi^{(2)}}\big(  {\bm Y}_{i}^\star-I_2 \big)\big\|_{\mathcal{M}_2(\mathbb{C})} \\  
 &\hspace{1cm} \times \tr_2 \big|\varphi^{(2)}\big(U_iJ  U_i^{\star}D_{i,u \alpha_1}U_i  J U_i^{\star} D_{i,u \alpha_0}\big)  \big|.
\end{align*}
We observe that
\begin{align*}
&\tr_2 \big|\varphi^{(2)}\big(U_iJ  U_i^{\star}D_{i,u \alpha_1}U_i  J U_i^{\star} D_{i,u \alpha_0}\big)  \big| \\
&\leq \|U_iJ  U_i^{\star}\|_{L_2(\trphi)} \big\| D_{i,u \alpha_1}\| \|U_iJ  U_i^{\star}\|_{L_2(\trphi)} \big\| D_{i,u \alpha_0}\| \leq \|\varphi^{(2)}((U_iJ  U_i^{\star})^2)\| \\
&\leq \varphi\big(\pi^{g,x}_{[1,i-1]}(\pi^{h,y}_{[i+1,n]})^{\star}\pi^{h,y}_{[i+1,n]}(\pi^{g,x}_{[1,i-1]})^{\star}\big) \leq \|\pi^{g,x}_{[1,i-1]}\|^2_{L^4_\varphi} \| \pi^{h,y}_{[i+1,n]}\|^2_{L^4_\varphi},
\end{align*}
where the last inequality holds by applying the Cauchy-Schwartz inequality. Applying the inequality in Lemma \ref{thm:NC L^p-inequality} yields
\begin{align*}
\|\pi^{g,x}_{[1,i-1]}\|^2_{L^4_\varphi} \| \pi^{h,y}_{[i+1,n]}\|^2_{L^4_\varphi} 
&\leq  |\varphi(X_i)\cdots \varphi(X_n)|^{-2}\|\pi_{[1,n]}^{g,x}\|_{L^4_{\varphi}}^2 |\varphi(Y_1)\cdots \varphi(Y_i)|^{-2}\|\pi_{[1,n]}^{h,y}\|_{L^4_{\varphi}}^2  \\
&\leq  (M_{g,x}^n)^2\|\pi_{[1,n]}^{g,x}\|_{L^4_{\varphi}}^2 (M_{h,y}^n)^2\|\pi_{[1,n]}^{h,y}\|_{L^4_{\varphi}}^2  \\ 
&\leq  \big( M_{g,x}^n \|\pi_{[1,n]}^{g,x} \|_{L^4_\varphi} +
 M_{h,y}^n\|\pi_{[1,n]}^{h,y} \|_{L^4_\varphi}  \big)^4.
\end{align*}
Thus, we obtain
\begin{align}\label{eqn:secondordersecond}
\big|\trphi \big(L_{i,2}^{(3)}\big) \big|
&\leq \big\| {\varphi^{(2)}}\big( {\bm X}_{i}-I_2  \big) {\varphi^{(2)}}\big(   {\bm X}_{i}^\star-I_2 \big)-  {\varphi^{(2)}}\big({\bm Y}_{i}-I_2   \big) {\varphi^{(2)}}\big(  {\bm Y}_{i}^\star-I_2 \big)\big\|_{\mathcal{M}_2(\mathbb{C})} \nonumber\\ 
  &\hspace{1.5cm} \times \big( M_{g,x}^n \|\pi_{[1,n]}^{g,x} \|_{L^4_\varphi} +
 M_{h,y}^n\|\pi_{[1,n]}^{h,y} \|_{L^4_\varphi}  \big)^4.
\end{align}
Following the same steps, we bound the term $ \trphi \big(L_{i,2}^{(4)}\big) $. Combining the estimates in \eqref{eqn:secondorderfirst} and \eqref{eqn:secondordersecond}, and the monotonicity of $L^p$-norms, we get 
\begin{align*}
\Big|& \trphi \big(L_{i,2}\big) \Big|  
\\ &\leq \big( M_{g,x}^n \|\pi_{[1,n]}^{g,x} \|_{L^4_\varphi} +
M_{h,y}^n\|\pi_{[1,n]}^{h,y} \|_{L^4_\varphi}  \big)^4
\\ &  \times\bigg(
\big\| {\varphi^{(2)}}\big( {\bm X}_{i}-I_2  \big) {\varphi^{(2)}}\big(   {\bm X}_{i}^\star-I_2 \big)-  {\varphi^{(2)}}\big({\bm Y}_{i}-I_2   \big) {\varphi^{(2)}}\big(  {\bm Y}_{i}^\star-I_2 \big)\big\|_{\mathcal{M}_2(\mathbb{C})} \nonumber\\
 &\hspace{0.5cm}+\big\| {\varphi^{(2)}}\big( {\bm X}_{i}-I_2  \big) {\varphi^{(2)}}\big(   {\bm X}_{i}-I_2 \big)-  {\varphi^{(2)}}\big({\bm Y}_{i}-I_2   \big) 
 {\varphi^{(2)}}\big(  {\bm Y}_{i}-I_2 \big)\big\|_{\mathcal{M}_2(\mathbb{C})} \nonumber\\
&\hspace{0.5cm} + \big\| \bm{\kappa}^{\varphi,\, {\bm X}_{i}-I_2}_2(1,0) - \bm{\kappa}^{\varphi, \,{\bm Y}_{i}-I_2}_2(1,0)\big\|_{\mathcal{M}_{2}(\mathbb{C})}+\big\| \bm{\kappa}^{\varphi,\, {\bm X}_{i}-I_2}_2(0,0) - \bm{\kappa}^{\varphi, \,{\bm Y}_{i}-I_2}_2(0,0)\big\|_{\mathcal{M}_{2}(\mathbb{C})}\bigg).\nonumber 
\end{align*}
Applying Lemma \ref{Lemma:moment-est}, our estimate for $\trphi \big(L_{i,2}\big)$ follows. 
\end{proof}
%%%%%%%%%%%%%%%%%%%%%%%%%%%%%%%%%%%%%%%%%%%
% Subsection: 3rd order term
%%%%%%%%%%%%%%%%%%%%%%%%%%%%%%%%%%%%%%%%%%%

\subsubsection*{\bf Third Order Term}

\begin{proposition}
\label{prop:thirdorderterm}
Let $1\leq i\leq n$.
Referring to the notations introduced in Section \ref{sec:invarianceppl}, for all $\zeta > 0$:
\begin{multline*}
|\trphi\big(L_{i,3} \big)|\leq( M_{g,x}^n\|\pi^{g,x}_{[1,n]} \|_{L^{8}_\varphi} +M_{h,y}^n\|\pi^{h,y}_{[1,n]} \|_{L^{8}_\varphi})^6 \\\times (\|\|g\|_{1,[x_i^-,x_i^+]}\|_{L^{4(1+\zeta^{-1})}_\varphi}^3 \vee \|\|h\|_{1,[y_i^-,y_i^+]}\|_{L^{4(1+\zeta^{-1})}_\varphi}^3) 
(\|x\|_{L^{4(1+\zeta)}_\varphi}^3 \vee\|y\|_{L^{4(1+\zeta)}_\varphi}^3 )
Q_3 ,
\end{multline*}
where $M_{g,x}^n$ and $M_{h,y}^n$ are defined in \eqref{eqn:mngx} and $Q_3$ is a polynomial expression of a finite number of positive fractional powers of $\| \|g\|_{1,[x_i^-,x_i^+]} \|_{L^{4(1+\zeta^{-1})}_\varphi}$, $\|\|h\|_{1,[y_i^-,y_i^+]}\|_{L^{4(1+\zeta^{-1})}_\varphi}$, $\|x\|_{L^{4(1+\zeta)}_\varphi}$, and $\|y\|_{L^{4(1+\zeta)}_\varphi}$.
\end{proposition}

\begin{proof}[Proof of Proposition \ref{prop:thirdorderterm}]
We follow the same steps as in the two previous subsections. We set first
\begin{align*}
L_{i,3}^{x} &:= e^{{\sf i}u \alpha_0 Z_i}U_i (\bm{X}_{i}-I_2) J U^{\star}_i e^{{\sf i}u \alpha_1 Z^1_i}U_i (\bm{X}_{i}-I_2) J U^{\star}_i e^{{\sf i}u \alpha_2 Z^1_i}U_i (\bm{X}_{i}-I_2) J U^{\star}_i e^{{\sf i}u \alpha_3 Z_i^1}
\end{align*}
\begin{remark}
As a first attempt to estimate the term $\trphi(L_{i,3}^{x})$, one would use H\"older inequality with respect to $\trphi$ and use the fact that, for any $\alpha > 0$, $\max\{\|e^{{\sf i}u \alpha Z_i}\|, \|e^{{\sf i}u \alpha Z_i^1}\|\} \leq 1$ to obtain: 
\[
|\trphi(L_{i,3}^{x})| \leq \trphi (|U_i (\bm{X}_{i}-I_2) J U^{\star}_i|^3).
\]
However, this approach does not allow taking advantage of the independence between $U_i$ and $\bm{X}_{i}$, resulting in less optimal moment conditions for the $x_i$ compared to proceed as follows. 
\end{remark}
Following on the above remark, we apply the fact that $\trphi$ is tracial and then use the Cauchy-Schwartz inequality to obtain
\begin{align}
&\trphi\Big(e^{{\sf i}u \alpha_0 Z_i} U_i  (\bm{X}_{i}-I_2) J U_i^{\star}e^{{\sf i}u \alpha_1 Z^1_i} U_i   (\bm{X}_{i}-I_2)  J U_i^{\star}e^{{\sf i}u \alpha_2 Z^1_i} U_i   (\bm{X}_{i}-I_2)  J U_i^{\star}e^{{\sf i}u \alpha_3 Z_i^1}\Big) \nonumber\\
& = \trphi(U_i  (\bm{X}_{i}-I_2)  J U_i^{\star}e^{{\sf i}u \alpha_1 Z_i^1}U_i  (\bm{X}_{i}-I_2)  J  U_i^{\star}e^{{\sf i}u \alpha_2 Z_i^1} \cdot U_i (\bm{X}_{i}-I_2)  J U_i^{\star}e^{{\sf i}u \alpha_3 Z_i^1}e^{{\sf i}u \alpha_0 Z_i}) \nonumber \\
& \leq \trphi\Big(U_i   (\bm{X}_{i}-I_2)  J U_i^{\star}e^{{\sf i}u \alpha_1 Z_i^1}U_i   (\bm{X}_{i}-I_2) J  U_i^{\star}\big|e^{{\sf i}u \alpha_2 Z_i^1}\big|^2 U_i J (\bm{X}_{i}-I_2)^{\star}   U_i^{\star}e^{-{\sf i}u \alpha_1 Z_i^1}U_i J   \nonumber\\
&\hspace{1cm} \cdot (\bm{X}_{i}-I_2)^{\star}  U_i^{\star}\Big)^{\frac{1}{2}} \trphi\Big(U_i   (\bm{X}_{i}-I_2)  J  U_i^{\star}e^{{\sf i}u \alpha_3 Z_i^1}|e^{{\sf i}u \alpha_0 Z_i}|^2e^{-{\sf i}u \alpha_3 Z_i^1}U_i J (\bm{X}_{i}-I_2)^{\star} U_i^{\star}\Big)^{\frac{1}{2}} . \nonumber
\end{align}
Now, set :
\begin{eqnarray*}
\widetilde{L}_{i,3}^{x} &= &  U_i  (\bm{X}_{i}-I_2)  J U_i^{\star}e^{{\sf i}u \alpha_1 Z_i^1}U_i  (\bm{X}_{i}-I_2) J  U_i^{\star} U_i J(\bm{X}_{i}^{\star}-I_2)  U_i^{\star}e^{-{\sf i}u \alpha_1 Z_i^1}U_i J (\bm{X}_{i}^{\star}-I_2) U_i^{\star}, 
\\
\widehat{L}_{i,3}^{x} & = & U_i  (\bm{X}_{i}-I_2)J  U_i^{\star} U_i J(\bm{X}_{i}^{\star}-I_2)U_i^{\star}.
\end{eqnarray*}
By the traciality of $ \varphi^{(2)}$ on the subalgebra of diagonal matrices in $\mathcal{M}_2(\mathcal{A)}$ and the fact that  $U_i$, $\bm{X}_{i}$ are $\star$-free with respect to $\varphi^{(2)}$, we get 
\begin{align*}
\varphi^{(2)}\big(\widehat{L}_{i,3}^{x}\big) 
&=  \kappa_2^{\varphi^{(2)}} \big((\bm{X}_{i}-I_2)\varphi^{(2)} \big(J  U_i^{\star} U_i J ),\bm{X}_{i}^{\star}-I_2 \big) \varphi^{(2)}\big(U_i^{\star}U_i \big) 
\\& \qquad\qquad\qquad+ \varphi^{(2)}\big(\bm{X}_{i}-I_2\big) \varphi^{(2)} \big(J  U_i^{\star} U_i J \varphi^{(2)} \big( \bm{X}_{i}^{\star}-I_2 \big) U_i^{\star}U_i \big)
\\&= \kappa_2^{\varphi^{(2)}} \big( \bm{X}_{i}-I_2 ,\bm{X}_{i}^{\star}-I_2 \big) \varphi^{(2)} \big(J  U_i^{\star} U_i J )\varphi^{(2)}\big(U_i^{\star}U_i \big) 
\\& \qquad\qquad\qquad+ \varphi^{(2)}\big(\bm{X}_{i}-I_2 \big) \varphi^{(2)} \big(\bm{X}_{i}^{\star}-I_2 \big) \varphi^{(2)} \big(J  U_i^{\star} U_i J  U_i^{\star}U_i \big),
\end{align*}
where the second equality follows from the fact that $\varphi^{(2)} \big(J  U_i^{\star} U_i J )$ and $\varphi^{(2)} \big(\bm{X}_{i}^{\star}-I_2 \big)$ are complex-valued diagonal matrices. Then, applying the trace to both sides of the previous equality yields
\begin{align*}
\Big|\trphi\big(\widehat{L}_{i,3}^{x}\big) \Big|
&\leq \big\| \kappa_2^{\varphi^{(2)}} \big(  \bm{X}_{i}-I_2  , \bm{X}_{i}^{\star}-I_2 \big)\big\|_{\mathcal{M}_2(\mathbb{C})} \big\| \varphi^{(2)}(U_i^{\star} U_i)\big\|_{\mathcal{M}_2(\mathbb{C})}^2 \\
&\hspace{4cm}+\big\| \varphi^{(2)}\big( \bm{X}_{i}-I_2 \big)\big\|^2_{\mathcal{M}_2(\mathbb{C})} \tr |\varphi^{(2)}({J}U_i^{\star} U_i J  U_i^{\star}U_i) | \\
&\leq \big\| \kappa_2^{\varphi^{(2)}} \big(  \bm{X}_{i}-I_2  , \bm{X}_{i}^\star-I_2 \big)\big\|_{\mathcal{M}_2(\mathbb{C})} \big\| \varphi^{(2)}(U_i^{\star} U_i)\big\|^2_{\mathcal{M}_2(\mathbb{C})} \\
&\hspace{4cm}+\big\| \varphi^{(2)}\big( \bm{X}_{i}-I_2  \big)\big\|^2_{\mathcal{M}_2(\mathbb{C})} \| \varphi^{(2)}(JU^{\star}_i U_iJU^{\star}_i U_i)\|_{\mathcal{M}_2(\mathbb{C})}.
\end{align*}
By applying the inequality in Lemma \ref{thm:NC L^p-inequality}, we infer 
\begin{align*}
\|\varphi^{(2)}(U_i^{\star}U_i) \|_{\mathcal{M}_2(\mathbb{C})}&=\max\big(\varphi(\pi_{[1,i-1]}^{g,x}(\pi_{[1,i-1]}^{g,x})^{\star}), \varphi(\pi_{[i+1,n]}^{h,y}(\pi_{[i+1,n]}^{h,y})^{\star}) \big)\\ 
&\leq  \max\big(\varphi(X_{i})\cdots \varphi(X_{n}){|^{-2}}\| \pi^{g,x}_{[1,n]}\|^2_{L^2_\varphi}, {|}\varphi(Y_{1})\cdots \varphi(Y_{i}){|^{-2}}\|\pi^{h,y}_{[1,n]} \|^2_{L^2_\varphi} \big) \\
&\leq \max\big((M_{g,x}^n)^2\pi^{g,x}_{[1,n]}\|^2_{L^2_\varphi}, (M_{h,y}^n)^2\|\pi^{h,y}_{[1,n]} \|^2_{L^2_\varphi} \big) \\
&\leq 
\big(M_{g,x}^n\|\pi^{g,x}_{[1,n]} \|_{L^2_{\varphi}}+M_{h,y}^n\|\pi^{h,y}_{[1,n]} \|_{L^2_\varphi} \big)^2
\end{align*}
and also
\begin{align*}
\|\varphi^{(2)}(JU_i^{\star}U_iJU_i^{\star}U_i) \|_{\mathcal{M}_2(\mathbb{C})} &=\varphi(\pi_{[1,i-1]}^{g,x}(\pi_{[1,i-1]}^{g,x})^{\star}\pi_{[i+1,n]}^{h,y}(\pi_{[i+1,n]}^{h,y})^{\star})\\ 
&\leq  {|}\varphi(X_{i})\cdots \varphi(X_{n}){|^{-2}}\| \pi^{g,x}_{[1,n]}\|^2_{L^4_\varphi} {|}\varphi(Y_{1})\cdots \varphi(Y_{i}){|^{-2}}\|\pi^{h,y}_{[1,n]} \|^2_{L^4_\varphi}  \\
&\leq  (M_{g,x}^n)^2(M_{h,y}^n)^2\| \pi^{g,x}_{[1,n]}\|^2_{L^4_\varphi}\|\pi^{h,y}_{[1,n]} \|^2_{L^4_\varphi} \\
&\leq  \big(M_{g,x}^n\| \pi^{g,x}_{[1,n]}\|_{L^4_\varphi}+M_{h,y}^n\|\pi^{h,y}_{[1,n]} \|_{L^4_\varphi} \big)^4.
\end{align*}
Then, by applying monotonicity of $L^p$-norms, we get that 
\begin{align}
\Big|\trphi\big(\widehat{L}_{i,3}^{x}\big)\Big|  &\leq 
\big(M_{g,x}^n\| \pi^{g,x}_{[1,n]}\|_{L^4_\varphi} + M_{h,y}^n\|\pi^{h,y}_{[1,n]} \|_{L^4_\varphi} \big)^4 \nonumber\\
&\hspace{0.5cm}\times \Big(
\big\| \kappa_2^{\varphi^{(2)}} \big(  (\bm{X}_{i}-I_2)  , (\bm{X}^{\star}_{i}-I_2) \big)\big\|_{\mathcal{M}_2(\mathbb{C})}^2 
+\big\| \varphi^{(2)}\big( (\bm{X}_{i}-I_2)  \big)\big\|_{\mathcal{M}_2(\mathbb{C})}^2\Big). \nonumber
\end{align}
Then we apply the inequalities in Remark \ref{remark: 1st 2rd moment estimate } to obtain
\begin{align}\label{eqn: estimate lambda-L_3,x}
\Big|\trphi\big(\widehat{L}_{i,3}^{x}\big)\Big|  &\lesssim 
\big(M_{g,x}^n\| \pi^{g,x}_{[1,n]}\|_{L^4_\varphi} + M_{h,y}^n\|\pi^{h,y}_{[1,n]} \|_{L^4_\varphi} \big)^4 \nonumber\\
&\hspace{0.5cm}\times 
\Big( \| \|g\|_{1,[x_i^-,x_i^+]}\|_{L^{1+\zeta^{-1}}_{\varphi}}^2 \|x_i\|_{L^{1+\zeta}_{\varphi}}^2 + C_1 \|\|g\|_{1,[x_i^-,x_i^+]}\|_{L^{2(1+\zeta^{-1})}_{\varphi}}^4 \|x_i\|_{L^{2(1+\zeta)}_{\varphi}}^4\Big),
\end{align}
for some $C_1>0$.
To control the term $\widetilde{L}_{i,3}^{x}$, we first use Lemma \ref{Prop Fourier cosine} to express $e^{{\sf i}u \alpha_1 Z_i^1} = D_{i, u\alpha_1} + \widetilde{D}_{i,u\alpha_1} J $, thus obtaining expressions that contain only diagonal matrices and $J$'s. {Recall that $D_{i, u\alpha_1}$ and $\widetilde{D}_{i, u\alpha_1}$ are defined in \eqref{eqn:D_{i,u}_and_tilde-D_{i,u}}.} Note that the terms where the matrix $J$ appears an odd number of times have zero as diagonal entries; hence, their trace is zero. Therefore, by taking the trace,
\begin{align*}
    &\trphi\big(\widetilde{L}_{i,3}^{x} \big) 
    \\&= \trphi \Big(  (\bm{X}_{i}-I_2)  J U_i^{\star}D_{i, u\alpha_1} U_i  (\bm{X}_{i}-I_2) J  U_i^{\star} U_i J(\bm{X}_{i}^{\star}-I_2) U_i^{\star}D_{i, -u\alpha_1} U_i J (\bm{X}_{i}^{\star}-I_2) U_i^{\star}U_i \Big)
    \\ & - \trphi \Big(  (\bm{X}_{i}-I_2)  J U_i^{\star}\widetilde{D}_{i,u\alpha_1} J  U_i  (\bm{X}_{i}-I_2) J  U_i^{\star} U_i J(\bm{X}_{i}^{\star}-I_2) U_i^{\star} \widetilde{D}_{i,-u\alpha_1} J U_i J(\bm{X}_{i}^{\star}-I_2)U_i^{\star}U_i \Big)
    \\ & :=   \trphi\big(\widetilde{L}_{i,3}^{x,(1)} \big) +\trphi\big(\widetilde{L}_{i,3}^{x,(2)} \big).
\end{align*}
Using the moment-cumulant formula, the fact that the matrix-valued cumulants of $(\bm{X}_{i}-I_2)$ commute with diagonal matrices, we obtain that 
\begin{align*}
\varphi^{(2)}\big(\widetilde{L}_{i,3}^{x,(1)} \big)  & =   \sum_{\pi \in NC(4)} \kappa_\pi^{\varphi^{(2)}}  \big( (\bm{X}_{i}-I_2)J, \bm{X}_{i}-I_2 ,\bm{X}_{i}^{\star}-I_2 , J(\bm{X}_{i}^{\star}-I_2) \big) \\
&\hspace{2cm} \times \varphi^{(2)}_{\pi^c} \big(U_i^{\star}D_{i, u\alpha_1} U_i, J  U_i^{\star} U_i J,  U_i^{\star}D_{i, -u\alpha_1} U_i ,U_i^{\star}U_i\big).
\end{align*}
We apply the trace to each term of the sum in the right-hand of the above equation, and discuss first how to bound the following term
\begin{align*}
\tr_2 |\varphi^{(2)}_{\pi^c} \big(U_i^{\star}D_{i, u\alpha_1} U_i, J  U_i^{\star} U_i J,  U_i^{\star}D_{i, -u\alpha_1} U_i ,U_i^{\star}U_i\big)|.
\end{align*}
With this aim, we would need to bound terms of the kind
$\|\varphi^{(2)}(W)\|_{\mathcal{M}_{2}(\mathbb{C})}$ where $W$ is a product of terms drawn from $\{U_i^{\star}D_{i, u\alpha_1} U_i, J  U_i^{\star} U_i J,  U_i^{\star}D_{i, -u\alpha_1} U_i ,U_i^{\star}U_i\}$.
\begin{remark}
Set 
$R_i = U^{\star}_iU_i$, and note that $\| \varphi^{(2)}(R_i) \|_{\mathcal{M}_{2}(\mathbb{C})} \leq \varphi(\pi_{[1,i-1]}^{g,x}(\pi_{[1,i-1]}^{g,x})^{\star})+\varphi(\pi_{[i+1,n]}^{h,y}(\pi_{[i+1,n]}^{h,y})^{\star})$. 
We could be tempted to write
\begin{equation*}
U_i^{\star} D_{i,u\alpha_1} U_i = U_i^{\star} \cos(u\alpha_1|U_iJU_i^{\star}|)U_i =\cos(u\alpha_1U_i^{\star}U_iJ)U_i^{\star} U_i = \cos(u\alpha_1 R_iJ)R_i
\end{equation*}
but $R_iJ$ is not self-adjoint, so one can not bound the operator norm of $\cos(u\alpha_1 R_iJ)$ by $1$.
\end{remark}
 \noindent
 We will proceed with bounding the terms corresponding to each partition $\pi$, one by one.
 \\ {\bf $\bullet \pi^c = \{ \{ 1,2\}, \{3,4\}\}$:} Note that since $ \varphi^{(2)}$ is tracial on the subalgebra of diagonal matrices in $\mathcal{M}_2(\mathcal{A)}$, 
\begin{align*}
\|\varphi^{(2)}(U_i^{\star}D_{i,u\alpha_1}U_i JU^{\star}_iU_iJ) \|_{\mathcal{M}_2(\mathbb{C})}& =  \|\varphi^{(2)}(D_{i,u\alpha_1}U_i JU^{\star}_iU_iJU_i^{\star})\|_{\mathcal{M}_2(\mathbb{C})} \\
& \leq \|\varphi^{(2)}(|U_i JU^{\star}_iU_iJU_i^{\star}|)\|_{\mathcal{M}_2(\mathbb{C})} = \|\varphi^{(2)}(U_i JU^{\star}_iU_iJU_i^{\star})\|_{\mathcal{M}_2(\mathbb{C})} \\&= \|\varphi^{(2)}(U_i^{\star}U_i JU^{\star}_i U_iJ)\|_{\mathcal{M}_2(\mathbb{C})} 
\\& = \varphi\big( \pi^{g,x}_{[1,i-1]}(\pi^{g,x}_{[1,i-1]})^{\star}\pi^{h,y}_{[i+1,n]}(\pi^{h,y}_{[i+1,n]})^{\star}\big) \\
&\leq \| \pi^{g,x}_{[1,i-1]} \|_{L^4_\varphi}^2\| \pi^{h,y}_{[i+1,n]} \|_{L^4_\varphi}^2,
\end{align*}
where the last inequality holds by applying Cauchy-Schwartz inequality. Then by applying the inequality in Lemma \ref{thm:NC L^p-inequality}, one gets
\begin{align*}
\|\varphi^{(2)}(U_i^{\star}D_{i,u\alpha_1}U_i JU^{\star}_iU_iJ) \|_{\mathcal{M}_2(\mathbb{C})}
& \leq  | \varphi(X_i) \cdots \varphi(X_n)|^{-2}| \varphi(Y_1) \cdots \varphi(Y_{i})|^{-2} \|\pi^{g,x}_{[1,n]} \|_{L^4_\varphi}^2\| \pi^{h,y}_{[1,n]} \|_{L^4_\varphi}^2.
\end{align*}
Similarly, 
\begin{align*}
\|&\varphi^{(2)}(U_i^{\star}D_{i,-u\alpha_1}U_i U^{\star}_iU_i) \|_{\mathcal{M}_2(\mathbb{C})}\\& =  \|\varphi^{(2)}(D_{i,-u\alpha_1}U_i U^{\star}_iU_iU_i^{\star})\|_{\mathcal{M}_2(\mathbb{C})} \leq \|\varphi^{(2)}(|U_i U^{\star}_iU_iU_i^{\star}|)\|_{\mathcal{M}_2(\mathbb{C})} \\
&= \max\Big(\varphi\big(\big(\pi^{g,x}_{[1,i-1]}(\pi^{g,x}_{[1,i-1]})^{\star}\big)^2\big),  \varphi\big(\big(\pi^{h,y}_{[i+1,n]}(\pi^{h,y}_{[i+1,n]})^{\star}\big)^2\big)\Big) \\
&\leq \max \Big(|\varphi(X_i)\cdots\varphi(X_n)|^{-4}\|\pi^{g,x}_{[1,n]}\|^4_{L^4_\varphi},|\varphi(Y_1)\cdots \varphi(Y_i)|^{-4}\| \pi^{h,y}_{[1,n]}\|^4_{L^4_\varphi}\Big).
\end{align*}
By noting that $(a+b)^8 \geq  a^2b^2max(a^4,b^4)$, we get 
\begin{align}\label{eqn:boundpair}
\|\varphi^{(2)}(U_i^{\star} & D_{i,u\alpha_1}U_i JU^{\star}_iU_iJ) \|_{\mathcal{M}_2(\mathbb{C})} \|\varphi^{(2)}(U_i^{\star}D_{i,{-u\alpha_1}}U_i U^{\star}_iU_i) \|_{\mathcal{M}_2(\mathbb{C})} \nonumber \\ 
&\leq \Big(|\varphi(X_i)\cdots\varphi(X_n)|^{-1}\|\pi^{g,x}_{[1,n]}\|_{L^4_\varphi} + |\varphi(Y_1)\cdots \varphi(Y_i)|^{-1}\| \pi^{h,y}_{[1,n]}\|_{L^4_\varphi}\Big)^8\nonumber \\
&\leq \Big(M_{g,x}^n\|\pi^{g,x}_{[1,n]}\|_{L^4_\varphi} + M_{h,y}^n\| \pi^{h,y}_{[1,n]}\|_{L^4_\varphi}\Big)^8. 
\end{align}
{\bf $\bullet \pi^c = \{ \{ 1,4\}, \{2 ,3\}\}$:} It follows from the same computations as above that this term is bounded by \eqref{eqn:boundpair}.
{\bf $ \bullet \pi^c =   \{\{ 1,2,3,4\} \}$ :}  
We apply the Cauchy-Schwartz inequality and obtain
\begin{align*}
&\|\varphi^{(2)}(U_i^{\star}D_{i, u\alpha_1} U_iJU_i^{\star}U_i JU_i^{\star}D_{i,- u\alpha_1} U_i U_i^{\star}U_i) \|_{\mathcal{M}_2(\mathbb{C})} \\&= \| \varphi^{(2)}(D_{i, u\alpha_1} U_iJU_i^{\star}U_i JU_i^{\star}D_{i, -u\alpha_1} U_i U_i^{\star}U_iU_i^{\star}) \|_{\mathcal{M}_2(\mathbb{C})} \\
&\leq \| \varphi^{(2)}(|D_{i, u\alpha_1} U_iJU_i^{\star}U_i JU_i^{\star}|^2)\|_{\mathcal{M}_2(\mathbb{C})}^{\frac{1}{2}} \|\varphi^{(2)}(|D_{i, -u\alpha_1} U_i U_i^{\star}U_iU_i^{\star}|^2) \|_{\mathcal{M}_2(\mathbb{C})}^{\frac{1}{2}}  \\
&\leq \| \varphi^{(2)}(|U_iJU_i^{\star}U_i JU_i^{\star}|^2)\|_{\mathcal{M}_2(\mathbb{C})}^{\frac{1}{2}}  \|\varphi^{(2)}(|U_i U_i^{\star}U_iU_i^{\star}|^2)\|_{\mathcal{M}_2(\mathbb{C})}^{\frac{1}{2}}  .
\end{align*}
By applying Lemma \ref{thm:NC L^p-inequality} yields 
\begin{align}\label{bound: M_2 norm of UU*}
\|\varphi^{(2)}(|U_i U_i^{\star}U_iU_i^{\star}|^2) \|_{\mathcal{M}_2(\mathbb{C})}^{\frac{1}{2}} &= \|\varphi^{(2)}(U_i U_i^{\star}U_iU_i^{\star}U_i U_i^{\star}U_iU_i^{\star})\|_{\mathcal{M}_2(\mathbb{C})}^{\frac{1}{2}}  =\max\Big(\|\pi^{g,x}_{[1,i-1]}\|^4_{L^8_\varphi},  \| \pi^{h,y}_{[i+1,n]}\|_{L^8_\varphi}^4 \Big)  \nonumber\\
&\leq  \max\Big(|\varphi(X_i)\cdots\varphi(X_n)|^{-4}\|\pi^{g,x}_{[1,n]}\|^4_{L^8_\varphi},  |\varphi(Y_1)\cdots \varphi(Y_i)|^{-4}\| \pi^{h,y}_{[1,n]}\|_{L^8_\varphi}^4 \Big) . 
\end{align}
In addition, we use H\"older's inequality, Lemma \ref{thm:NC L^p-inequality}, and the monotonicity of $L^p$-norms to derive
\begin{align}\label{bound moments-1}
\| &\varphi^{(2)}(|U_iJU_i^{\star}U_i JU_i^{\star}|^2)\|_{\mathcal{M}_2(\mathbb{C})}^{\frac{1}{2}}\\&=\|\varphi^{(2)}(U_i^{\star}U_iJU_i^{\star}U_i JU_i^{\star}U_iJU_i^{\star}U_iJ)\|_{\mathcal{M}_2(\mathbb{C})}^{\frac{1}{2}} \nonumber \\
&\leq |\varphi\big(\pi^{g,x}_{{[1,i-1]}}(\pi^{g,x}_{{[1,i-1]}})^{\star}\pi^{h,y}_{{[i+1,n]}}(\pi^{h,y}_{{[i+1,n]}})^{\star} \pi^{g,x}_{{[1,i-1]}}(\pi^{g,x}_{{[1,i-1]}})^{\star} \pi^{h,y}_{{[i+1,n]}}(\pi^{h,y}_{{[i+1,n]}})^{\star} \big)|^{\frac{1}{2}} \nonumber\\ 
&\leq \|\pi^{g,x}_{[1, i-1]}(\pi^{g,x}_{[1, i-1]})^{\star} \|_{L^4_\varphi}\|\pi^{h,y}_{ [i+1,n]}(\pi^{h,y}_{[i+1,n]})^{\star} \|_{L^4_\varphi} \nonumber \\
&\leq \|\pi^{g,x}_{[1,i-1]} \|^2_{L^8_\varphi}\|\pi^{h,y}_{[i+1,n]}\|^2_{L^8_\varphi} \nonumber\\
&\leq  |\varphi(X_{i})\cdots \varphi(X_n) |^{-2} |\varphi(Y_{1})\cdots \varphi(Y_{i}) |^{-2} \|\pi^{g,x}_{[1,n]} \|^2_{L^8_\varphi}\|\pi^{h,y}_{[1,n]}\|^2_{L^8_\varphi} .
\end{align}
Combining the above bounds together and using the inequality $(a+b)^8\geq a^2b^2\max(a^4,b^4)$, we obtain
\begin{align*}
&\|\varphi^{(2)}(U_i^{\star}D_{i, u\alpha_1} U_iJU_i^{\star}U_i JU_i^{\star}D_{i,- u\alpha_1} U_i U_i^{\star}U_i) \|_{\mathcal{M}_2(\mathbb{C})} \\
&\hspace{1cm}\leq   \Big( |\varphi(X_{i})\cdots \varphi(X_n) |^{-1}  \|\pi^{g,x}_{[1,n]} \|_{L^8_\varphi}+ |\varphi(Y_{1})\cdots \varphi(Y_{i}) |^{-1}\|\pi^{h,y}_{[1,n]}\|_{L^8_\varphi} \Big)^8 \nonumber \\
&\hspace{1cm}\leq  \Big(M_{g,x}^n  \|\pi^{g,x}_{[1,n]} \|_{L^8_\varphi}+ M_{h,y}^n\|\pi^{h,y}_{[1,n]}\|_{L^8_\varphi} \Big)^8.
\end{align*}
$\bullet \pi^c = \{ \{1,2,3\}, \{ 4\}\}$: 
Note that
\begin{align}\label{estimate: phi-UU*}
\|\varphi^{(2)}(U_iU^{\star}_i) \|_{\mathcal{M}_2(\mathbb{C})} &= \max(\|\pi^{g,x}_{[1,i-1]} \|^2_{L^2_\varphi}, \|\pi^{h,y}_{[i+1,n]} \|^2_{L^2_\varphi} )\nonumber \\
&\leq \max (|\varphi(X_i)\cdots\varphi (X_n)|^{-2}\| \pi^{g,x}_{[1,n]}\|^2_{L^2_\varphi}, |\varphi(Y_1)\cdots \varphi(Y_i)|^{-2}\| \pi^{h,y}_{[1,n]}\|^2_{L^2_\varphi} ) .
\end{align}
In addition, by applying the Cauchy-Schwartz inequality, we derive
\begin{align*}
\| \varphi^{(2)}(U_i^{\star}D_{i, u\alpha_1} U_i  J  U_i^{\star} U_i J U_i^{\star}D_{i, -u\alpha_1} U_i)\|_{\mathcal{M}_2(\mathbb{C})} &= 
\| \varphi^{(2)}(D_{i, u\alpha_1} U_i  J  U_i^{\star} U_i J U_i^{\star}D_{i, -u\alpha_1} U_iU_i^{\star}) \|_{\mathcal{M}_2(\mathbb{C})}\nonumber\\
&\leq \| \varphi^{(2)}(|U_i  J  U_i^{\star} U_i J U_i^{\star}|^2) \|_{\mathcal{M}_2(\mathbb{C})}^{\frac{1}{2}} \| \varphi(|U_iU_i^{\star}|^2)\|_{\mathcal{M}_2(\mathbb{C})}^{\frac{1}{2}}. 
\end{align*}
Using an analogous estimation as in \eqref{bound: M_2 norm of UU*}, we observe
\begin{align}\label{estimate: UU*-1/2-2}
& \| \varphi^{(2)}(|U_iU_i^{\star}|^2)\|_{\mathcal{M}_2(\mathbb{C})}^{\frac{1}{2}} \leq  \max (|\varphi(X_i)\cdots\varphi (X_n)|^{-2}\| \pi^{g,x}_{[1,n]}\|^2_{L^4_\varphi}, |\varphi(Y_1)\cdots \varphi(Y_i)|^{-2}\| \pi^{h,y}_{[1,n]}\|^2_{L^4_\varphi} ).
,\end{align}
Thus, by \eqref{bound moments-1} and the above bounds, and by applying a similar argument as in the previous cases, we find that the term associated with the partition $\pi^c = \{ \{1,2,3\}, \{ 4\}\}$ is bounded by 
\begin{align*}
\Big( M_{g,x}^n \|\pi^{g,x}_{[1,n]} \|_{L^8_\varphi}+M_{h,y}^n\|\pi^{h,y}_{[1,n]}\|_{L^8_\varphi}\Big)^8.
\end{align*}
$\bullet \pi^c = \{\{1,3,4\},\{ 2\} \}:$
Applying the Cauchy-Schwartz inequality to deduce
\begin{align*}
&\| \varphi^{(2)}(U_i^{\star}D_{i, u\alpha_1} U_i \varphi^{(2)}(J  U_i^{\star} U_i J)U_i^{\star}D_{i, -u\alpha_1}U_iU_i^{\star}U_i)\|_{\mathcal{M}_2(\mathbb{C})}\\
&\hspace{2cm} \leq\|\varphi^{(2)}(U_i^{\star}D_{i, u\alpha_1} U_i U_i^{\star}D_{i, -u\alpha_1}U_iU_i^{\star}U_i)\|_{\mathcal{M}_2(\mathbb{C})}\|\varphi^{(2)}(J  U_i^{\star} U_i J)\|_{\mathcal{M}_2(\mathbb{C})} \\
&\hspace{2cm}\leq \| \varphi^{(2)}(|U_i  J  U_i^{\star} U_i J U_i^{\star}|^2) \|_{\mathcal{M}_2(\mathbb{C})}^{\frac{1}{2}} \| \varphi(|U_iU_i^{\star}|^2)\|_{\mathcal{M}_2(\mathbb{C})}^{\frac{1}{2}}\|\varphi^{(2)}(U_i^{\star} U_i )\|_{\mathcal{M}_2(\mathbb{C})}. 
\end{align*}
Hence, we obtain the same upper bound for the term corresponding to the partition $\pi^c=\{\{1,3,4\},\{2\}\}$ as for the term corresponding to the partition $\{\{1,2,3\},\{4\}\}.$ \\
$\bullet \pi^c = \{\{1,2,4 \}, \{ 3\}\}:$
we first note that 
\begin{align*}
&\| \varphi^{(2)}(U_i^{\star}D_{i, u\alpha_1} U_i J  U_i^{\star} U_i J\varphi^{(2)}(U_i^{\star}D_{i, -u\alpha_1}U_i)U_i^{\star}U_i)\|_{\mathcal{M}_2(\mathbb{C})}\\
&\hspace{2cm}\leq
\| \varphi^{(2)}(U_i^{\star}D_{i, u\alpha_1} U_i J  U_i^{\star} U_i JU_i^{\star}U_i)\|_{\mathcal{M}_2(\mathbb{C})} \| \varphi^{(2)}(U_i^{\star}D_{i, -u\alpha_1} U_i)\|_{\mathcal{M}_2(\mathbb{C})} \\
&\hspace{2cm}\leq \|\varphi^{(2)}(|U_i J  U_i^{\star} U_i JU_i^{\star}U_iU^{\star}_i|) \|_{\mathcal{M}_2(\mathbb{C})} \|\varphi^{(2)}(U_iU^{\star}_i) \|_{\mathcal{M}_2(\mathbb{C})}. 
\end{align*}
Observe that 
\begin{multline}\label{estimate: UJU*UJU*UU*}
\|\varphi^{(2)}(|U_i J  U_i^{\star} U_i JU_i^{\star}U_iU^{\star}_i|) \|_{\mathcal{M}_2(\mathbb{C})}
\\ \leq \max\big(\|\pi^{g,x}_{[1,i-1]}\|^4_{L^{6}_\varphi} \|_{\mathcal{M}_2(\mathbb{C})} \pi^{h,y}_{[i+1,n]} \|^2_{L^{6}_\varphi},\|\pi^{h,y}_{[i+1,n]}\|^4_{L^{6}_\varphi} \| \pi^{g,x}_{[1,i-1]} \|^2_{L^{6}_\varphi} \big)
\end{multline}
Thus, by combining \eqref{estimate: phi-UU*} and \eqref{estimate: UJU*UJU*UU*}, and  the monotonicity of $L^p$-norms, we obtain that the term corresponding to the partition $\pi^c$ is bounded by 
\begin{align*}
 \Big( M_{g,x}^n \|\pi^{g,x}_{[1,n]} \|_{L^6_\varphi}+M_{h,y}^n\|\pi^{h,y}_{[1,n]}\|_{L^6_\varphi}\Big)^8.
\end{align*}
 $\bullet \pi^c = \{\{2,3,4\},\{1\}\}:$
We observe that
\begin{align*}
&\| \varphi^{(2)}(U_i^{\star}D_{i, u\alpha_1} U_i)\varphi^{(2)}(J  U_i^{\star} U_i JU_i^{\star}D_{i, -u\alpha_1}U_iU_i^{\star}U_i)\|_{\mathcal{M}_2(\mathbb{C})}\\
&\hspace{2cm}\leq \|\varphi^{(2)}(U_i^{\star}D_{i, u\alpha_1} U_i)\|_{\mathcal{M}_2(\mathbb{C})} \|\varphi^{(2)}(J  U_i^{\star} U_i JU_i^{\star}D_{i, -u\alpha_1}U_iU_i^{\star}U_i)\|_{\mathcal{M}_2(\mathbb{C})} \\
&\hspace{2cm}\leq \|\varphi^{(2)}(U_iU^{\star}_i) \|_{\mathcal{M}_2(\mathbb{C})}\|\varphi^{(2)}(|U_i J  U_i^{\star} U_i JU_i^{\star}U_iU^{\star}_i|) \|_{\mathcal{M}_2(\mathbb{C})}.
\end{align*}
As a result, term corresponding to $\pi^c=\{\{2,3,4\},\{1\}\}$ has the same upper bound as for the term corresponding to the partition $\{\{1,2,4\},\{3\}\}.$
 \\
$\bullet \pi^c = \{ \{1\},\{2\}, \{3,4\}\}, \{ \{1\},\{3\}, \{2,4\}\}, \{ \{1\},\{4\}, \{2,3\}\}, \{ \{2\},\{3\}, \{1,4\}\},$ and $\{ \{3\},\{4\}, \{1,2\}\}:$
By Cauchy-Schwartz inequality, we have 
\begin{align*}
\|\varphi^{(2)}_{\pi^c}(U_i^{\star}D_{i, u\alpha_1} U_i,J U_i^{\star} U_i J,U_i^{\star}D_{i, -u\alpha_1}U_i,(U_i^{\star}U_i)\| \leq \|\varphi^{(2)}(|U_iU_i^{\star}|^2) \|_{\mathcal{M}_2(\mathbb{C})} \|\varphi^{(2)}(U_iU^{\star}_i) \|^2_{\mathcal{M}_2(\mathbb{C})}.
\end{align*}
Then, we combine \eqref{estimate: phi-UU*}, \eqref{estimate: UU*-1/2-2}, and use the monotonicity of $L^p$-norms to obtain that the term associated with the partition $\pi^c$ is bounded by
\begin{align*}
 \Big( M_{g,x}^n\|\pi^{g,x}_{[1,n]} \|_{L^4_\varphi} +M_{h,y}^n\|\pi^{h,y}_{[1,n]} \|_{L^4_\varphi}\Big)^8.
\end{align*}
$\bullet \pi^c = \{ \{1\},\{2\}, \{3\},\{ 4\}\}:$ note that
\begin{align*}
\|\varphi^{(2)}(U_i^{\star}D_{i, u\alpha_1} U_i) \varphi^{(2)}(J U_i^{\star} U_i J)\varphi^{(2)}(U_i^{\star}D_{i, -u\alpha_1}U_i)\varphi^{(2)}(U_i^{\star}U_i)\|_{\mathcal{M}_2(\mathbb{C})} \leq \|\varphi^{(2)}(U_iU^{\star}_i)\|^4_{\mathcal{M}_2(\mathbb{C})} 
\end{align*}
By applying \eqref{estimate: phi-UU*}, we show that the term corresponding to the partition $\pi^c$ is bounded by
\begin{align*}
 \Big( M_{g,x}^n\|\pi^{g,x}_{[1,n]} \|_{L^2_\varphi} +M_{h,y}^n\|\pi^{h,y}_{[1,n]} \|_{L^2_\varphi}\Big)^8.
\end{align*}
Again, by using the fact that the noncommutative $L^p$ norms are non-decreasing, we conclude 
\begin{align}\label{eqn: estimate tilde-L_3,x}
&|\trphi\big(\widetilde{L}_{i,3}^{x,(1)} \big)| \lesssim \Big( M_{g,x}^n\|\pi^{g,x}_{[1,n]} \|_{L^{8}_\varphi} +M_{h,y}^n\|\pi^{h,y}_{[1,n]} \|_{L^{8}_\varphi}\Big)^8 \nonumber \\ 
&\hspace{4cm}\sum_{\pi \in NC(4)} \| \kappa_\pi^{\varphi^{(2)}}  \big( (\bm{X}_{i}-I_2)J, \bm{X}_{i}-I_2 ,\bm{X}_{i}^{\star}-I_2 ,J(\bm{X}_{i}^{\star}-I_2) \big)\|_{\mathcal{M}_2(\mathbb{C})}.
\end{align}
Note that, using a similar argument, we can bound $|\trphi\big(\widetilde{L}_{i,3}^{x,(2)} \big)|$ with the same upper bound for $|\trphi\big(\widetilde{L}_{i,3}^{x,(1)} \big)|$. Then, we apply Corollary \ref{remark: 1st 2rd moment estimate } to conclude 
\begin{align*}
&|\trphi\big(\widetilde{L}_{i,3}^{x} \big)| \leq C_3 \Big( M_{g,x}^n\|\pi^{g,x}_{[1,n]} \|_{L^{8}_\varphi} +M_{h,y}^n\|\pi^{h,y}_{[1,n]} \|_{L^{8}_\varphi}\Big)^8 \big\|\|g\|_{1,[x_i^-,x_i^+]}\big\|_{L^4(1+\zeta^{-1})_\varphi}^4 \|x\|_{L^{4(1+\zeta)}_\varphi}^4,
\end{align*}
for some $C_3>0$. Therefore, we combine \eqref{eqn: estimate lambda-L_3,x} and \eqref{eqn: estimate tilde-L_3,x} to get
\begin{align*}
|&\trphi\big(L_{i,3}^{x} \big)|\\&\leq |\trphi\big(\widehat{L}_{i,3}^{x}\big)|^{\frac{1}{2}}|\trphi\big(\widetilde{L}_{i,3}^{x} \big)|^{\frac{1}{2}} \\
&\leq C_3  \big(M_{g,x}^n\| \pi^{g,x}_{[1,n]}\|_{L^4_\varphi} + M_{h,y}^n\|\pi^{h,y}_{[1,n]} \|_{L^4_\varphi} \big)^2 \Big( M_{g,x}^n\|\pi^{g,x}_{[1,n]} \|_{L^{8}_\varphi} +M_{h,y}^n\|\pi^{h,y}_{[1,n]} \|_{L^{8}_\varphi}\Big)^4 \\
&\quad \times \big\|\|g\|_{1,[x_i^-,x_i^+]}\big\|_{L^{4(1+\zeta^{-1})}_\varphi}^2 \|x\|_{L^{4(1+\zeta)}_\varphi}^2  \Big( \big\|\|g\|_{1,[x_i^-,x_i^+]}\big\|_{L^{1+\zeta^{-1}}_{\varphi}}^2 \|x\|_{L^{1+\zeta}_{\varphi}}^2\\&\qquad\qquad + \big\|\|g\|_{1,[x_i^-,x_i^+]}\big\|_{L^{2(1+\zeta^{-1})}_{\varphi}}^4 \|x\|_{L^{2(1+\zeta)}_{\varphi}}^4\Big)^{\frac{1}{2}},
\end{align*}
for some $C_3>0$. Then by monotonicity of $L^p$-norms and using the inequality $(a+b)^{\frac{1}{2}}\leq a^{\frac{1}{2}}+b^{\frac{1}{2}}$ for all $a,b>0$, we infer
\begin{align*}
|\trphi\big(L_{i,3}^{x} \big)|&\leq D \Big( M_{g,x}^n\|\pi^{g,x}_{[1,n]} \|_{L^{8}_\varphi} +M_{h,y}^n\|\pi^{h,y}_{[1,n]} \|_{L^{8}_\varphi}\Big)^6 \big\|\|g\|_{1,[x_i^-,x_i^+]}\big\|_{L^{4(1+\zeta^{-1})}_\varphi}^3 \|x\|_{L^{4(1+\zeta)}_\varphi}^3 
Q,\end{align*}
for some constant $D>0$ and   a commutative polynomial $Q$ in $\big\|\|g\|_{1,[x_i^-,x_i^+]}\big\|_{L^{4(1+\zeta^{-1})}_\varphi}$ and $ \|x\|_{L^{4(1+\zeta)}_\varphi}$. Finally, we conclude our proof by noting that $L_{i,3}=L_{i,3}^x-L_{i,3}^{y}$ and $|\trphi\big(L_{i,3}^{y}\big)|$ can be bounded in the same way as $|\trphi\big(L_{i,3}^{x} \big)|$ 
where
\begin{align*}
L_{i,3}^{y} &:= e^{{\sf i}u \alpha_0 Z_{i-1}}U_i (\bm{Y}_{i}-I_2) J U^{\star}_i e^{{\sf i}u \alpha_1 Z^1_i}U_i (\bm{Y}_{i}-I_2) J U^{\star}_i e^{{\sf i}u \alpha_2 Z^1_i}U_i (\bm{Y}_{i}-I_2) J U^{\star}_i e^{{\sf i}u \alpha_3 Z_i^1}.
\end{align*}
\end{proof}
%%%%%%%%%%%%%%%%%%%%%%%%%%%%%%%%%%%%%%%%%%%
% Section: rate of convergence
%%%%%%%%%%%%%%%%%%%%%%%%%%%%%%%%%%%%%%%%%%%
\subsection{Proof of Theorem \ref{thm:mainthm}}
\label{sec:proofmainthm}
Let $g$ and $x$ be as in the statement of Theorem \ref{thm:mainthm} and set $n^{-1/2}x = (n^{-1/2}x_i)_{i\geq 1}$.
First, since $|\varphi(g(n^{-1/2}x_i))| \geq 1$, then
\begin{align}
\label{eqn:ineqmn}
M^{n}_{g,n^{-1/2}x} \geq 1.
\end{align}
Let $r \geq 1$ and $f$ be a pair function in the class $\Lambda_r(\mathbb{R})$. We aim to estimate, with $Y$ a random variable distributed according to the free multiplicative semicircular law with parameter ${ \frac{1}{2}|g^2|^\prime (0)\sigma}$, 
\begin{align}
\label{equation: estimate- pi-g-x vs pi-h-y}
|\varphi(f(|\pi_n^{g,n^{-1/2}x}|)) -  \varphi(f(Y))|.
\end{align}
To use the estimates we have proved in Section \ref{sec:invarianceppl}, we must first write $Y = \exp(y_1^{(n)}(\sigma))\cdots\exp(y_n^{(n)}(\sigma))$ where the $y^{(n)}_i(\sigma),\, i\geq 1$ are distributed according to the logarithm of the free multiplicative semicircular law with parameter $p$: 
$$
y^{(n)}_i\big({ \frac{1}{2}|g^2|^\prime (0)\sigma}\big) = { |g^2|^\prime (0)\sigma} s_i + {\frac{1}{2}|g^2|^\prime (0)^2\sigma^2} w_i, \quad i\geq 1.
$$
Secondly, we \emph{truncate} $f$ and then \emph{regularise} this truncation as follows. 
We let $\theta\colon\mathbb{R}\to\mathbb{R}$ be a mollifier that is equal to $1$ on the interval $[-\frac{1}{2}, \frac{1}{2}]$ and $0$ outside of the interval $[-1,1]$. For any $0<\zeta \leq 1$ and $x\in\mathbb{R}$, we set $\theta_{\zeta}(x) = \theta(\zeta x)$. We denote by $f_{\zeta}$ the truncation of $f$ defined by $f_{\zeta}(x) = (f\theta_{\zeta})(x)$. Now to regularize $f_{\zeta}$, we pick a smooth pair convolution kernel $\rho$ in the Schwartz class of functions and set for any $\varepsilon>0$,  $\rho_{\varepsilon}(x)=\tfrac{1}{\varepsilon}\rho(\tfrac{x}{\varepsilon})$ for any $x\in\mathbb{R}$. As a result, to estimate \eqref{equation: estimate- pi-g-x vs pi-h-y}, we break our estimate into several parts:
\begin{align}
&\big|\varphi(f(|\pi_n^{g,n^{-1/2}x}|)) -  \varphi(f(|\pi_n^{\exp,y^{(n)}}|))\big| \nonumber \\
& \label{eqn:x-terms}\leq  \big|\varphi\big(f_{\zeta}(|\pi_n^{\exp,y^{(n)}}|)\big)-\varphi\big(f_{\zeta}\star\rho_{\varepsilon}(|\pi_n^{\exp,y^{(n)}}|)\big)\big|
 + \big|\varphi\big(f_{\zeta}(|\pi_n^{g,n^{-1/2}x}|)\big)-\varphi\big(f_{\zeta}\star\rho_{\varepsilon}(|\pi_n^{g,n^{-1/2}x}|)\big)\big|  \\
& \label{eqn:y-terms} + \big|\varphi\big(f(|\pi_n^{\exp,y^{(n)}}|)\big)-\varphi\big(f_{\zeta}(|\pi_n^{\exp,y^{(n)}}|)\big)\big| + \big|\varphi\big(f(|\pi_n^{g,n^{-1/2}x}|)\big)-\varphi\big(f_{\zeta}(|\pi_n^{g,n^{-1/2}x}|)\big)\big|
\\
& \label{eqn:x,y-term}+ \big|\varphi\big(f_{\zeta}\star\rho_{\varepsilon}(|\pi_n^{g,n^{-1/2}x}|)\big)-\varphi\big(f_{\zeta}\star\rho_{\varepsilon}(|\pi_n^{\exp,y^{(n)}}|)\big)\big|.
\end{align}
To estimate the term in \eqref{eqn:x,y-term}, we make use of the estimates in Section \ref{sec:invarianceppl}. For the terms appearing in \eqref{eqn:x-terms} and \eqref{eqn:y-terms}, we rely on the following two lemmas, which we prove in  Appendix \ref{sec:estimatefouriermoments}.
\begin{lemma}[Case $r=1$]
\label{lem:unboundedlipschitz} 
Let $f \in\Lambda_1(\mathbb{R})$ be a Lipschitz function with $\|f\|_1 \leq 1$. For any $0 < \varepsilon,\zeta < 1$, and $x\in \mathbb{R}$:\begin{align*}
%\label{eqn:freqspatcut}
|f(x) -f_{\zeta}(x)|  \leq \zeta |x|^2,\quad \text{ and } \quad |f_{\zeta}(x) - f_{\zeta}\star\rho_{\varepsilon}(x)| \leq \varepsilon(1+\zeta |x|)\int_{\mathbb{R}}|y|\rho(y)\mathrm{d}y
 .
\end{align*}
Moreover, the $k^{th}$ absolute moment of the Fourier transform satisfies
\[
I_k:= \int_{\mathbb{R}} |y|^k|\mathcal{F}{\rho}_\varepsilon(y)\mathcal{F}{f_{\zeta}}(y)| \mathrm{d} y \lesssim \begin{cases}
    \varepsilon^{-1} \zeta^{-2} & {\rm if} \quad k=1,\\
     \varepsilon^{-2} \zeta^{-2} & {\rm if} \quad k=2,\\
      \varepsilon^{-k} \zeta & {\rm if} \quad k\geq 3.\\
\end{cases} 
\]
\end{lemma}

\begin{lemma}[Case $r> 1$]
\label{lemma:UnbddZo} 
Let $f\in \Lambda_{r}(\mathbb{R})$ with $\|f^{\ell}\|_{\beta} \leq 1$. For any $x\in \mathbb{R}$ and $\zeta > 0$ and $\varepsilon > 0$,
\begin{align}\label{est:f-feta Unbdd Z}
|(f-f_{\zeta})(x)|\leq \zeta |x|^{r+1}\quad \text{and} \quad 
|f_{\zeta}(x)-f_{\zeta} \star \rho_{\varepsilon}(x)| \lesssim (1+|x|^r)\varepsilon^{r+1-\beta}.
\end{align}
Moreover, the $k^{th}$ absolute moment of the Fourier transform satisfies
\[
I_k  \lesssim 
\begin{cases}
{\zeta^{-(r+1)}\varepsilon^{-(k-r+1)}} & {\rm if}\quad k>r-1,\\
{\zeta^{-(r+1)}}{|\ln(\varepsilon)|} & {\rm if} \quad k=r-1, \\
{\zeta^{-(r+1)}} & {\rm if } \quad k<r-1.
\end{cases}
\]
\end{lemma}

%%%%%
%%%%%
\subsubsection*{\bf $1$-Wasserstein Case}
We now collect all the ingredients to prove Theorem~\ref{thm:mainthm} in the case $r = 1$. We begin by bounding the terms in~\eqref{eqn:x-terms} and~\eqref{eqn:y-terms} using Lemma~\ref{lem:unboundedlipschitz}, and obtain
\begin{align}
|\eqref{eqn:x-terms}| \leq \zeta (\| \pi_n^{g,n^{-1/2}x} \|_{L^{2}_{\varphi}} +  \| \pi_n^{\exp,y^{(n)}} \|_{L^{2}_{\varphi}})  \text{ and }   |\eqref{eqn:y-terms}| \lesssim \varepsilon(1+\zeta (\| \pi_n^{g,n^{-1/2}x}\|_{L^1_{\varphi}} + \| \pi_n^{\exp,y^{(n)}}\|_{L^1_{\varphi}})).
\end{align} It remains to bound the term in~\eqref{eqn:x,y-term}. Noting that $f_{\zeta}\star \rho_{\varepsilon}$ defines a smooth pair function, we proceed by writing
\begin{align}
&\varphi\big(f_{\zeta}\star \rho_{\varepsilon}(|\pi^{g,n^{-1/2}x}_{n}|)\big) - 
\varphi\big(f_{\zeta}\star \rho_{\varepsilon}(|\pi^{\exp,y^{(n)}_{n}}|)\big) \nonumber \\ 
&=  \int_{\mathbb{R}} \big(\cos(u | \pi_n^{g,n^{-1/2}x} |)-\cos(u | \pi_n^{\exp,y^{(n)}} |)\big)\mathcal{F}f_{\zeta}(u) \mathcal{F}\rho_{\varepsilon}(u){\rmd }u \nonumber \\
&=  \int_{\mathbb{R}} \Big({\rm tr}_2 \varphi^{(2)}(e^{iu H(\pi_n^{g,n^{-1/2}x})})-{\rm tr}_2 \varphi^{(2)}(e^{iu H(\pi_n^{\exp,y^{(n)}})})\Big)\mathcal{F}f_{\zeta}(u) \mathcal{F}\rho_{\varepsilon}(u){\rmd }u .\label{eqn:todominate}
\end{align}
We use the bounds established in Propositions~\ref{prop:firstorderterm}, \ref{prop:secondorderterm}, and~\ref{prop:thirdorderterm} when taking $y = y^{(n)}$, $h = \exp$, and incorporating the normalization by substituting $n^{-1/2}x$ in place of $x$. We start by noting that 
\begin{align*}
\big|\|y_{i}^{(n)}\|^2_{L^2_{\varphi}} - \|n^{-1/2}x_i\|^2_{L^{2}_{\varphi}}\big| =& O(n^{-2}),\\\big\|n^{-1/2}x_i\big\|^{2+\gamma}_{L_\varphi^{(2+\gamma)(1+\zeta)}} =& O(n^{-(1+\gamma/2)}),\\\big\|y^{(n)}_i\big\|^{2+\gamma}_{L_\varphi^{(2+\gamma)(1+\zeta)}} =& O(n^{-(1+\gamma/2)}),
\end{align*}
and, under the hypothesis of Theorem \ref{thm:mainthm}, for any $i \geq 1$ the quantities 
\begin{align*}
\big\|\|g''\|_{[0,x_i],\gamma} \| \big\|_{L_\varphi^{1+\zeta^{-1}}} \quad \text{and} \quad \big\|\|\exp\|_{[0,y_i^{(n)}],\gamma}\big\|_{L_\varphi^{1+\zeta^{-1}}}
\end{align*}
are bounded uniformly on $n$. Hence, we infer from Proposition \ref{prop:firstorderterm} and the estimate \eqref{eqn:ineqmn} that
\begin{align*}
|{\rm tr}_2(\varphi^{(2)}(L_{i,1}))| =  (\|\pi_{n}^{g,n^{-1/2}x}\|_{L^2_{\varphi}}+\|\pi_{n}^{\exp,y^{(n)}}\|_{L^2_{\varphi}})^2O(n^{-(1+\gamma/2)}).
\end{align*}
After summing over $1 \leq i \leq n$ and using Lemma \ref{lem:unboundedlipschitz} with $k=1$, we obtain that the order $1$ term contributing to \eqref{eqn:todominate} is $O(n^{-\gamma/2}\zeta^{-1}\varepsilon )$. We move on to bounding the contributions of the second term. We apply Proposition \ref{prop:secondorderterm} for $y = y^{(n)}$, $h = \exp$, and by incorporating the normalization by substituting $n^{-1/2}x$ in place of $x$, to obtain, under the assumptions of Theorem \ref{thm:mainthm}:
\begin{align*}
| {\rm tr}_2(\varphi^{(2)}( L_{i,2})) | & \leq (\| \pi_n^{g,n^{-1/2}x}\|_{L^4_{\varphi}} + \| \pi_n^{\exp,y^{(n)}}\|_{L^4_{\varphi}})^4(O(n^{-2}) + O(n^{-3/2}) + O(n^{-3/2+\gamma/2})) \\
& =  (\| \pi_n^{g,n^{-1/2}x}\|_{L^4_{\varphi}} + \| \pi_n^{g,y^{(n)}}\|_{L^4_{\varphi}} )^4(O(n^{-3/2})).
\end{align*}
After summing over $ 1\leq i \leq n$ and using Lemma \ref{lem:unboundedlipschitz} with $k=2$, we obtain that the order two term contributing to \eqref{eqn:todominate} is $(\| \pi_n^{g,n^{-1/2}x}\|_{L^4_{\varphi}} + \| \pi_n^{g,y^{(n)}}\|_{L^4_{\varphi}} )^4O(n^{-1/2}\varepsilon^{-2}\zeta^{-2})$. 
We finally bound the contribution to \eqref{eqn:todominate} of the order three term. By using Proposition \ref{prop:thirdorderterm} with the same parameters as for the first and second order terms, we obtain, under the assumption of Theorem \ref{thm:mainthm}:
\begin{align*}
|{\rm tr}_2(\varphi^{(2)}(L_{i,3}))| \leq (\|\pi_{n}^{g,n^{-1/2}x}\|_{L^{8}_{\varphi}} + \|\pi_{n}^{\exp,y^{(n)}}\|_{L^{8}_{\varphi}})^{6} O(n^{-3/2}).
\end{align*}
After summing over $1 \leq i \leq n$ and applying Lemma~\ref{lem:unboundedlipschitz} with $k = 3$, we obtain that the contribution of the third-order term to~\eqref{eqn:todominate} is
 \[ (\|\pi_{n}^{g,n^{-1/2}x}\|_{L^{8}_{\varphi}} + \|\pi_{n}^{\exp,y^{(n)}}\|_{L^{8}_{\varphi}})^{6} O(n^{-1/2}\zeta^{-2}\varepsilon^{-2}).\]
Collecting the leading-order terms, we now seek the asymptotic behavior in $n$ as $n \to \infty$, 
with $\varepsilon = \zeta = n^{-a}$ for some $a > 0$, of:
\begin{equation*}
\varepsilon + \zeta + n^{-\frac{\gamma}{2}}\frac{1}{\zeta^2\varepsilon} + n^{-\frac{1}{2}}\frac{1}{\zeta^2\varepsilon^2} + n^{-\frac{1}{2}}\frac{1}{\zeta\varepsilon^3}.
\end{equation*}
We seek to solve either $-\frac{1}{2} + 4a = -a $ or  $-\frac{\gamma}{2} + 3a = -a$ depending on the sign of $-\frac{1}{2} + 4a - (-\frac{\gamma}{2} + 3a)$.
This last expression is negative when $\gamma \leq 1 - 2a$. Solving $-\frac{\gamma}{2} + 3a = -a$ yields $\gamma /8= {a}$, which further implies $\gamma \leq \frac{4}{5}$. We conclude that whenever $\gamma \leq 4/5$, 
\begin{equation*}
\varepsilon + \zeta + n^{-\frac{\gamma}{2}}\frac{1}{\zeta^2\varepsilon} + n^{-\frac{1}{2}}\frac{1}{\zeta^2\varepsilon^2} + n^{-\frac{1}{2}}\frac{1}{\zeta\varepsilon^3} = O\Big(\frac{1}{n^{\frac{\gamma}{8}}}\Big).
\end{equation*}
Conversely, when $ \gamma > 4/5$, we have $-\frac{1}{2} + 4a > -\frac{\gamma}{2} + 3a$, and solving $-\frac{1}{2} + 5a = -a $ gives $a = {1}/{10}$. In this case, we obtain 
\begin{equation*}
\varepsilon + \zeta + n^{-\frac{\gamma}{2}}\frac{\ln(\varepsilon)}{\zeta^2} + n^{-\frac{1}{2}}\frac{1}{\zeta^2\varepsilon^2} + n^{-\frac{1}{2}}\frac{1}{\zeta\varepsilon^3} = O\Big(\frac{1}{n^{1/10}}\Big).
\end{equation*}
Finally, from the above computations, we deduce that whenever $\gamma \leq 4/5$ : 
$$
|\varphi(f(|\pi_n^{g,n^{-1/2}x}|)) -  \varphi(f(|\pi_n^{\exp,y^{(n)}}|))| = (1\vee (\| \pi_n^{g,n^{-1/2}x}\|_{L^{8}_{\varphi}}+\| \pi_n^{\exp,y^{(n)}}\|_{L^{8}_{\varphi}})^6)O(n^{-\gamma / 8}),
$$
and otherwise when $\gamma > 4/5$:
$$
|\varphi(f(|\pi_n^{g,n^{-1/2}x}|)) -  \varphi(f(|\pi_n^{\exp,y^{(n)}}|))| = (1\vee (\| \pi_n^{g,n^{-1/2}x}\|_{L^{8}_{\varphi}}+\| \pi_n^{\exp,y^{(n)}}\|_{L^{8}_{\varphi}})^6)O(n^{-1/10}).
$$
This finishes the proof of the rate of convergence in the $W_1$-distance.
\subsubsection*{\bf General Case} We now turn to the case $r > 1$, and follow the same approach used for the $r = 1$ case to determine the asymptotic order in the Zolotarev metric. Subsequently, we apply Rio's inequality~\eqref{inequ:Rio} to derive the convergence rate for the Wasserstein distance of order $r$.  

\vspace{0.3cm}
\noindent
{\bf Case $r=2$:}
Following an analogous approach to the proof of Theorem \ref{thm:mainthm} for $r=1$, we obtain an upper bound of the form
\begin{equation*} \varepsilon^2 + \zeta + n^{-\frac{\gamma}{2}}|\ln(\varepsilon)|\frac{1}{\zeta^3} + n^{-\frac{1}{2}}\frac{1}{\zeta^3\varepsilon} + n^{-\frac{1}{2}}\frac{1}{\zeta^3\varepsilon^2}. \end{equation*}
Setting $\varepsilon=n^{-a/2}$ and $\zeta = n^{-a}$, the first two terms in the above expression are of order $n^{-a}$. The remaining terms are of order $|\ln(n)|n^{-\gamma/2 + 3a}$, $n^{-1/2 + 7a/2}$ and $n^{-1/2+4a}$, respectively.
Since $-1/2+4a > -1/2 + 7a/2$, the sum of the last two terms is of order $n^{-1/2+4a}$.

If $-\gamma/2 + 3a > -1/2+4a $; that is, when $\gamma < 1-2a$, the first of these terms dominates. Solving the equation $a=-\frac{\gamma}{2} + 3a$ yields $a=\gamma/4$, which implies $\gamma < 2/3$.
When $\gamma \geq 1-2a$, the second term dominates. Solving $-1/2+4a = a$ gives $a=1/6$, which holds for $\gamma \geq 2/3$. Finally, applying inequality \eqref{inequ:Rio} yields the rate of convergence for the Wasserstein distance $W_2$. 
Now, let us consider the general case.

\vspace{0.3cm}
\noindent
{\bf Case: $1< r< 2$:} Setting $\zeta=n^{-a}$ and $\varepsilon=n^{-a/r}$, we obtain the expression  
\begin{align*}
&\zeta + \varepsilon^{r}+ \frac{n^{-\gamma/2}}{\zeta^{r+1}\varepsilon^{2-r}} + \frac{n^{-1/2}}{\zeta^{r+1}\varepsilon^{3-r}} + \frac{n^{-1/2}}{\zeta^{r+1} \varepsilon^{4-r}} 
\\&\sim n^{-a} + n^{-\gamma/2 + (1+r)a + ((2-r)a)/r} + n^{-1/2 + (1+r)a + ((4-r)a)/r}.
\end{align*}
Observe that the inequality
$$
-\frac{\gamma}{2} + (1+r)a + \frac{(2-r)a}{r} > -\frac{1}{2} + (1+r)a + \frac{(4-r)a}{r} $$ 
implies 
$\gamma <1-4a/r.$
To determine ${\rm rate}_-$, we solve $$-a = -\frac{\gamma}{2} + (1+r)a + \frac{(2-r)a}{r},\ \text{ which yields }\ a=\frac{\gamma}{2r+2+4/r}.$$ 
This leads to the critical value
$$\gamma_{\rm crit} = \frac{r^2+r+2}{r^2+r+4},\ \text{ and }\ {\rm rate}_-=\frac{\gamma}{2r^2+2r+4}.$$
To compute ${\rm rate}_+$, we solve
$$-a = -\frac{1}{2} + (1+r)a + \frac{(4-r)a}{r},\ \text{ which results in }\ a=\frac{1}{2r+2+8/r}.$$ Hence, we conclude ${\rm rate}_+=1/(2r^2+2r+8).$

\vspace{0.3cm}
\noindent
{\bf Case: ${2}< r< 3$:} We set $\zeta=n^{-a}$ and $\varepsilon=n^{-a/r}$, then the expression becomes
\begin{align*}
 &\zeta + \varepsilon^{r} + \frac{n^{-\gamma/2}}{\zeta^{r+1}} + \frac{n^{-1/2}}{\zeta^{r+1}\varepsilon^{3-r}} + \frac{n^{-1/2}}{\zeta^{r+1}\varepsilon^{4-r}}\sim n^{-a} + n^{-\gamma/2 + (1+r)a} + n^{-1/2 + (1+r)a + ((4-r)a)/r}.
\end{align*}
To compare terms, observe that the inequality 
$$-\frac{\gamma}{2} + (1+r)a > -\frac{1}{2} + (1+r)a + \frac{(4-r)a}{r}$$ holds when
$\gamma<1-8a/r+2a$. 
To determine ${\rm rate}_-$, solving 
$$-a = -\frac{\gamma}{2} + (1+r)a \ \text{ yields }\ a=\frac{\gamma}{2r+4}.$$ From this, we obtain the critical value $$\gamma_{\rm crit}=\frac{r^2+2r}{r^2+r+4},\ \text{ and }\ {\rm rate}_-=\frac{\gamma}{2r^2+4r}.$$
For ${\rm rate}_+$, we solve 
$$ -a = -\frac{1}{2} + (1+r)a + \frac{(4-r)a}{r}\ \text{ leads to }\ a=\frac{1}{2(1+r+4/r)}.$$ Thus, we conclude 
$
{\rm rate}_+ = 1/(2r^2+2r+8).
$

\vspace{0.3cm}
\noindent
{\bf Case: $r = 3$:} We set $\zeta=n^{-a}$ and $\varepsilon=n^{-a/r}$, leading to the expression
\begin{align*}
&\zeta + \varepsilon^{3} + \frac{n^{-\gamma/2}}{\zeta^{4}} + \frac{n^{-1/2}|\ln(\varepsilon)|}{\zeta^{4}} + \frac{n^{-1/2}}{\zeta^{4}\varepsilon} 
\sim n^{-a} + n^{-\gamma/2 + 4a} + n^{-1/2 + 4a + a/3}.
\end{align*}
Notice that $-\gamma/2 + 4a > -1/2 + 4a + a/3$ is satisfied when $\gamma <1-2a/3.$ 

To determine ${\rm rate}_-$, we solve $ -a = \gamma/2 + 4a$, which gives $ a = \gamma/10$. Therefore, we find
$$\gamma_{\rm crit} = \frac{15}{16},\ \text{ and }\ {\rm rate}_-=\frac{1}{3}\frac{\gamma}{10}=\frac{\gamma}{30}.$$ 
For ${\rm rate}_+$, solving 
$$-a = -\frac{1}{2} + 4a + \frac{a}{3} \ \text{ yields }\ 
a=\frac{3}{32}.$$ 
Hence,
$
{\rm rate}_+ = 1/3\cdot 3/32=1/32.
$

\vspace{0.3cm}
\noindent
{\bf Case: $ 3 < r < 4$:} Set $\zeta=n^{-a}$ and $\varepsilon=n^{-a/r}$, which simplifies the expression to
\begin{equation*}
 \zeta + \varepsilon^{r}+ \frac{n^{-\gamma/2}} {\zeta^{r+1}} + \frac{n^{-1/2}}{\zeta^{r+1}} + \frac{n^{-1/2}}{\zeta^{r+1}\varepsilon^{4-r}}\sim n^{-a} + n^{a(r+1)-\gamma/2}+n^{a(r+1)-1/2+a(4-r/r)}.
\end{equation*}
To compare terms, we note that
$$a(r+1)-\frac{\gamma}{2} > a(r+1)-\frac{1}{2}+\frac{a(4-r)}{r}\ \text{ holds if and only if }\ \gamma <1-2(1-\frac{4}{r})a.$$ 
Solving $-a=a(r+1)-\gamma/2$ yields $a=\gamma/(2r+4)$. Therefore, we obtain
$$
\gamma_{\rm crit}=\frac{r+2}{r+1+4/r}=\frac{r^2+2r}{r^2+r+4},\ \text{ and }\ {\rm rate}_-=\frac{1}{r}\frac{\gamma}{2r+4}=\frac{\gamma}{2r^2+4r}. 
$$
For ${\rm rate}_+$, solving 
$$-a=a(r+1)-\frac{1}{2}+\frac{a(4-r)}{r}\ \text{ gives }\ a=\frac{1}{2r+2+8/r}.$$ 
Thus, the rate is ${\rm rate}_+=1/(2r^2+2r+8).$

\vspace{0.3cm}
\noindent
{\bf Case: $r = 4$:} We set $\zeta=n^{-a}$ and $\varepsilon=n^{-a/4}$, leading to  
\begin{equation*}
    \varepsilon^{4} + \zeta + \frac{n^{-\gamma/2}}{\zeta^{5}} + \frac{n^{-1/2}}{\zeta^{5}} + \frac{n^{-1/2}|\ln(\varepsilon)|}{\zeta^{5}}\sim n^{-a}+n^{-\gamma/2+5a}+n^{-1/2+5a}|\ln(\varepsilon)|. 
\end{equation*}
Observing that $-\gamma/2+5a> -1/2+5a$ holds if and only if $\gamma <1$. Thus the critical value is $\gamma_{\rm crit}=1$. Solving $-a=-\gamma/2+5a$ yields $a=\gamma/12$, which gives ${\rm rate}_-=\gamma/48$. Similarly, solving $-a=-1/2+5a$ yields $a=1/12$. Hence, ${\rm rate}_+=1/48.$

\vspace{0.3cm}
\noindent
{\bf Case: $r > 4$:} Since the expression
\begin{equation*}
    \varepsilon^{r} + \zeta + \frac{n^{-\gamma/2}}{\zeta^{5}} + \frac{n^{-1/2}}{\zeta^{5}} + \frac{n^{-1/2}}{\zeta^{5}}
\end{equation*}
follows the same pattern as the case $r=4$, the results remain unchanged. Thus, for brevity, we omit the detailed computations. 
\appendix
\label{annex:Appendix}
\section{Complementary material }
\subsection{Free cumulants of ampliations}
\begin{proposition}\label{prop:cumulantsampliation}
Suppose $(\cA,\varphi)$ is a tracial probability space.
Given an integer $m\geq 1$, let $x_1,\ldots,x_m$ be elements in $\mathcal{A}$ or affiliated to $\mathcal{A}$ with enough moments for $\kappa_m(x_1,\ldots,x_m)$ to be well-defined. Define for each $1 \leq i \leq m$, 
\begin{align*}
\bm{x}_i = \begin{bmatrix} x_i & 0 \\ 0 & x^{\star}_i \end{bmatrix} 
\end{align*}
The $m^{th}$ free $\star$-cumulants of $\bm{x}_1,\ldots,\bm{x}_m$ with respect to the state $\varphi^{(2)}$ satisfies, with $\varepsilon_i \in \{ 0,1\}$:
\begin{align*}
\kappa_m^{\varphi^{(2)}}({\bm x}_1J^{\varepsilon_1},\ldots,{\bm x}_mJ^{\varepsilon_m})&=\begin{bmatrix}\kappa_m^{\varphi}\big({x}_1,{x}_2^{(\iota_1)},\ldots,  {x}_m^{(\iota_{m-1})}\big) & 0 \\ 0 & \kappa_m^{\varphi} \big({x}_1^{\star}, {x}_2^{(\iota_1+1)},\ldots, {x}_m^{(\iota_{m-1}+1)}\big)
\end{bmatrix} J^{\iota_m}
\\&:= \bm{\kappa}^{\varphi,\,\bm{x}}_m(\varepsilon)J^{\iota_m},
\end{align*}
where  $\iota_i = \varepsilon_1 + \dots + \varepsilon_{i}$ and the following conventions, for any $l \in \mathbb{N}$, ${ x}^{(2l)}= { x}, {x}^{(2l+1)} = { x}^{\star}$, and $\varepsilon=(\varepsilon_1,\dots,\varepsilon_m)$. 
\end{proposition}
\begin{proof}
We prove by induction that for all $m \geq 1$:
\begin{equation}\label{eqn:the-matrix-xe^uxJ}
{\bm x}_1J^{\varepsilon_1}\cdots{\bm x}_mJ^{\varepsilon_m}
= \begin{bmatrix}
x_1{ x}^{(\iota_1)}_{2}\cdots x_m^{(\iota_{m-1})} & 0 \\
0 & x_1^{\star}x_2^{(\iota_1+1)}\cdots x_m^{(\iota_{m-1}+1)}
\end{bmatrix}J^{\iota_m}.
\end{equation}
Observe that \eqref{eqn:the-matrix-xe^uxJ} holds when $m=1$. Now we assume \eqref{eqn:the-matrix-xe^uxJ} holds for $m=n$, and let us consider 
${\bm x}_1J^{\varepsilon_1}\cdots{\bm x}_{n+1}J^{\varepsilon_{n+1}}$.
Our assumption implies  that ${\bm x}_1J^{\varepsilon_1}\cdots{\bm x}_{n}J^{\varepsilon_{n}}{\bm x}_{n+1}J^{\varepsilon_{n+1}}$ is equal to
\begin{align*}
&\begin{bmatrix}
x_1{ x}^{(\iota_1)}_{2}\cdots x_n^{(\iota_{n-1})} & 0 \\
0 & x_1^{\star}x_2^{(\iota_1+1)}\cdots x_n^{(\iota_{n-1}+1)}
\end{bmatrix} \times J^{\iota_n}\begin{bmatrix}
    x_{n+1} & 0 \\ 0 & x^{\star}_{n+1}
\end{bmatrix}J^{\varepsilon_{n+1}}.
\end{align*}
We first consider the case $J^{\varepsilon_{n+1}}=1$ ($\varepsilon_{n+1}$ = 0). If $\iota_{n}$ is even, $\iota_{n+1}=\iota_{n}+\varepsilon_{n+1}$ is even and $J^{\iota_n}=J^{\iota_{n+1}}=1$. Besides $x=x^{(\iota_n)}$ and $x^{\star}=x^{(\iota_n+1)}$. We obtain  
\begin{align*}
{\bm x}_1J^{\varepsilon_1}\cdots{\bm x}_{n}J^{\varepsilon_{n}}{\bm x}_{n+1}J^{\varepsilon_{n+1}}
&= 
\begin{bmatrix}
x_1{ x}^{(\iota_1)}_{2}\cdots x_n^{(\iota_{n-1})}x_{n+1} & 0 \\
0 & x_1^{\star}x_2^{(\iota_1+1)}\cdots x_n^{(\iota_{n-1}+1)}x^{\star}_{n+1}
\end{bmatrix} \\
&= \begin{bmatrix}
x_1{ x}^{(\iota_1)}_{2}\cdots x_{n+1}^{(\iota_{n})} & 0 \\
0 & x_1^{\star}x_2^{(\iota_1+1)}\cdots x_{n+1}^{(\iota_{n}+1)}
\end{bmatrix}J^{\iota_{n+1}},
\end{align*}
which implies that \eqref{eqn:the-matrix-xe^uxJ} also holds for $m=n+1$. In the case $\iota_n$ is odd, $\iota_{n+1}$ is also odd. Thus, $J^{\iota_n}=J^{\iota_{n+1}}=J$ and  $x^{\star}=x^{(\iota_n)}$ and $x=x^{(\iota_n+1)}$. 
Hence
\begin{align*}
{\bm x}_1J^{\varepsilon_1}\cdots{\bm x}_{n}J^{\varepsilon_{n}}{\bm x}_{n+1}J^{\varepsilon_{n+1}}
&= \begin{bmatrix}
x_1{ x}^{(\iota_1)}_{2}\cdots x_n^{(\iota_{n-1})} & 0 \\
0 & x_1^{\star}x_2^{(\iota_1+1)}\cdots x_n^{(\iota_{n-1}+1)}
\end{bmatrix} \begin{bmatrix}
0 & x^{\star}_{n+1}  \\ x_{n+1} & 0
\end{bmatrix}\\
&= \begin{bmatrix}
x_1{ x}^{(\iota_1)}_{2}\cdots x^{(\iota_{n})}_{n+1} & 0 \\
0 & x_1^{\star}x_2^{(\iota_1+1)}\cdots x_{n+1}^{(\iota_{n}+1)}
\end{bmatrix}J^{\iota_{n+1}},
\end{align*}
implies that \eqref{eqn:the-matrix-xe^uxJ} holds for $m=n+1$. The proof for the case  $\varepsilon_{n+1}=1$ is similar, and we leave it to the reader. By induction, we have proved that \eqref{eqn:the-matrix-xe^uxJ} holds for all $m\geq 1$. 
We apply $\varphi^{(2)}$ to both sides of equation \eqref{eqn:the-matrix-xe^uxJ} and obtain
\begin{align*}
\varphi^{(2)}({\bm x}_1J^{\epsilon_1}\cdots{\bm x}_mJ^{\varepsilon_m})
&= \begin{bmatrix}
\varphi(x_1{x}^{(\iota_1)}_{2}\cdots x_m^{(\iota_{m-1})}) & 0 \\
0 & \varphi(x_1^{\star}x_2^{(\iota_1+1)}\cdots x_m^{(\iota_{m-1}+1)})
\end{bmatrix}J^{\iota_m} := \varphi^{(2), {\bm x}}(\varepsilon)J^{\iota_m}.
\end{align*}
Next, given $p,q \geq 0$,
\begin{multline*}
\varphi^{(2)}({\bm x}_1J^{\varepsilon_1}, \ldots, {\bm x}_{p-1}J^{\varepsilon_{p-1}}\varphi^{(2)}\big({\bm x}_{p}J^{\varepsilon_{p}}, \ldots, {\bm x}_{p+q-1}J^{\varepsilon_{p+q-1}}\big),{\bm x}_{p+q}J^{\varepsilon_{p+q}},\ldots,{\bm x}_{m}J^{\varepsilon_{m}}) \\
=\varphi^{(2)}({\bm x}_1J^{\varepsilon_1}, \ldots, {\bm x}_{p-1}J^{\varepsilon_{p-1}}\varphi^{(2),\,\bm{x}_{[p,p+q-1]}}(\varepsilon_{[p,p+q-1]})J^{\varepsilon_{p} + \cdots + \varepsilon_{p+q-1} },{\bm x}_{p+q}J^{\varepsilon_{p+q}},\ldots,{\bm x}_{m}J^{\varepsilon_{m}}),
\end{multline*}
where we denote by $\bm{x}_{[p,p+q-1]}=(x_p, \dots, x_{p+q-1})$ and $\varepsilon_{[p,p+q-1]}=(\varepsilon_p, \dots, \varepsilon_{p+q-1})$. Upon noticing that 
\begin{align*}
&J^{\varepsilon_{p-1}} \varphi^{(2),\,\bm{x}_{[p,p+q-1]}}(\varepsilon_{[p,p+q-1]})J^{\varepsilon_{p-1}+\dots + \varepsilon_{p+q-1}}  \\ &=\begin{bmatrix}
\varphi \big(x^{\varepsilon_{p-1}}_1{x}^{(\varepsilon_{p-1}+\varepsilon_p)}_{2}\cdots x_{p+q-1}^{(\varepsilon_{p-1}+\varepsilon_p+\cdots+\varepsilon_{p+q-2})}\big) & 0 \\
0 & \varphi\big(x_1^{\varepsilon_{p-1}+1}x_2^{(\varepsilon_{p-1}+\varepsilon_p+1)}\cdots x_m^{(\varepsilon_{p-1}+\varepsilon_p+\cdots+\varepsilon_{p+q-2}+1)}\big)
\end{bmatrix}.
\end{align*}
The moment-cumulant formula yields the desired result. 
\end{proof}
\begin{lemma}\label{lem: kappa_pi estimate}
Suppose $(\cA,\varphi)$ is a tracial probability space, and $p\geq 0$. Let $X=g(x)$ for some self-adjoint random variable $x\in \cA$ and $g\in \C^{2,\gamma}(\mathbb{R})$. Then for any $\epsilon_i\in \{1,\star\}$ and $\pi\in NC(p)$, we have  
\begin{align*}
\big| \kappa_{\pi}^{\varphi}\big( X^{\epsilon_1}-1,\dots,X^{\epsilon_p}-1 \big) \big|\lesssim 
\big\| \|g\|_{1,[x^-,x^+]} \big\|_{L^{p(1+\zeta^{-1})}_\varphi}^p \|x\|_{L^{p(1+\zeta)}_\varphi}^p,
\end{align*}
where $x^-:=\min(x,0)$, $x^+:=\max(x,0)$.
\end{lemma}
\begin{proof}
We first observe by the moment-cumulant formula that
\begin{align}\label{eqn:estimate for kappa_pi}
\big| \kappa_{\pi}^{\varphi}(X^{\epsilon_1}-1,\dots,X^{\epsilon_p}-1) \big| \leq \sum_{\sigma\in NC(p);\sigma\leq \pi}|\mu(\sigma,\pi)| |\varphi_{\sigma}(X^{\epsilon_1}-1,\dots,X^{\epsilon_p}-1)|. 
\end{align}
Then for any $1\leq i_1< \dots <i_k \leq p$, we note that 
\begin{align*}%\label{eqn:estimate for moment of |x-1|}
|\varphi((X^{\epsilon_{i_1}}-1)\dots (X^{\epsilon_{i_k}}-1))|&\leq \varphi(|X-1|^k)
= \int_{\mathbb{R}} |g(t)-1|^k\nu_{x}(\rmd t)\\ &\leq \int_{\mathbb{R}} \|g\|_{1,[t^-,t^+]}^k|t|^k \nu_{x} 
(\rmd t)  
\\&\leq \big\| \|g\|_{1,[x^-,x^+]} \big\|_{L^{k(1+\zeta^{-1})}_\varphi}^k \varphi(|x|^{k(1+\zeta)})^{\frac{1}{1+\zeta}} \\&= \big\| \|g\|_{1,[x^-,x^+]} \big\|_{L^{k(1+\zeta^{-1})}_\varphi}^k \|x\|_{L^{k(1+\zeta)}_\varphi}^k,
\end{align*}
where the first and last inequalities follow from H\"older's inequality and the second one from the fact that $g \in \C^{2,\gamma}(\mathbb{R})$ and is hence Lipschitz. Since $L^p$-norms are increasing in $p$ and the fact that $\sum_{V \in \sigma}|V|=p$, we infer 
\begin{align}\label{eqn:estimate for varphi_sigma}
 |\varphi_{\sigma}(X^{\epsilon_1}-1,\dots,X^{\epsilon_p}-1)|\leq  \big\| \|g\|_{1,[x^-,x^+]} \big\|_{L^{p(1+\zeta^{-1})}_\varphi}^p \|x\|_{L^{p(1+\zeta)}_\varphi}^p \text{ for any }\sigma\in NC(p).
\end{align}
By combining \eqref{eqn:estimate for kappa_pi} with \eqref{eqn:estimate for varphi_sigma}, we obtain the desired result.
\end{proof}

\begin{corollary}\label{remark: 1st 2rd moment estimate }
Note that by applying Proposition \ref{prop:cumulantsampliation} and Lemma \ref{lem: kappa_pi estimate}, for any $\pi_1\in NC(2)$ and $\pi_2\in NC(4)$, we obtain the following estimates: 
\begin{align*}
\Big\| \varphi^{(2)}\big( \bm{X}-I_2 \big)\Big\|_{\mathcal{M}_2(\mathbb{C})}&\leq  \big\| \|g\|_{1,[x^-,x^+]}\big\|_{L^{1+\zeta^{-1}}_\varphi} \|x\|_{L^{1+\zeta}_\varphi}, \\
\Big\| \kappa_{\pi_1}^{\varphi^{(2)}}\big( \bm{X}-I_2, \bm{X}^{\star}-I_2 \big) \Big\|_{\mathcal{M}_2(\mathbb{C})} &\lesssim \big\| \|g\|_{1,[x^-,x^+]}
\big\|_{L^{2(1+\zeta^{-1})}_\varphi}^2 \|x\|_{L^{2(1+\zeta)}_\varphi}^2, \\
\Big\| \kappa_{\pi_2}^{\varphi^{(2)}}\big( (\bm{X}-I_2)J, (\bm{X}-I_2), (\bm{X}^{\star}-I_2), J(\bm{X}^{\star}-I_2) \big) \Big\|_{\mathcal{M}_2(\mathbb{C})} &\lesssim \big\| \|g\|_{1,[x^-,x^+]} \big\|_{L^{4(1+\zeta^{-1})}_\varphi}^4 \|x\|_{L^{4(1+\zeta)}_\varphi}^4.
\end{align*}
\end{corollary}
\begin{lemma}\label{Lemma:moment-est}
Let $(\cA,\varphi)$ be a tracial probability space and $x,y\in \cA$ be self-adjoint freely independent centered random variables. Consider two functions $g,h\in \C^{2,\gamma}(\mathbb{R})$ such that $g(0)=h(0)=1$. Then for $\bm{X}=diag(X,X^{\star})$ and $\bm{Y}=diag(Y,Y^{\star})$ with $X=g(x)$ and $Y=h(y)$, we have
\begin{align*}
&\big\|\varphi^{(2)}\big( \bm{X}-I_2  \big)-{\varphi^{(2)}}\big( \bm{Y}-I_2  \big) \big\|_{\mathcal{M}_2(\mathbb{C})} \\&
 \lesssim  2\big(\big\| \|  g^{\prime\prime} \|_{\gamma,[x^-,x^+]}\big\|_{L^{1+\zeta^{-1}}_\varphi} + \big\| \|  h^{\prime\prime} \|_{\gamma,[y^-,y^+]} \big\|_{L^{1+\zeta^{-1}}_\varphi} \big)\big(
  \|x\|_{L^{(2+\gamma)(1+\zeta)}_\varphi}^{{2+\gamma}}+ \|y\|_{L^{(2+\gamma)(1+\zeta)}_\varphi}^{{2+\gamma}}\big)
 \\& \qquad+ \Big| g^{\prime\prime}(0)\| x\|_{L^2_\varphi}^2 - h^{\prime\prime}(0)\|y\|_{L^2_\varphi}^2\Big|.
\end{align*}
Moreover, for any $\epsilon_1, \epsilon_2 \in \{1,\star\}$, we have 
\begin{align*}
\| &\varphi^{(2)}(\bm{X}^{\epsilon_1} -I_2) \varphi^{(2)}(\bm{X}^{\epsilon_2} -I_2) - \varphi^{(2)}(\bm{Y}^{\epsilon_1} -I_2)\varphi^{(2)}(\bm{Y}^{\epsilon_2} -I_2)\|_{M_2(\mathbb{C})}
\\ & \lesssim \Big(\big\| \|g\|_{1,[x^-,x^+]}\big\|_{L^{1+\zeta^{-1}}_\varphi} \|x\|_{L^{1+\zeta}_\varphi}+\big\| \|h\|_{1,[y^-,y^+]} \big\|_{L^{1+\zeta^{-1}}_\varphi} \|y\|_{L^{1+\zeta}_\varphi}\Big)
\\& \quad \times \Big(\big\| \|  g^{\prime\prime} \|_{\gamma,[x^-,x^+]} \big\|_{L^{1+\zeta^{-1}}_\varphi} \|x\|_{L^{(2+\gamma)(1+\zeta)}_\varphi}^{{2+\gamma}} +  \big\| \|  h^{\prime\prime} \|_{\gamma,[y^-,y^+]} \big\|_{L^{1+\zeta^{-1}}_\varphi} \|y\|_{L^{(2+\gamma)(1+\zeta)}_\varphi}^{{2+\gamma}}
\\&\quad + \Big| g^{\prime\prime}(0) \| x\|_{L^2_\varphi}^2 - h^{\prime\prime}(0)\|y\|_{L^2_\varphi}^2\Big|\Big),
\end{align*}
and
\begin{multline*}
\big\|  {\kappa}^{\varphi^{(2)}}_2 \big(\bm{X}^{\epsilon_1} -I_2, \bm{X}^{\epsilon_2} -I_2\big) - {\kappa}^{\varphi^{(2)}}_2 \big(\bm{Y}^{\epsilon_1} -I_2 ,  \bm{Y}^{\epsilon_2} -I_2\big)\big\|_{M_2(\mathbb{C})} \\
\lesssim \big(\|x\|_{L^{6(1+\zeta)}_{\varphi}}^{3+\gamma}\vee \|y\|_{L^{6(1+\zeta)}_{\varphi}}^{3+\gamma}\big)P  + \|x\|_{L^3_{\varphi}}^3 + \|y\|_{L^3_{\varphi}}^3+  \big|
g^{\prime}(0)^{\epsilon_1}g^{\prime}(0)^{\epsilon_2}\|x\|_{L^2_{\varphi}}^2 - h^{\prime}(0)^{\epsilon_1} h^{\prime}(0)^{\epsilon_2}\|y\|_{L^2_{\varphi}}^2\big|,
\end{multline*}
with  $P$  being a noncommutative polynomial with constant term in  $\big\| \|  g\|_{1,[x^-,x^+]} \big\|_{L^{1+\zeta^{-1}}_\varphi}$, $ \big\| \|  h \|_{1,[y^-,y^+]} \big\|_{L^{1+\zeta^{-1}}_\varphi}$, $\big\| \|  g^{\prime\prime} \|_{\gamma,[x^-,x^+]} \big\|_{L^{2(1+\zeta^{-1})}_\varphi} $, and $ \big\| \|  h^{\prime\prime} \|_{\gamma,[y^-,y^+]} \big\|_{L^{2(1+\zeta^{-1})}_\varphi}$. 

\end{lemma}
\begin{proof} We first note that $\big\| {\varphi^{(2)}}\big( \bm{X}-I_2  \big)- {\varphi^{(2)}}\big( \bm{Y}-I_2  \big)\big\|_{\mathcal{M}_2(\mathbb{C})} = \big| \varphi\big( X-1  \big)- \varphi\big( {Y}-1  \big)\big|$, where we have used for $\bm{Y}$ the same conventions above relative to $\bm{X}$. Using the fact that $\varphi({x})=\varphi({y}) ={0}$, we obtain

\begin{align*}
\varphi&\big( {X}-1  \big)-{\varphi}\big( {Y}-1  \big)
\\&=\int_0^1\Big(\varphi(g^{\prime}(u{x}){x})
-\varphi (h'(u{y}){y}\big)\Big){\rmd u} \\
&=\int_0^1\Big(\varphi\big(\big(g^{\prime}(u {x})-g^{\prime}(0)\big) {x}\big) -\varphi \big(\big(h'(u {y})-h^{\prime}(0)) {y}\big)\Big){\rmd u} \\
&=\int_0^1 \Big(\varphi \big(\big(g^{\prime}(u {x})- g^{\prime}(0)- {g^{\prime\prime}(0)}u {x}\big) {x}\big) -\varphi  \big(\big(h'(u {y})- h^{\prime}(0) -h"({  0})u{ y}\big) {y}\big)\Big) \rmd u\\
&\hspace{4cm}-\int_0^1u \big(h"({  0})\varphi({  y}^2) - g"({  0})\varphi({ x}^2)\big) {\rmd u} .
\end{align*} 
Then we observe that for some $\xi$ between $0$ and $u$,
\begin{align*}
\Big|\varphi\Big(\big(g^{\prime}\big(u {x}\big)-g^{\prime}(0)-g^{\prime\prime}(0)u{x}\big)x\Big)\Big| 
&\leq \int_{\mathbb{R}} \big|g^{\prime}(ut)-g^{\prime}(0)-g^{\prime\prime}(0){ut}\big|\ \big|{t}\big| \nu_{x}(\rmd t ) \\
&\leq \int_{\mathbb{R}} \big|\big(g^{\prime\prime}(\xi t)- g^{\prime\prime}(0)\big)ut\big|\ \big|{t}\big| \nu_{x}(\rmd t ) \\
&\leq \int_{\mathbb{R}} \|g^{\prime\prime}\|_{\gamma,[(ut)^-,(ut)^+]} |t|^{2+\gamma} \nu_{x} (\rmd t).
\end{align*}
Here, the second inequality follows from  Taylor-Lagrange theorem while the last one follows from the fact that $g^{\prime\prime}$ is $\gamma$-H\"older continuous. Finally,  another application of H\"older's inequality yields for any $\zeta \geq 0$: 
$$
\Big|\varphi\Big(\big(g^{\prime}\big(u {x}\big)-g^{\prime}(0)-g^{\prime\prime}(0)u{x}\big){x}\Big)\Big|
\leq \big\| \|  g^{\prime\prime} \|_{\gamma,[x^-,x^+]} \big\|_{L^{1+\zeta^{-1}}_\varphi} \ \|x\|_{L^{(2+\gamma)(1+\zeta)}_\varphi}^{{2+\gamma}}.
$$
By similarly treating the term corresponding to $Y$, we finally deduce the desired bound. 
Letting  $\epsilon_1, \epsilon_2 \in \{1,\star\}$, we move now to estimating the terms 
\begin{align*}
\| &\varphi^{(2)}(\bm{X}^{\epsilon_1} -I_2) \varphi^{(2)}(\bm{X}^{\epsilon_2} -I_2) - \varphi^{(2)}(\bm{Y}^{\epsilon_1} -I_2)\varphi^{(2)}(\bm{Y}^{\epsilon_2} -I_2)\|_{M_2(\mathbb{C})}
\\&= |\varphi ( {X}^{\epsilon_1} -1)\varphi ({X}^{\epsilon_2} -1) - \varphi ( {Y}^{\epsilon_1} -1)\varphi ( {Y}^{\epsilon_2} -1)|  
\\&  =\big|\varphi ( {X}^{\epsilon_1} -1)\big[ \varphi ( {X}^{\epsilon_2} -1) - \varphi ( {Y}^{\epsilon_2} -1)\big]+ \varphi ( {Y}^{\epsilon_2} -1)\big[ \varphi ( {X}^{\epsilon_1} -1) - \varphi ( {Y}^{\epsilon_1} -1)\big]\big|
\\& \leq \big( |\varphi ( {X} -1)| + |\varphi ( {Y} -1)| \big) \big| \varphi ( {X} -1) - \varphi ( {Y} -1)\big|.
\end{align*}
Putting together the estimates of the above treated terms, we obtain
\begin{align*}
\| \varphi^{(2)}&(\bm{X}^{\epsilon_1} -I_2) \varphi^{(2)}(\bm{X}^{\epsilon_2} -I_2) - \varphi^{(2)}(\bm{Y}^{\epsilon_1} -I_2)\varphi^{(2)}(\bm{Y}^{\epsilon_2} -I_2)\|_{M_2(\mathbb{C})}
\\ & \leq \Big(\big\| \|g\|_{1,[x^-,x^+]} \big\|_{L^{1+\zeta^{-1}}_\varphi} \|x\|_{L^{1+\zeta}_\varphi}+\big\| \|h\|_{1,[y^-,y^+]} \big\|_{L^{1+\zeta^{-1}}_\varphi} \|y\|_{L^{1+\zeta}_\varphi}\Big)
\\& \qquad \times \Big(\big\| \|  g^{\prime\prime} \|_{\gamma,[x^-,x^+]} \big\|_{L^{1+\zeta^{-1}}_\varphi} \|x\|_{L^{(2+\gamma)(1+\zeta)}_\varphi}^{{2+\gamma}} +  \big\| \|  h^{\prime\prime} \|_{\gamma,[y^-,y^+]} \big\|_{L^{1+\zeta^{-1}}_\varphi} \|y\|_{L^{(2+\gamma)(1+\zeta)}_\varphi}^{{2+\gamma}}
\\&\qquad + \Big| g^{\prime\prime}(0) \| x\|_{L^2_\varphi}^2 -h^{\prime\prime}(0)\|y\|_{L^2_\varphi}^2\Big|\Big).
\end{align*}
As for the term
\begin{multline*}
\big\|  {\kappa}^{\varphi^{(2)}}_2 \big(\bm{X}^{\epsilon_1} -I_2, \bm{X}^{\epsilon_2} -I_2\big) - {\kappa}^{\varphi^{(2)}}_2 \big(\bm{Y}^{\epsilon_1} -I_2 ,  \bm{Y}^{\epsilon_2} -I_2\big)\big\|_{M_2(\mathbb{C})}
\\= | {\kappa}^{\varphi }_2\big( {X}^{\epsilon_1} -1,  {X}^{\epsilon_2} -1\big) - {\kappa}^{\varphi }_2 \big( {Y}^{\epsilon_1} -1 ,   {Y}^{\epsilon_2} -1\big) \big|,
\end{multline*}
we write 
\begin{align*}
{\kappa}^{\varphi }_2&\big( {X}^{\epsilon_1} -1,  {X}^{\epsilon_2} -1\big) - {\kappa}^{\varphi }_2 \big( {Y}^{\epsilon_1} -1 ,   {Y}^{\epsilon_2} -1\big)
\\& =\varphi \big( ({X}^{\epsilon_1} -1)(  {X}^{\epsilon_2} -1)\big) - \varphi \big( ({Y}^{\epsilon_1} -1)(  {Y}^{\epsilon_2} -1)\big) 
\\ & \qquad - \Big( \varphi ( {X}^{\epsilon_1} -1)\varphi ({X}^{\epsilon_2} -1) - \varphi ( {Y}^{\epsilon_1} -1)\varphi ( {Y}^{\epsilon_2} -1) \Big).
\end{align*}
Since the second term has already been estimated above, it remains to estimate the first term. We consider the case $\epsilon_1=\epsilon_2=1$, noting that the other cases can be handled similarly.
\begin{align*}
\varphi \big((X-1)^2\big) &=\varphi \big((g(x)-1- g^{\prime}(0)x +g^{\prime}(0)x)^2\big) \\ 
&= \varphi\big( (g(x)-1- g^{\prime}(0)x)^2 \big) +2 g'(0)\varphi \big( (g(x)-1- g^{\prime}(0)x)x\big)  + g^{\prime}(0)^2\varphi \big(x^2\big). 
\end{align*}
Then we observe 
\begin{align*}
\varphi\big( (g(x)-1 - g^{\prime}(0)x) x\big) &= \int_0^1 \varphi\big((g'(ux)-g^{\prime}(0))x^2\big) \rmd u \\&=  \int_0^1 \varphi\big((g'(ux)-g^{\prime}(0)-g''(0)ux)x^2\big) \rmd u +  \int_0^1 \varphi\big(g''(0)ux^3\big)\rmd u.
\end{align*}
Thus, by  Taylor-Lagrange Theorem and the fact that $g^{\prime\prime}$ is $\gamma$-H\"older continuous, we infer
\begin{align*}
 \Big|& \varphi\big( (g(x)-1- g^{\prime}(0)x ) x\big)\Big| 
\\&\leq  \int_0^1\int_{\mathbb{R}} |g'(ut)-g^{\prime}(0)-g''(0)ut| |t|^2\nu_x(\rmd t)\rmd u + | g''(0)|{\int_0^1}\int_{\mathbb{R}}u|t|^3\nu_x(\rmd t) \rmd u\\  
&\leq  \int_0^1\int_{\mathbb{R}} \|g''\|_{\gamma,[(ut)^-,(ut)^+]} |t|^{3+\gamma}\nu_x(\rmd x)\rmd u +  | g''(0)|\varphi(|x|^3).
\end{align*}
Now, by applying H\"older's inequality, we obtain, for any $\zeta \geq 0$, 
\begin{align*}
\Big| \varphi\big( (g(x)-1- g^{\prime}(0)x ) x\big)\Big| &\leq  \big\|\|g''\|_{\gamma,[x^-,x^+]} \big\|_{L^{1+\zeta^{-1}}_{\varphi}} \varphi(|x|^{(3+\gamma)(1+\zeta)})^{\frac{1}{1+\zeta}}+| g''(0)|\varphi(|x|^3) \\
&\lesssim \big\|\|g''\|_{\gamma,[x^-,x^+]}\big\|_{L^{1+\zeta^{-1}}_{\varphi}} \|x\|_{L^{(3+\gamma)(1+\zeta)}_{\varphi}}^{3+\gamma}+\|x\|_{L^3_{\varphi}}^3.
\end{align*}
Then we move to estimate the term $|\varphi\big((g(x)-1- g^{\prime}(0)x)^2\big)|$. 
Following the same steps as above, we write
\begin{align*}
|\varphi\big((g(x)-1- g^{\prime}(0) x)^2\big)| 
&\leq \int_{\mathbb{R}} \big(\big|g(t)-1-g^{\prime}(0)t-\frac{1}{2}g''(0)t^2\big|+\big|\frac{1}{2}g''(0)t^2\big|\big)^2\nu_x(\rmd t) 
\\&=\int_{\mathbb{R}} \big(\frac{1}{2}|g''(\xi)-g''(0)|+\frac{1}{2}|g''(0)|\big)^2|t|^4\nu_x(\rmd t),
\end{align*}
for some $\xi$ between $0$ and $t$. Then using the fact that $g^{\prime\prime}$ is $\gamma$-H\"older continuous,  the convexity of the square function followed by H\"older's inequality, we obtain for $\zeta \geq 0$
\begin{align*}
|\varphi\big((g(x)-1-g'(0)x)^2\big)| &\leq \frac{1}{2}\int_{\mathbb{R}} \big(\|g''\|_{\gamma,[t^-,t^+]}^2|t|^2+|g''(0)|^2\big)|t|^4\nu_x(\rmd t).
\\ & \lesssim  \big\|\|g''\|_{\gamma,[x^-,x^+]}\big\|_{L^{2(1+\zeta^{-1})}_{\varphi}}^2\|x\|_{L^{6(1+\zeta)}_{\varphi}}^6 +  \|x\|_{L^{4}_{\varphi}}^4.
\end{align*}
By applying a similar approach to the term corresponding to $Y$, we conclude that
\begin{align}\label{eqn: bounded btw difference second moment}
&\Big| \varphi\big((X-1)^2\big) - \varphi\big((Y-1)^2\big) 
\Big| \nonumber\\
&\lesssim \frac{1}{2}\Big(\big\|\|g''\|_{\gamma,[x^-,x^+]}\big\|_{L^{2(1+\zeta^{-1})}_{\varphi}}^2 \|x\|_{L^{6(1+\zeta)}_{\varphi}}^6  +  
\big\|\|h''\|_{\gamma,[y^-,y^+]}\big\|_{L^{2(1+\zeta^{-1})}_{\varphi}}^2 \|y\|_{L^{6(1+\zeta)}_{\varphi}}^6 
+\|x\|_{L^{4}_{\varphi}}^4 +  \|y\|_{L^{4}_{\varphi}}^4 \Big) \nonumber \\ 
&\hspace{0.5cm}+ \big\|\|g''\|_{\gamma,[x^-,x^+]} \big\|_{L^{1+\zeta^{-1}}_{\varphi}} \|x\|_{L^{(3+\gamma)(1+\zeta)}_{\varphi}}^{3+\gamma}+  \big\|\|h''\|_{\gamma,[y^-,y^+]} \big\|_{L^{1+\zeta^{-1}}_{\varphi}} \|y\|_{L^{(3+\gamma)(1+\zeta)}_{\varphi}}^{3+\gamma}+ \|x\|_{L^3_{\varphi}}^3 + \|y\|_{L^3_{\varphi}}^3 \nonumber
\\& \hspace{0.5cm}+ \big(g^{\prime}(0)^2\|x\|_{L^2_{\varphi}}^2 -  h^{\prime}(0)^2\|y\|_{L^2_{\varphi}}^2\big).  
\end{align}
Note that in the cases of $(\epsilon_1,\epsilon_2)=(1,\star),(\star,1),$ or $(\star,\star)$, these terms can be bounded by the same estimate as in \eqref{eqn: bounded btw difference second moment}, so we will omit the computations. Hence, we establish that 
\begin{align*}
\big\|  {\kappa}^{\varphi^{(2)}}_2 &\big(\bm{X}^{\epsilon_1} -I_2, \bm{X}^{\epsilon_2} -I_2\big) - {\kappa}^{\varphi^{(2)}}_2 \big(\bm{Y}^{\epsilon_1} -I_2 ,  \bm{Y}^{\epsilon_2} -I_2\big)\big\|_{M_2(\mathbb{C})} \nonumber\\
&\lesssim   \Big(\big\| \|g\|_{1,[x^-,x^+]} \big\|_{L^{1+\zeta^{-1}}_\varphi} \|x\|_{L^{1+\zeta}_\varphi}+\big\| \|h\|_{1,[y^-,y^+]} \big\|_{L^{1+\zeta^{-1}}_\varphi} \|y\|_{L^{1+\zeta}_\varphi}\Big)
\\& \qquad \times \Big(\big\| \|  g^{\prime\prime} \|_{\gamma,[x^-,x^+]} \big\|_{L^{1+\zeta^{-1}}_\varphi} \|x\|_{L^{(2+\gamma)(1+\zeta)}_\varphi}^{{2+\gamma}} +  \big\| \|  h^{\prime\prime} \|_{\gamma,[y^-,y^+]} \big\|_{L^{1+\zeta^{-1}}_\varphi} \|y\|_{L^{(2+\gamma)(1+\zeta)}_\varphi}^{{2+\gamma}}
\notag \\ &\qquad \quad + \big| g^{\prime\prime}(0)\| x\|_{L^2_\varphi}^2 - h^{\prime\prime}(0)\|y\|_{L^2_\varphi}^2\big|\Big) \\
&+ \big\|\|g''\|_{\gamma,[x^-,x^+]}\big\|_{L^{2(1+\zeta^{-1})}_{\varphi}}^2 \|x\|_{L^{6(1+\zeta)}_{\varphi}}^6 +  
\big\|\|h''\|_{\gamma,[y^-,y^+]}\big\|_{L^{2(1+\zeta^{-1})}_{\varphi}}^2 \|y\|_{L^{6(1+\zeta)}_{\varphi}}^6 
  \nonumber \\ 
&+\|x\|_{L^{4}_{\varphi}}^4 + \|y\|_{L^{4}_{\varphi}}^4+\big\|\|g''\|_{\gamma,[x^-,x^+]} \big\|_{L^{1+\zeta^{-1}}_{\varphi}} \|x\|_{L^{(3+\gamma)(1+\zeta)}_{\varphi}}^{3+\gamma} \\&+ \big\|\|h''\|_{\gamma,[y^-,y^+]} \big\|_{L^{1+\zeta^{-1}}_{\varphi}} \|y\|_{L^{(3+\gamma)(1+\zeta)}_{\varphi}}^{3+\gamma}+\|x\|_{L^3_{\varphi}}^3 +\|y\|_{L^3_{\varphi}}^3\\& + \big|
g^{\prime}(0)^{\epsilon_1}g^{\prime}(0)^{\epsilon_2}\|x\|_{L^2_{\varphi}}^2 - h^{\prime}(0)^{\epsilon_1} h^{\prime}(0)^{\epsilon_2}\|y\|_{L^2_{\varphi}}^2\big|.  
\end{align*}
Hence, taking into account the monotonicity of $L^{p}_{\varphi}$-norms, we obtain the desired bound.
\end{proof}
\subsection{Technical lemmas}
\label{technicallemmas}

\begin{lemma}
\label{lemma exponential identity}
    Let $X$ and $Y$ be two bounded operators. Then for any $m \in \mathbb{N}$, 
   \begin{align*}
       e^{X+Y} -e^X&= \sum_{k=1}^m \int_{\mathcal{T}_k} e^{\alpha_0X}Y\cdots e^{\alpha_{k-1}X}Ye^{\alpha_kX} \rmd \alpha   + \int_{\mathcal{T}_{m+1}} e^{\alpha_0X}Y\cdots e^{\alpha_{m}X}Ye^{\alpha_{m+1}(X+Y)} \rmd \alpha
    \\ & =  \sum_{k=1}^m \int_{\mathcal{T}_k} e^{\alpha_0X}Y\cdots e^{\alpha_{k-1}X}Ye^{\alpha_kX} \rmd \alpha   + \int_{\mathcal{T}_{m+1}} e^{\alpha_0 (X+Y)}Y\cdots e^{\alpha_{m}X}Ye^{\alpha_{m+1}X} \rmd \alpha,
   \end{align*}
where $\mathcal{T}_k=\{\alpha=(\alpha_0, \dots , \alpha_k) \ | \sum_{i=0}^k \alpha_i =1 \; \text{and} \; 0\leq \alpha_i \leq 1 \; \text{ for } i =0,\dots, k \}$ and $\rmd \alpha = \rmd \alpha_0 \cdots \rmd \alpha_k$.  
\end{lemma}
\begin{proof}
We   start by noting that for any $t \in [0,1]$, 
\begin{align}\label{derivation-iteration}
    e^{t(X+Y)}-e^{tX}
&= \int_0^t \frac{\rmd}{\rmd \alpha}\big(e^{(t-\alpha) X} e^{\alpha (X+Y)}\big) \rmd \alpha
= \int_0^t e^{\alpha X} Y e^{(t-\alpha)(X+Y)} \rmd \alpha,
\end{align}
which yields 
\begin{align*}
    e^{X+Y}-e^{X}
&= \int_0^1 e^{\alpha X} Y e^{(1-\alpha)(X+Y)} \rmd \alpha
\\ &= \int_{\mathcal{T}_1} e^{\alpha_0 X} Y  e^{\alpha_1 X} \rmd \alpha  +\int_{\mathcal{T}_1} e^{\alpha_0 X} Y \big(e^{\alpha_1(X+Y)} - e^{\alpha_1 X}\big) \rmd \alpha .
\end{align*}
We end the proof of the first identity using \eqref{derivation-iteration} and iterating $m$ times the above identity. The second follows in the same way.
\end{proof}
For any $X \in \cA$, we shall denote by  $\bm{X} \in M_2(\cA)$  its Hermitization  matrix  defined by
\begin{equation}\label{Hermitization matrix}
    \bm{X}=\begin{bmatrix}
    0 & X \\ X^\star & 0
\end{bmatrix}.
\end{equation}

\begin{lemma}\label{Prop Fourier cosine}
Let $X$ be a bounded operator. Set $|X|= (XX^{\star})^{\frac{1}{2}}$ and let $\bm{X}$ be its Hermitization matrix defined in \eqref{Hermitization matrix}. Then for any $u\in \mathbb{R}$, the following relation holds in $\mathcal{M}_2(\mathcal{A})$:
\[
e^{{\sf i}u \bm{X}
%\begin{bmatrix} 0 & X \\ X^{\star} & 0 \end{bmatrix}
} = \begin{bmatrix} \cos(u|X|) & {\sf i}u X^{\star}{\rm sinc}(u|X|) \\ {\sf i}u X\ {\rm sinc}(u|X^{\star}|)&\cos(u|X^{\star}|) \end{bmatrix}.
\]
In particular, 
\begin{equation*}
(\tr_2 \otimes \varphi)(e^{{\sf i}u \bm{X}}) = \varphi(\cos(u|X|)).
\end{equation*}
    
\end{lemma}
\begin{proof}
For any $u \in \mathbb{R}$, we have that
\begin{align*}
 e^{{\sf i}u\bm{X}}&= \sum_{\ell \geq 0}\frac{{\sf i}^\ell u^\ell}{\ell !}\bm{X}^\ell = \sum_{\ell \geq 0}\frac{(-1)^{\ell} u^{2\ell}}{(2\ell) !}\bm{X}^{2\ell} + {\sf i}\sum_{\ell \geq 0}\frac{(-1)^{\ell} u^{2\ell+1}}{(2\ell +1) !}\bm{X}^{2\ell +1}
 \\ &= \sum_{\ell \geq 0}\frac{(-1)^{\ell} u^{2\ell}}{(2\ell) !} \begin{bmatrix} |X|^{2\ell} &0 \\ 0 & |X^\star|^{2\ell} \end{bmatrix} + {\sf i}u\sum_{\ell \geq 0}\frac{(-1)^{\ell} u^{2\ell}}{(2\ell +1) !}\begin{bmatrix} 0 & X^{\star} |X|^{2\ell}\\ X |X^{\star}|^{2\ell} & 0 \end{bmatrix}
 \\ &  = \begin{bmatrix} \cos(u|X|) & {\sf i}u X^{\star}{\rm sinc}(u|X|) \\ {\sf i}u X \ {\rm sinc}(u|X^{\star}|)&\cos(u|X^{\star}|) \end{bmatrix}.
\end{align*}
Note that $\varphi(\cos(u|X|)) = \varphi(cos(u|X^{\star}|))$, we infer the second part of the proof.
\end{proof}

Finally, we recall the following inequality that we used to get uniform bounds for the estimates obtained under the Lindeberg method.
\begin{lemma}\label{thm:NC L^p-inequality}
Let $(\cA,\varphi)$ be a tracial probability space  and $X_1,\dots,X_n$ be  freely independent bounded variables such that $\varphi(X_i)\neq 0$ for any $i=1,\dots,n$. The partial product $\pi_{i,n}^{X} = X_1\cdots X_i$ then satisfies
$$
\|\pi_{i,n}^X\|_{L^p_{\varphi}}\leq \frac{1}{|\varphi(X_{i+1})\cdots \varphi(X_n)|} \|\pi_{n,n}^X\|_{L^p_{\varphi}}. 
$$
\end{lemma}
\begin{proof}
Let $\mathcal{F}_{i}$ be the von Neumann algebra generated by $X_1,\dots,X_i$. Noting that 
\begin{align*}
\frac{X_1}{\varphi(X_1)}\cdots \frac{X_{i}}{\varphi(X_{i})} = \varphi\Big(\frac{X_1}{\varphi(X_1)}\cdots \frac{X_{n}}{\varphi(X_{n})}\Big|\mathcal{F}_{i}\Big),
\end{align*}
we obtain by the fact that the conditional expectation $\varphi (\cdot|\mathcal{F}_i)$ is a contraction on $L^p_{\varphi}$ that
\begin{align*}
\Big\|\frac{\pi_{i,n}^X}{\varphi(X_1)\cdots \varphi(X_i)}\Big\|_{L^p_{\varphi}}= \Big\|\varphi \Big(\frac{\pi_{n,n}^X}{\varphi(X_1)\cdots \varphi(X_n)}\Big| \mathcal{F}_i\Big)\Big\|_{L^p_{\varphi}}
\leq\Big\|\frac{\pi_{n,n}^X}{\varphi(X_1)\cdots \varphi(X_n)}\Big\|_{L^p_{\varphi}}.
\end{align*}
By simplifying the above expression, we obtain the desired result.
\end{proof}
\subsection{Proofs of Lemmas \ref{lem:unboundedlipschitz} and \ref{lemma:UnbddZo}}
\label{sec:estimatefouriermoments}
\begin{proof}[Proof of Lemma \ref{lem:unboundedlipschitz}]
First, note that since $\|f\|_1\leq 1$ and $f(0)=0$, then for any $x,y\in \mathbb{R}$, $|f(x)|\leq |x|$,
\begin{align}
\label{eqn:lip}
|f_{\zeta}(x)-f_{\zeta}(x+y)| &\leq |f(x)-f(x+y)||\theta(\zeta (x+y))| + |\theta(\zeta x)-\theta(\zeta(x+y))||f(x)| \nonumber \\
& \leq |y|(1+\zeta |x|).
\end{align}
This implies, for any $x\in\mathbb{R}$:
\begin{align*}
|f_{\zeta} - f_{\zeta}\star\rho_{\varepsilon}(x)| \leq 
\int_{\mathbb{R}}|f_{\zeta}(x) - f_{\zeta}(x+y)||\rho_{\varepsilon}(y)| \mathrm{d}y \leq \varepsilon(1+\zeta |x|)\int_{\mathbb{R}}|y|\rho(y)\mathrm{d}y.
\end{align*}
Moreover, for any $x\in \mathbb{R}$: 
\begin{align*}
 |f(x) -f_{\zeta}(x)| &=
|f(x)||(1-\theta_{\zeta}(x))| = |f(x)||(\theta_{\zeta}(0)-\theta_{\zeta}(x))|
\leq \zeta  |x|^2
\end{align*}
and finally, for any $y \in \mathbb{R}$, by using the estimate \eqref{eqn:lip} for $|x|\leq \frac{1}{\zeta}+ \frac{\pi}{|y|}$ and $y := \frac{\pi}{|y|}$:
\begin{align}
\label{eqn:riemanlebesgueslipschitz}
|\hat{f_{\zeta}}(y)| = \frac{1}{2}|\widehat{f_{\zeta}-(f_{\zeta})_{\frac{\pi}{y}}}(y)| &\leq \frac{1}{2}\|f_{\zeta}-(f_{\zeta})_{\frac{\pi}{y}} \|_{L^{1}_{[-\frac{1}{\zeta}-\frac{\pi}{|y|},\frac{1}{\zeta}+\frac{\pi}{|y|}]}} \nonumber\\ 
&\lesssim \bigg(\frac{\pi}{|y|} + \frac{1}{\zeta}\bigg)\bigg(1 + {\zeta}\bigg(\frac{\pi}{|y|} + \frac{1}{\zeta}\bigg)\bigg)  \frac{\pi}{|y|}.
\end{align}
To estimate $I_k$, we first assume $k \geq 3$ and use the above inequality to obtain
\begin{align*}
I_k &\leq\int_{\mathbb{R}} |y|^k|\hat{\rho}_\varepsilon(y)\hat{f_{\zeta}}(y)| \mathrm{d} y =\frac{1}{\varepsilon^{k+1}}\int_{\mathbb{R}} |u|^k|\hat{\rho}(u)\hat{f_{\zeta}}(\frac{u}{\varepsilon})| \mathrm{d} y \nonumber  \\ 
&\lesssim \frac{1}{\varepsilon^{k}}\int_{\mathbb{R}}|u|^{k-1}|\hat{\rho}(u)| \bigg(\frac{\varepsilon}{|u|} + \frac{1}{\zeta}\bigg)\bigg(1 + {\zeta}\bigg(\frac{\varepsilon}{|u|} + \frac{1}{\zeta}\bigg)\bigg) \mathrm{d}u =O\Big(\frac{1}{\varepsilon^{k}\zeta}\Big), 
\end{align*}
since  $\rho$ is smooth with rapidly decaying derivatives. We note that $O(\frac{1}{\varepsilon^{k}\zeta})$ is uniform on Lipschitz functions $f$ with $\|f\|_1 \leq 1$ and depends only on the $k-3, k-2, k-1 \geq 0$ absolute moments of $\hat{\rho}$. We consider now the cases when $ k < 3$ and write as before:
\begin{equation}
I_k\leq \int_{\mathbb{R}} |y|^k|\hat{\rho}_\varepsilon(y)\hat{f_{\zeta}}(y)| \mathrm{d} y  = \frac{1}{\varepsilon^{k+1}}\int_{|u|\leq \varepsilon} |u|^k|\hat{\rho}(u)\hat{f_{\zeta}}(\frac{u}{\varepsilon})| \mathrm{d} u + \frac{1}{\varepsilon^{k+1}}\int_{|u|\geq \varepsilon} |u|^k|\hat{\rho}(u)\hat{f_{\zeta}}(\frac{u}{\varepsilon})| \mathrm{d} u. 
\end{equation}
The first integral is bounded as,
\begin{align}
\label{eqn:firstintegral}
\frac{1}{\varepsilon^{k+1}}\int_{|u|\leq \varepsilon} |u|^k|\hat{\rho}(u)\hat{f_{\zeta}}(\frac{u}{\varepsilon})| \mathrm{d} u \lesssim  \|f_{\zeta}\|_{L^1} \lesssim \frac{1}{\zeta^2}, 
\end{align}
since $|f_{\zeta}(x)| \leq |x|$.
For $k=2$, we obtain using the bound in \eqref{eqn:riemanlebesgueslipschitz} that
\begin{align*}
    \frac{1}{\varepsilon^3}\int_{|u|\geq \varepsilon} |u|^2|\hat{\rho}(u)\hat{f_{\zeta}}(\frac{u}{\varepsilon})| \mathrm{d} u &\lesssim \frac{1}{\varepsilon^2}\int_{|u|\geq \varepsilon} |u||\hat{\rho}(u)|\bigg(\frac{\varepsilon}{|u|} + \frac{1}{\zeta}\bigg)\bigg(1 + {\zeta}\bigg(\frac{\varepsilon}{|u|} + \frac{1}{\zeta}\bigg)\bigg) \mathrm{d} u .
\end{align*}
Expanding the right-hand side of the last inequality, we obtain a sum of terms. After calculating each integral, using the fact that $\hat{\rho}$ is a Schwartz function, we find that the leading order is $O(\varepsilon^{-2}\zeta^{-1})$. For $k=1$, we obtain using the bound in \eqref{eqn:riemanlebesgueslipschitz} that
\begin{align*}
    \frac{1}{\varepsilon^2}\int_{|u|\geq \varepsilon} |u||\hat{\rho}(u)\hat{f_{\zeta}}(\frac{u}{\varepsilon})| \mathrm{d} u &\lesssim \frac{1}{\varepsilon}\int_{|u|\geq \varepsilon}|\hat{\rho}(u)|\bigg(\frac{\varepsilon}{|u|} + \frac{1}{\zeta}\bigg)\bigg(1 + {\zeta}\bigg(\frac{\varepsilon}{|u|} + \frac{1}{\zeta}\bigg)\bigg) \mathrm{d} u .
\end{align*}
We expand the right-hand side of the above equality to obtain a sum of terms; each is $O(\frac{1}{\varepsilon\zeta})$. The estimates of the proposition are obtained after taking into account \eqref{eqn:firstintegral}.
\end{proof}

\begin{proof}[Proof of Lemma \ref{lemma:UnbddZo}]
Let   $\theta$ and $\theta_\zeta$ be defined as in Lemma \ref{lem:unboundedlipschitz}. Then,
\begin{align*}
& |(f-f_{\zeta})(x)| = |(1-\theta_{\zeta}(x))f(x)| = |(\theta_{\zeta}(0)-\theta_{\zeta}(x))f(x)| \leq \zeta |x| |x|^r = \zeta |x|^{r+1}. \\
\end{align*}
To proceed, we first give an upper bound on $|f_{\zeta}^{(\ell)}(z) - f_{\zeta}^{(\ell)}(x)|$. For this, we recall that by Leibniz rule applied to $f_{\zeta}^{(\ell)}$, we have
\begin{align}
f_{\zeta}^{(\ell)}=(\theta_{\zeta }f)^{(\ell)} = \sum_{k=0}^{\ell}\binom{\ell}{k} \theta_{\zeta}^{(k)}f^{(\ell-k)},
\end{align}
  which yields for any $|x|,|z| \leq R $ with $R \geq \frac{1}{\zeta}$ and $\zeta \leq 1$, that
\begin{align*}
|&f_{\zeta}^{(\ell)}(z) - f_{\zeta}^{(\ell)}(x)| \\&\leq \sum_{k=0}^{\ell}\binom{\ell}{k} |\theta_{\zeta}^{(k)}(x)||f^{(\ell-k)}(z)-f^{(\ell-k)}(x)| + \sum_{k=0}^{\ell}\binom{\ell}{k}|\theta^{(k)}_{\zeta}(x)-\theta^{(k)}_{\zeta}(z)||f^{(\ell-k)}(z)| % && \textrm{by the Leibniz identity} 
\\&\leq \sum_{k=0}^{\ell}\binom{\ell}{k}\zeta^{k}|f^{(\ell-k)}(z)-f^{(\ell-k)}(x)| + \sum_{k=0}^{\ell}\binom{\ell}{k}|\theta^{(k)}_{\zeta}(x)-\theta^{(k)}_{\zeta}(z)||f^{(\ell-k)}(z)|
%&& \textrm{ by }
\\
&\lesssim |x-z|^{\beta} + \sum_{k=1}^{\ell}\binom{\ell}{k}\zeta^{k}|x-z| \sup_{-R\leq u\leq R} |f^{{(\ell-k+1)}}(u)| + \sum_{k=0}^{\ell}\binom{\ell}{k}\zeta^{k} |x-z||f^{(\ell-k)}(z)|,
\end{align*}
where the second and third inequalities follow by the fact that $|\theta_{\zeta}^{(k)}(x)| = |\zeta^k \theta^{(k)}(x)| \leq \zeta^k $.
To control the derivatives of $f$, we first note that $f$ and all its derivatives up to the order $\ell$ vanish at $0$. Then by Taylor-Lagrange, we obtain for some $\theta_x$ between $0$ and $x$ 
\begin{align*}
   \big| f^{(\ell - k)}(x) \big| &= \big|  f^{(\ell -k)}(0) + f^{(\ell - k + 1)}(0)x + \cdots + \frac{1}{(k-1)!}f^{(\ell - 1)}(0)x^{k-1} + \frac{1}{k!}f^{(\ell)}(\theta_x)x^{k}\big|
   \\ &= \frac{1}{k!} |f^{(\ell)} (\theta_x) - f^{(\ell)} (0)| |x|^{k} 
   \\& \leq \frac{1}{k!}  |x|^{r} ,
   \end{align*}
where we recall that $r=k+\beta$. Hence, using the above bound we obtain for any $|x|,|z| \leq R $
\begin{align}\label{est: l derivative feta}
|f_{\zeta}^{(\ell)}(z) - f_{\zeta}^{(\ell)}(x)| 
&\lesssim |x-z|^{\beta} + \sum_{k=1}^{\ell}\binom{\ell}{k}\zeta^{k}|x-z|\frac{R^{r-\ell+k-1}}{(k-1)!}  + \sum_{k=0}^{\ell}\binom{\ell}{k}\zeta^{k} |x-z|\frac{R^{r-\ell+k}}{k!} \nonumber\\
&\lesssim |x-z|^{\beta} + \sum_{k=1}^{\ell}\binom{\ell}{k}(R\zeta)^{k}|x-z|\frac{R^{r-\ell-1}}{(k-1)!}  + \sum_{k=0}^{\ell}\binom{\ell}{k}(R\zeta)^{k} |x-z|\frac{R^{r-\ell}}{k!} \nonumber \\
&\lesssim (|x-z|^{\beta} \wedge 1)(|x-z| \vee 1)(1+R^{r}\zeta^{\ell}),
\end{align}
where we have used the fact that $R \geq \frac{1}{\zeta}$ and $r\geq 1$ to write the last inequality. 
Noting that $\hat{\rho}^{(k)}(0)=0$ for $k \leq \ell$ and $\hat{\rho}(0)=1$, then 
$$
\int_{\mathbb{R}} \rho(y) \mathrm{d}y = 1,\quad \int_{\mathbb{R}} y^{k}\rho(y) \mathrm{d}y = 0,\quad 1 \leq k \leq \ell,
$$
which by the steps to get \eqref{est: l derivative feta} and Taylor-Lagrange formula yield that
\begin{align*}
|&f_{\zeta}(x)-f_{\zeta} \star \rho_{\varepsilon}(x)| \\&= \Big|\int_{\mathbb{R}} (f_{\zeta}(x) - f_{\zeta}(x+y)) \rho_{\varepsilon}(y){\rm d}y \Big| \\ 
&= \Big|\int_{\mathbb{R}} \big(f_{\zeta}(x) - f_{\zeta}(x+y) + f_\zeta^{\prime}(x)y + \cdots + \frac{f_\zeta^{(\ell)}(x)}{\ell!}y^{\ell}\big){\rho}_{\varepsilon}(y){\rm d}y\Big| \\
&\leq \int_{\mathbb{R}} |f_{\zeta}^{(\ell)}(\xi_{x,y}) -f_{\zeta}^{(\ell)}(x)||y|^{\ell}|{\rho}_{\varepsilon}(y)| {\rm d}y \quad \text{for some  } \xi_{x,y} \text{  between } x \text{  and  } x+y\\
&\lesssim \int_{\mathbb{R}} \sup_{z\in[x,x+y]}|f_{\zeta}^{(\ell)}(z) -f_{\zeta}^{(\ell)}(x)||y|^{\ell} |{\rho}_{\varepsilon}(y)|{\rm d}y \\
&\lesssim \int_{\mathbb{R}} (|y|^{\beta} \wedge 1)(|y| \vee 1)(1+(|y|^r \vee |x+y|^{r})\zeta^{\ell}) |y|^{\ell}|{\rho}_{\varepsilon}(y)| {\rm d}y.
\end{align*}
To estimate the above integral, we consider the domains corresponding to $|y| \geq 1$ and $|y| \leq 1$. For the former case,
\begin{align*}
\int_{|y|\geq 1} (|y|^{\beta}  \wedge 1)&(|y| \vee 1)(1+ (|y|^r \vee |x+y|^{r})\zeta^{\ell}) |y|^{\ell}|{\rho}_{\varepsilon}(y)| {\rm d}y \\ & \lesssim \int_{|y|\geq 1} |y| (1+|x|^r+|y|^{r}\zeta^{\ell})|{\rho}_{\varepsilon}(y)| {\rm d}y 
\\&= \int_{|u|\geq \varepsilon^{-1}} \varepsilon|u| (1+|x|^r+|\varepsilon u|^{r}\zeta^{\ell})\varepsilon^{\ell}|u^{\ell}||{\rho}(u)| {\rm d}u \\
&={\varepsilon^{\ell+1}} \int_{|u|\geq \varepsilon^{-1}}{|u|} (1+|{x}|^r+|{\varepsilon}{u}|^{r}\zeta^{\ell}) |{u^{\ell}}||{\rho}(u)| {\rm d}u \\
&\lesssim {\varepsilon^{\ell+1}} \int_{|u|\geq  \varepsilon^{-1}}{|u|} (1+|{x}|^r+|{u}|^{r})|{\rho}(u)||u|^{\ell} {\rm d}u .
\end{align*}
Since $\rho$ is rapidly decaying, the integral in the right-hand side of the above inequality tends to zero rapidly as $\varepsilon$ goes to $0$,
and the bound is $o(\varepsilon^{\ell+1})$.
As for the second part, 
\begin{align*}
\int_{|y|\leq 1} &(|y|^{\beta} \wedge 1)(|y| \vee 1)(1+|x+y|^{r}\zeta^{\ell})|{\rho}_{\varepsilon}(y)| {\rm d}y \\&\lesssim 
\int_{|u|\leq \varepsilon^{-1}}\varepsilon^{\ell+\beta}|u|^{\beta}(1+|x|^r+|\varepsilon u|^{r}\zeta^{\ell})|u|^{\ell}|{\rho}(u)| {\rm d}y \\
&\leq \varepsilon^{r}\int_{|u|>0}|u|^{\beta}(1+|x|^r+|u|^{r})|u|^{\ell} |{\rho}(u)| {\rm d}y.
\end{align*}
Putting the above bounds together, we obtain \eqref{est:f-feta Unbdd Z}. 
Now by choosing $R = \frac{1}{\zeta} + \frac{\pi}{|y|}$ in the estimate \eqref{est: l derivative feta}, we infer
\begin{multline*}
|\widehat{f_{\zeta}^{(\ell)}}(y)| = 
\frac{1}{2}|\widehat{f_{\zeta}^{(\ell)}}(y)-(\widehat{f_{\zeta}^{(\ell)})_{\frac{\pi}{y}}}(y)|
\lesssim \|f_{\zeta}^{(\ell)}-(f_{\zeta}^{(\ell)})_{\frac{\pi}{y}}\|_{L^1([-\frac{1}{\zeta}-\frac{\pi}{|y|},\frac{1}{\zeta}+\frac{\pi}{|y|}])} \\ 
\lesssim \big(\frac{1}{|y|^{\beta}} \wedge 1\big)\big(\frac{1}{|y|} \vee 1\big) \Big(1+\big(\frac{1}{\zeta} + \frac{1}{|y|}\big)^{r}\zeta^{\ell}\Big)\big(\frac{1}{\zeta} + \frac{1}{|y|}\big).
\end{multline*}
Since for any $y\in \mathbb{R}$, $|\widehat{f^{(\ell)}}(y)| = |y^{\ell}\hat{f}(y)|$, we infer that 
\begin{equation}\label{est: l-derivative.FT.UNbddZolotarev}
|\widehat{f_{\zeta}}(y)| \lesssim \big(\frac{1}{|y|^{\beta}} \wedge 1\big)\big(\frac{1}{|y|} \vee 1\big) \Big(1+\big(\frac{1}{\zeta} + \frac{1}{|y|}\big)^{r}\zeta^{\ell}\Big)\big(\frac{1}{\zeta} + \frac{1}{|y|}\big)\frac{1}{|y|^{\ell}}.
\end{equation}
To estimate the integral $I_k$, we cut it into two parts, corresponding to the domains $|y|\leq 1$ and $|y|\geq 1$,
\begin{align*}
I_k=\int_{|y|\leq 1} |y|^k |\hat{\rho}_\varepsilon(y)\hat{f}_{\zeta}(y)| \mathrm{d} y + \int_{|y|\geq 1} |y|^k |\hat{\rho}_\varepsilon(y)\hat{f}_{\zeta}(y)| \mathrm{d} y,
\end{align*}
and, after a change of variables,  we obtain for the former

$$
\frac{1}{\varepsilon^{k+1}}\int_{|u|\leq \varepsilon} |u|^k |\hat{\rho}(u)|\hat{f}(\frac{u}{\varepsilon})| \mathrm{d} u \leq \|\hat{f}_{\zeta}\|_{\infty} \|\hat{\rho} \| \lesssim \| f_{\zeta}\|_{L^1} \leq \int_{-\frac{1}{\zeta}}^{\frac{1}{\zeta}} |\theta_\zeta(x)||f(x)| \mathrm{d} x \lesssim \frac{1}{\zeta^{r+1}}.
$$
As for the other integral,  we  use the bound in \eqref{est: l-derivative.FT.UNbddZolotarev} and get
\begin{align*}
\frac{1}{\varepsilon^{k+1}} &
\int_{|u|\geq \varepsilon} |u|^k |\hat{\rho}(u)|\hat{f}_{\zeta}(\frac{u}{\varepsilon})| \mathrm{d} u
\\ &\leq \frac{1}{\varepsilon^{k-\ell+1}}\int_{|u|\geq \varepsilon} |u|^{k-\ell} |\hat{\rho}(u)| \big(\frac{\varepsilon^{\beta}}{|u|^{\beta}} \wedge 1 \big)\big(\frac{\varepsilon}{|u|} \vee 1\big)\Big(1+\big(\frac{1}{\zeta} + \frac{\varepsilon}{|u|}\big)^{r}\zeta^{\ell}\Big)\big(\frac{1}{\zeta} + \frac{\varepsilon}{|u|}\big) \ \mathrm{d} u \\
&  = \frac{1}{\varepsilon^{k-\ell+1}}\int_{|u|\geq \varepsilon} |u|^{k-\ell} |\hat{\rho}(u)|\frac{\varepsilon^{\beta}}{|u|^{\beta}} \Big(1+\big(\frac{1}{\zeta} + \frac{\varepsilon}{|u|}\big)^{r}\zeta^{\ell}\Big)\big(\frac{1}{\zeta} + \frac{\varepsilon}{|u|}\big) \ \mathrm{d}  u\\
& \lesssim  \frac{1}{\varepsilon^{k-r+1}}\int_{|u|\geq \varepsilon} |u|^{k-r} |\hat{\rho}(u)|\Big(1+\big(\frac{1}{\zeta} + \frac{\varepsilon}{|u|}\big)^{r}\zeta^{\ell}\Big)\big(\frac{1}{\zeta} + \frac{\varepsilon}{|u|}\big) \ \mathrm{d}  u\\
&  \lesssim \frac{1}{\varepsilon^{k-r+1}}\int_{|u|\geq \varepsilon} |u|^{k-r} |\hat{\rho}(u)|\Big(1+\frac{1}{\zeta^{\beta}} + \frac{\zeta^{\ell}\varepsilon^{r}}{|u|^{r}}\Big)\big(\frac{1}{\zeta} + \frac{\varepsilon}{|u|}\big) \ \mathrm{d}  u \\
&  \lesssim \frac{1}{{\zeta^{\beta+1}}}\frac{1}{\varepsilon^{k-r+1}}\int_{|u|\geq \varepsilon} |u|^{k-r} |\hat{\rho}(u)|\big(1 + \frac{\varepsilon^{r}}{|u|^{r}}\big)\big(1+ \frac{\varepsilon}{|u|}\big) \ \mathrm{d}  u.
\end{align*}
From the above integral, we identify the critical values of $k=r-1$, $k=r$, $k=2r-1$ and $k=2r$. To determine the order of $I_k$, we distinguish all possible cases. Whenever $k > 2r$, the integral in the right-hand side converges as $\epsilon$ goes to zero, and thus $I_k = O({{\zeta^{-(\beta+1)}}\varepsilon^{-(k-r+1)}})$. On the other hand, for  $k \leq 2r$, we have to sort out the dominant term as $\varepsilon$ goes to $0$ among the four following integrals:
$$
I_k^1 = \frac{1}{\varepsilon^{k-r+1}}\int_{|u|\geq \varepsilon} |u|^{k-r} |\hat{\rho}(u)| \mathrm{d} u,\quad I_k^2=  \frac{1}{\varepsilon^{k-2r+1}}\int_{|u|\geq \varepsilon} |u|^{k-2r} |\hat{\rho}(u)| \mathrm{d} u, $$
$$
I_k^3 = \frac{1}{\varepsilon^{k-r}}\int_{|u|\geq \varepsilon} |u|^{k-r-1} |\hat{\rho}(u)| \mathrm{d} u, 
\quad I_k^4 = \frac{1}{\varepsilon^{k-2r}}\int_{|u|\geq \varepsilon} |u|^{k-2r-1} |\hat{\rho}(u)| \mathrm{d} u .
 $$
$\bullet$\, For $k = 2r$: It is easy to see that the integrals in $I_k^1$, $I_k^2$ and $I_k^3$ are respectively of order $O({\varepsilon^{-(r+1)}})$, $O({\varepsilon^{-1}}) $, and $O({\varepsilon^{-r}})$. As for the last integral, we use similar arguments to show that $I_k^4 = O(|\ln(\varepsilon)|)$.  Hence, we infer that 
$$I_{2r} = O\Big(\frac{1}{\zeta^{\beta+1}\varepsilon^{r+1}}\Big).
$$
$\bullet$\, For  $ 2r-1 < k  < 2r$: One has $k  - r >  r - 1 > 0 $ and consequently it follows that $
I^1_k = O({\varepsilon^{-(k-r+1)}} )$. On the other hand, as $0<k-2r<0$ and $k-r-1>0$, it follows using the fact that $\hat{\rho}$ is a Schwartz function that $I^2_k = O({\varepsilon^{-(k-2r+1)}})$ and $I^3_k = O({\varepsilon^{-(k-r)}})$. Finally, noting that $\|\hat{\rho}\|\leq \|\rho\|_{L^1} <\infty$, we infer that $I^4_k =  O(1)$. Consequently, we obtain that 
$$
I_k = O\Big(\frac{1}{\zeta^{\beta+1}\varepsilon^{k+1-r}}\Big).$$
Using similar arguments, one can prove the following for the remaining cases:

\noindent
$\bullet$\, For  $k=2r-1 $: 
$$
I^1_{2r-1} =  O\big(\frac{1}{\varepsilon^{r}} \big),\;  I^2_{2r-1} =  O(|\ln(\varepsilon)|),\; I_{2r-1}^3 = O\big(\frac{1}{\varepsilon^{r-1}}\big), \; I^4_{2r-1}= O(\varepsilon), $$ and hence  $I_{2r-1} = O\Big(\frac{1}{\zeta^{\beta+1}\varepsilon^{r}}\Big).$

\noindent
$\bullet$\, For  $r<k<2r-1 $: 
$$
I^1_{k} =  O\big(\frac{1}{\varepsilon^{k-r+1}} \big),\; 
I^2_{k} =  O(\frac{1}{\varepsilon^{k-2r+1}}),\; I_{k}^3 = O\big(\frac{1}{\varepsilon^{k-r}}\big), \; I^4_{k}= O(\frac{1}{\varepsilon^{k-2r}}),$$ and hence $I_{2r-1} =  O\Big(\frac{1}{\zeta^{\beta+1}\varepsilon^{k-r+1}}\Big).
$

\noindent
$\bullet$\, For  $k=r$: 
$$
I^1_{r} =  O\big(\frac{1}{\varepsilon} \big),\;  I^2_{r} =  O(\varepsilon^{r-1}),\; I_{r}^3 = O\big(|\ln(\varepsilon)|\big), \; I^4_{r}= O(\varepsilon^r),$$  and hence $I_{r} = O\Big(\frac{1}{\zeta^{\beta+1}\varepsilon}\Big).
$

\noindent
$\bullet$\, For  $ r-1<k <r $: 
$$
I^1_{k} =  O\big(\frac{1}{\varepsilon^{k-r+1}} \big),\;  I^2_{k} =  O\big(\frac{1}{\varepsilon^{k-2r+1}} \big),\; I_{k}^3 = O\big(\frac{1}{\varepsilon^{k-r}}\big), \; I^4_{k}= O\big(\frac{1}{\varepsilon^{k-2r}} \big), $$ and hence  $I_{k} = O\Big(\frac{1}{\zeta^{\beta+1}\varepsilon^{k-r+1}}\Big).
$

\noindent
$\bullet$\, For  $k=r-1$:
$$
I^1_{r-1} =  O (|\ln(\varepsilon)|),\;  I^2_{r-1} =  O (1),\; I_{r-1}^3 = O\big(1\big), \; I^4_{r-1}= O\big(1),$$ and hence $I_{r-1} = O\Big(\frac{|\ln(\varepsilon)|}{\zeta^{\beta+1}}\Big).
$

\noindent
$\bullet$\, For  $k < r-1$: all the integrals are of order $1$ when $\varepsilon$ goes to $0$, hence 
$$I_k=O\Big(\frac{1}{\zeta^{\beta+1}}\Big).$$
Considering the above bounds we obtain the estimates desired. 
\end{proof}

\subsection{A combinatorial proof of the free multiplicative CLT}\label{Appendix:combinatorialProof}
\begin{theorem}\label{thm: moment CLT}
Suppose $(x_i)_{i\geq 1}$ is a sequence of centered self-adjoint random variables with the same variance $\sigma^2$. For a fixed $ k\geq 1$, assume 
\begin{equation}\label{Assumption}
\sup_{i\geq 1} {\varphi(x^{k(1+\zeta)})} < + \infty
\end{equation}
Let $g\colon \mathbb{R}\to \mathbb{C}$ be such that $g \in \C^{2,\gamma}(\mathbb{R})$ and $g(0)=1$. Assume that there exists $0<\alpha<1$ such that
\begin{align*}
    &\sup_{i\geq 1} \| \||g^2|\|_{1, [\alpha x^{-}_i,\alpha x^{+}_i]} \|_{k(1+\zeta^{-1})},\quad \sup_{i\geq 1} \| \||g^2|'\|_{1, [\alpha x^{-}_i,\alpha x^{+}_i]} \|_{k(1+\zeta^{-1})},\\&\sup_{i\geq 1} \| \||g^2|''\|_{\gamma, [\alpha x^{-}_i,\alpha x^{+}_i]} \|_{k(1+\zeta^{-1})} < +\infty,
\end{align*}
for some $\zeta > 0$ and $\alpha > 0$,  where    $x_i^-:=\min(x_i,0)$, $x_i^+:=\max(x_i,0)$ and  $\|.\|_{L_\varphi^p}$ denotes the $L^p$ norm with respect to $\varphi$. Then, 
\begin{align*}
\lim_{n\to\infty} \kappa^{\varphi}_k \big(\pingx(\pingx)^{\star}\big)=\frac{k^{k-1}}{k!}(|g^2|^{\prime}(0)\sigma)^{2(k-1)}e^{\frac{k}{2} |g^2|^{\prime\prime}(0)\sigma^2}.
\end{align*}
\end{theorem}

\begin{remark}
Note that if we take $g(x)=\exp(x)$, then $|g^2|'(0)=\frac{1}{2}|g^2|''(0)=2$, which yields
$$
\lim_{n\to\infty} \kappa^{\varphi}_k \big(\pingx(\pingx)^{\star}\big)=\frac{k^{k-1}}{k!}(2\sigma)^{2(k-1)}e^{2k\sigma^2},
$$
in agreement with Equation \eqref{eqn: free cumulant of p.s.c} and Ho's result.
\end{remark}

\begin{remark} \label{remark:unbddCLT}
It is worth emphasizing that Theorem \ref{thm: moment CLT}, combined with the results of \cite{ber08, chistyakov08}, establishes the free multiplicative central limit theorem (CLT) under the Wasserstein distance and even extends to the unbounded setting. Specifically, \cite{ber08, chistyakov08} guarantee the existence and convergence of the limit, while Theorem \ref{thm: moment CLT} identifies the corresponding cumulants. In contrast to the approach of Ho, this combinatorial approach applies in the unbounded setting as well and yields convergence in the Wasserstein distance to a broader class of products involving functions of freely independent,   not necessarily identically distributed elements.
\end{remark}
To prove Theorem~\ref{thm: moment CLT}, we begin by controlling the cumulants of $|g|^2(\frac{x_i}{\sqrt{n}})$ for all $i=1,\dots, n$ as illustrated in the following lemma.

\begin{lemma}\label{lem: Multi-CLT}
Under the same assumptions as in Theorem \ref{thm: moment CLT}, we have 
\begin{align*}
   \kappa^{\varphi}_1\Big(|g|^2\big(\frac{x_i}{\sqrt{n}}\big)\Big) &= 1 + \big(\Re(g^{\prime\prime}(0)) + |g^{\prime}(0)|^2\big) \frac{\sigma^2}{n} + O\Big(\frac{1}{n^{1+\frac{\gamma}{2}}}\Big),  \\
   \kappa^{\varphi}_2\Big(|g|^2\Big(\frac{x_i}{\sqrt{n}}\Big)\Big) &= \frac{\sigma^2}{n}|g^2|'(0)^2 +O\Big(\frac{1}{n^{\frac{3}{2}}}\Big), \\
   \kappa^{\varphi}_k\Big(|g|^2\Big(\frac{x_i}{\sqrt{n}}\Big)\Big) &\leq  \frac{4^{k-1}}{n^{k/2}}\big\|   \||g|^2\|_{1,[\alpha x_i^-, \alpha x_i^+]}    \big\|^k_{L_{\varphi}^{k(1+\frac{1}{\zeta})}}\|x_i\|_{L_{\varphi}^{k(1+\zeta)}}^k, \qquad \text{for } k\geq3.
\end{align*}
\end{lemma}
\begin{proof}
Let $\mathcal{T}^{|g|^2}_p$  denote the $p$-th order Taylor polynomial of $|g|^2$ around the origin. Then, for any $i=1,\dots,n$,   $\mathcal{T}^{|g|^2}_2(\frac{x_i}{\sqrt{n}})$ is given  by
\[
\mathcal{T}^{|g|^2}_2\Big(\frac{x_i}{\sqrt{n}}\Big) = |g(0)^2| +|g(0)^2| ' \frac{x_i}{\sqrt{n}}+ \frac{1}{2} |g^2|^{\prime\prime}(0) \frac{x^2_i}{n}.\] 
We first note that by the Taylor-Lagrange theorem, we have \begin{align*}
\Big|\varphi\Big(|g|^2 \Big(\frac{x_i}{\sqrt{n}}\Big) - \mathcal{T}^{|g|^2}_2 \Big(\frac{x_i}{\sqrt{n}}\Big)\Big)\Big| 
&\leq \frac{1}{2}\int_{\mathbb{R}}\Big||g^2|''\Big(\frac{\xi}{\sqrt{n}}\Big) - |g^2|''(0) \Big| \frac{|t|^2}{n} \nu_{x_i}(\rmd t), 
\end{align*}
for some $\xi$ between $0$ and $t$. 
Using the fact that $g^{\prime\prime}$ is $\gamma$-H\"older continuous then applying H\"older's inequality, we deduce that for any $\zeta \geq 0$,
\begin{align*}
\Big|\varphi\Big(|g|^2 \Big(\frac{x_i}{\sqrt{n}}\Big) - \mathcal{T}^{|g|^2}_2 \Big(\frac{x_i}{\sqrt{n}}\Big)\Big)\Big| &\leq \frac{1}{n^{1+\frac{\gamma}{2}}}\int_{\mathbb{R}} \| |g^2|^{\prime\prime}\|_{ \gamma,[\frac{t^-}{\sqrt{n}},\frac{t^+}{\sqrt{n}}]} |t|^{2+\gamma}\nu_{x_i}(\rmd t) \\
&\leq 
\frac{1}{n^{1+\frac{\gamma}{2}}}\Big\| \| |g^2|^{\prime\prime}\|_{ \gamma, [\frac{x_i^-}{\sqrt{n}},\frac{x_i^+}{\sqrt{n}}]} \Big\|_{L_{\varphi}^{1+\frac{1}{\zeta}}}\|x_i\|_{L_{\varphi}^{(2+\gamma)(1+\zeta)}}^{2+\gamma},
\end{align*}
 which is of order $n^{-(1+\gamma/2)}$ under  Assumption \eqref{Assumption}. On the other hand, we note, by the fact that $\varphi(x_i)=0$ and $g(0)=1$, that
\begin{align*}
\varphi\Big(\mathcal{T}^{|g|^2}_2\Big(\frac{x_i}{\sqrt{n}}\Big)\Big) &= |g(0)|^2 + \frac{1}{2} |g^2|^{\prime\prime}(0) \varphi \Big(\frac{x^2_i}{n}\Big) \\
&= |g(0)|^2 + \big(\Re(g^{\prime\prime}(0)\bar{g}(0)) + |g^{\prime}(0)|^2\big) \frac{\sigma^2}{n} 
= 1 + \big(\Re(g^{\prime\prime}(0)) + |g^{\prime}(0)|^2\big) \frac{\sigma^2}{n}.
\end{align*}
Hence, we conclude from the above computations and bounds that
$$
\kappa^{\varphi}_1\Big(|g|^2\big(\frac{x_i}{\sqrt{n}}\big)\Big) = \varphi\Big(|g|^2\Big(\frac{x_i}{\sqrt{n}}\Big)\Big) = 1 + \big(\Re(g^{\prime\prime}(0)) + |g^{\prime}(0)|^2\big) \frac{\sigma^2}{n} + O\Big(\frac{1}{n^{1+\frac{\gamma}{2}}}\Big).
$$
For the cumulant of second order, we first write
\begin{align*}\label{X cumulant order 2}
    \kappa^{\varphi}_2&\Big(|g|^2\Big(\frac{x_i}{\sqrt{n}}\Big)\Big) 
    \\&=\varphi\Big(\Big(|g|^2\big(\frac{x_i}{\sqrt{n}}\big)-1\Big)^2\Big) - \varphi\Big(|g|^2\big(\frac{x_i}{\sqrt{n}}\big)-1\Big)^2  \\
    &= \varphi\Big(\Big(|g|^2\big(\frac{x_i}{\sqrt{n}}\big)-\mathcal{T}^{|g|^2}_{1}\big(\frac{x_i}{\sqrt{n}}\big) + |g^2|^{\prime}(0)\frac{x_i}{\sqrt{n}}\Big)^2\Big) - \varphi\Big(|g|^2\big(\frac{x_i}{\sqrt{n}}\big)-\mathcal{T}^{|g|^2}_1\big(\frac{x_i}{\sqrt{n}}\big)\Big)^2.
\end{align*}
where the last equation holds because $x_i$ is centered with respect to $\varphi$. Now, by Taylor-Lagrange theorem, we note that
\begin{align*}
\varphi\Big(|g|^2\big(\frac{x_i}{\sqrt{n}}\big)-\mathcal{T}^{|g|^2}_1\big(\frac{x_i}{\sqrt{n}}\big)\Big) 
&= \frac{1}{n}\int_{\mathbb{R}} \||g^2|'\|_{ 1, [\frac{t_i^-}{\sqrt{n}},\frac{t_i^+}{\sqrt{n}}]} |t|^{2}\nu_{x_i}({\rmd t})\\&\leq  \frac{1}{n} \Big\| \| |g^2|^{\prime}\big\|_{ 1, [\frac{x_i^-}{\sqrt{n}},\frac{x_i^+}{\sqrt{n}}]} \Big\|_{L_{\varphi}^{1+\frac{1}{\zeta}}}\|x_i\|_{L_{\varphi}^{2(1+\zeta)}}^{2},  
\end{align*}
where the last inequality follows for any choice of $\zeta \geq 0$ by applying H\"older's inequality.
As for the second-order moment, we first write
\begin{align*}
&\varphi\Big(\big(|g|^2\big(\frac{x_i}{\sqrt{n}}\big)-1\big)^2\Big)\\& = 
\varphi\Big(\Big(|g|^2\big(\frac{x_i}{\sqrt{n}}\big)-\mathcal{T}^{|g|^2}_{1}\big(\frac{x_i}{\sqrt{n}}\big) + |g^2|^{\prime}(0)\frac{x_i}{\sqrt{n}}\Big)^2\Big) \\&=\varphi \Big( \Big(|g|^2\big(\frac{x_i}{\sqrt{n}}\big)-\mathcal{T}^{|g|^2}_{1}\big(\frac{x_i}{\sqrt{n}}\big)\Big)^2\Big) + 2|g^2|'(0)\varphi\Big(\Big(|g|^2\big(\frac{x_i}{\sqrt{n}}\big)-\mathcal{T}^{|g|^2}_{1}\big(\frac{x_i}{\sqrt{n}}\big)\Big)\frac{x_i}{\sqrt{n}}\Big) + \frac{\sigma^2}{n}|g^2|'(0)^2.
\end{align*}
We now show that the first two terms on the right-hand side above are each of order $O(n^{-3/2})$. To begin, we observe  by Taylor-Lagrange theorem that
\begin{align*}
\varphi &\Big( \Big(|g|^2\big(\frac{x_i}{\sqrt{n}}\big)-\mathcal{T}^{|g|^2}_{1}\big(\frac{x_i}{\sqrt{n}}\big)\Big)^2\Big)\\&\leq
\frac{1}{n^2}\int_{\mathbb{R}} \||g^2|'\|_{1,[\frac{t^-}{\sqrt{n}},\frac{t^+}{\sqrt{n}}]}^2 |t|^{4}\nu_{x_i}(\rmd t) \leq \frac{1}{n^2}\big\| \| |g^2|^{\prime}\|_{ 1, [\frac{x_i^-}{\sqrt{n}},\frac{x_i^+}{\sqrt{n}}]} \Big\|_{L_{\varphi}^{1+\frac{1}{\zeta}}}^2\|x_i\|_{L_{\varphi}^{2(1+\zeta)}}^{4},
\end{align*}
where the last inequality follows again for any $\zeta\geq 0$ by applying H\"older's inequality. Following similar arguments, we also show that 
$$
\varphi\Big(\Big(|g|^2\big(\frac{x_i}{\sqrt{n}}\big)-\mathcal{T}^{|g|^2}_{1}\big(\frac{x_i}{\sqrt{n}}\big)\Big)\frac{x_i}{\sqrt{n}}\Big) \leq \frac{1}{n^{\frac{3}{2}}}\Big\| \| |g^2|^{\prime}\|_{1,[\frac{x_i^-}{\sqrt{n}},\frac{x_i^+}{\sqrt{n}}]} \Big\|_{L_{\varphi}^{1+\frac{1}{\zeta}}}\|x_i\|_{L_{\varphi}^{3(1+\zeta)}}^{3}.
$$
We conclude that
$$
\kappa^{\varphi}_2\Big(|g|^2\big(\frac{x_i}{\sqrt{n}}\big)\Big) = \frac{\sigma^2}{n}|g^2|'(0)^2 +O\Big(\frac{1}{n^{\frac{3}{2}}}\Big).
$$
Note that for $k\geq 3$, we have by the cumulant-moment formula that $$\kappa^{\varphi}_k\Big(|g|^2\big(\frac{x_i}{\sqrt{n}}\big)\Big)=\kappa^{\varphi}_k\Big(|g|^2\big(\frac{x_i}{\sqrt{n}}\big)-1\Big)=\sum_{\pi\in NC(k)}\mu(\pi,1_k)\varphi_{\pi}\Big(|g|^2\big(\frac{x_i}{\sqrt{n}}\big)-1,\dots,|g|^2\big(\frac{x_i}{\sqrt{n}}\big)-1\Big),$$
where $\mu$ is the M\"{o}bius function on the poset of non-crossing partitions. 
Observe that for any $3\leq p \leq k$, we have by  Taylor-Lagrange theorem
$$
\Big|\varphi\big((|g|^2(\frac{x_i}{\sqrt{n}})-1)^p\big)\Big| 
%=\Big|\int_{\mathbb{R}} \big(|g|^2(\frac{t}{\sqrt{n}})-1\big)^p\nu_{x_i}(\rmd t)\Big| 
\leq \int_{\mathbb{R}} \big| |g|^2(\frac{t}{\sqrt{n}})-1\big|^p\nu_{x_i}(\rmd t)
\leq \frac{1}{n^{\frac{p}{2}}} \int_{\mathbb{R}}\| |g|^2\|_{1,[\frac{t^-}{\sqrt{n}},\frac{t^+}{\sqrt{n}}]}^p |t|^p \nu_{x_i}(\rmd t),
$$
where $\nu_{x_i}$ is the spectral measure of $x_i$. Then, by Hölder's inequality, we obtain for any $\xi \geq 0$,   
\begin{align*}
\int_{\mathbb{R}}\big| |g(\frac{t}{\sqrt{n}})|^2 - 1\big|^p \nu_{x_i}( \rmd t) &
%\leq \frac{1}{n^{\frac{p}{2}}} \int_{\mathbb{R}}\| |g|^2\|_{1,[\frac{t^-}{\sqrt{n}},\frac{t^+}{\sqrt{n}}])}^p |t|^p \nu_{x_i}(\rmd t) \\&
\leq \frac{1}{n^{\frac{p}{2}}}\Big\| \||g|^2\|^p_{1,[\frac{x_i^-}{\sqrt{n}},\frac{x_i^+}{\sqrt{n}}]} \Big\|_{L^{1+\frac{1}{\zeta}}_\varphi} \varphi\big(|x_i|^{p(1+\zeta)}\big)^{\frac{1}{1+\zeta}} 
\\
&\leq \frac{1}{n^{\frac{p}{2}}} \Big\|   \||g|^2\|_{1,[\frac{x_i^-}{\sqrt{n}},\frac{x_i^+}{\sqrt{n}}]}  \Big\|^p_{L^{k(1+\frac{1}{\zeta})}_\varphi} \|x_i\|^p_{L^{k(1+\zeta)}_\varphi}.
\end{align*}
The proof is completed by applying the cumulant-moment formula together with the estimate $\mu(\cdot, 1_k) \leq 4^{k-1}$, see \cite{NicaSpeicher2006}.
\end{proof}
To prove Theorem \ref{thm: moment CLT}, we introduce some notation and recall relevant terminology from \cite{Arizmendi12}. Let $n$ and $k$ be positive integers, and define
$$
\rho_n^k:=\{(1,\dots,n),(n+1,\dots,2n),\dots,((n-1)k+1,\dots,nk)\}.
$$
A partition $\pi\in NC(kn)$ is said to be \emph{$k$-equal} if all its blocks have size $k$. The set of $k$-equal partitions of $[kn]$ is denoted $NC_n(k)$. A partition $\pi\in NC(kn)$ is termed \emph{$k$-completing} if each block of $\pi$ comprises only numbers with same congruence modulo $k$, and $\pi\vee \rho_{n}^k=1_{nk}:=\{(1,2,\dots,nk-1,nk)\}$ where
$NC(kn)$ is equipped with a natural lattice structure.  $\pi\in NC_n(k)$ if and only if ${\rm Kr}(\pi)$ is $k$-completing. Moreover, the free cumulants of the product of free random variables $a_1,\dots,a_n \in (\mathcal{A},\varphi)$ satisfy, for each integer $k$, 
$$
\kappa^{\varphi}_k(a_1\cdots a_n)=\sum_{\pi\in NC_n(k)}\kappa^{\varphi}_{{\rm Kr}(\pi)}(a_1,\dots,a_n). 
$$
In the following, $NC(n,k)_{2,1}$ will denote the set of $k$-equal partitions $\pi\in NC_n(k)$ with Kreweras complement ${\rm Kr}(\pi)$ consisting of only pairings and singletons. 
\begin{proof}[Proof of Theorem \ref{thm: moment CLT}]
Since the state $\varphi$ is tracial, 
$\kappa^{\varphi}_k(\pingx(\pingx)^\star)=\kappa^{\varphi}_k(Y_n)$ where $Y_n=|g|^2(\frac{x_1}{\sqrt{n}})\cdots |g|^2(\frac{x_n}{\sqrt{n}})$. 
Therefore, we write
\begin{align*}
\kappa^{\varphi}_k(Y_n)&=\sum_{\pi\in NC_n(k)}\kappa^{\varphi}_{{\rm Kr}(\pi)}(|g|^2(\frac{x_1}{\sqrt{n}}),\dots,|g|^2(\frac{x_n}{\sqrt{n}})) \\
&=\sum_{\pi\in NC(n,k)_{2,1}}\kappa^{\varphi}_{{\rm Kr}(\pi)}\big(|g|^2(\frac{x_1}{\sqrt{n}}),\dots,|g|^2(\frac{x_n}{\sqrt{n}})\big)\\& \qquad+\sum_{\pi\in NC_n(k)\setminus NC(n,k)_{2,1}}\kappa^{\varphi}_{{\rm Kr}(\pi)}\big(|g|^2(\frac{x_1}{\sqrt{n}}),\dots,|g|^2(\frac{x_n}{\sqrt{n}})\big),   
\end{align*}
where we recall $NC(n,k)_{2,1}$ is the set of $k$-equal partitions on $[nk]$ for which ${\rm Kr}(\pi)$ has only pairings and singletons.
We will show 
$$
\lim_{n\to\infty} S_1(n,k)=\frac{1}{k!} k^{k-1}\big(|g^2|^{\prime}(0)^2\sigma^2\big)^{k-1}e^{\frac{k}{2}|g^2|^{\prime\prime}(0)\sigma^2}  \quad\text{and} \quad\lim_{n\to\infty}S_2(n,k)=0,
$$
where
\begin{align*}
S_1(n,k)&= \sum_{\pi\in NC(n,k)_{2,1}}\kappa^{\varphi}_{{\rm Kr}(\pi)}\big(|g|^2(\frac{x_1}{\sqrt{n}}),\dots,|g|^2(\frac{x_n}{\sqrt{n}})\big)  \text{  and } \\
S_2(n,k)&= \sum_{\pi\in NC_n(k)\setminus NC(n,k)_{2,1}}\kappa^{\varphi}_{{\rm Kr}(\pi)}\big(|g|^2(\frac{x_1}{\sqrt{n}}),\dots,|g|^2(\frac{x_n}{\sqrt{n}})\big).
\end{align*}
Observe that for each partition $\pi\in NC(n,k)_{2,1}$,
$$\kappa^{\varphi}_{{\rm Kr}(\pi)}\big(|g|^2(\frac{x_1}{\sqrt{n}}),\dots,|g|^2(\frac{x_n}{\sqrt{n}})\big)=\prod_{s=1}^{k-1}\kappa^{\varphi}_2\big(|g|^2(\frac{x_{i_s}}{\sqrt{n}})\big)\cdot \prod_{r=1}^{kn-2(k-1)}\kappa^{\varphi}_1\big(|g|^2(\frac{x_{j_r}}{\sqrt{n}})\big) $$
for some $i_1,\dots,i_{k-1}\in [n]$ and $j_1,\dots,j_{kn-2(k-1)}\in [n]$.  Lemma \ref{lem: Multi-CLT} then yields
\begin{align*}
&\kappa^{\varphi}_{{\rm Kr}(\pi)}\big(|g|^2(\frac{x_1}{\sqrt{n}}),\dots,|g|^2(\frac{x_n}{\sqrt{n}})\big)\\
&=\Big(\frac{1}{n}|g^2|^{\prime}(0)^2\sigma^2+{ O}(\frac{1}{n^{\frac{3}{2}}})\Big)^{k-1}\Big(1 + \big(\Re(g^{\prime\prime}(0)) + |g^{\prime}(0)|^2\big) \frac{\sigma^2}{n} + O(\frac{1}{n^{1+\frac{\gamma}{2}}})\Big)^{kn-2(k-1)} \\
&=\Big(\frac{1}{n}|g^2|^{\prime}(0)^2\sigma^2+{ O}(\frac{1}{n^{\frac{3}{2}}})\Big)^{k-1}\Big(1 + \frac{1}{2}|g^2|''(0)\frac{\sigma^2}{n} + O(\frac{1}{n^{1+\frac{\gamma}{2}}})\Big)^{kn-2(k-1)} \\
&= \frac{1}{n^{k-1}} (|g^2|^{\prime}(0)^2\sigma^2)^{k-1}e^{\frac{k}{2}|g^2|^{\prime\prime}(0)\sigma^2 } + O(\frac{1}{n^{\frac{3}{2}(k-1)}}).
\end{align*}
The cardinality of $NC(k,n)_{2,1}$ has been found in \cite [Eq. (2.5)]{Arizmendi12}  to be equal to 
$$
|NC(n,k)_{2,1}|= \frac{n(kn-k)!}{(kn-k-(k-2))!(k-1)!} =\frac{1}{k!}k^{k-1}n^{k-1}+O(n^{k-2}), 
$$
which implies 
\begin{align*}
|NC(k,n)_{2,1}|\frac{1}{n^{k-1}} (|g^2|^{\prime}(0)^2\sigma^2)^{k-1}e^{\frac{k}{2}|g^2|^{\prime\prime}(0)\sigma^2 } = \frac{1}{k!} k^{k-1}\big(|g^2|^{\prime}(0)^2\sigma^2\big)^{k-1}e^{\frac{k}{2}|g^2|^{\prime\prime}(0)\sigma^2} + O({n^{-\frac{1}{2}(k-1)}}).
\end{align*}
We infer finally 
$$
\lim_{n\to\infty} S_1(n,k) =\frac{1}{k!} k^{k-1}\big(|g^2|^{\prime}(0)^2\sigma^2\big)^{k-1}e^{\frac{k}{2}|g^2|^{\prime\prime}(0)\sigma^2}.
$$
Now, we are left with proving that:
$ \lim_{n\to\infty}S_2(n,k)=0.$
For each partition $\pi\in NC_n(k)\setminus NC(n,k)_{2,1},$ there exists $m\leq k-1$ and $2\leq k_1,\dots,k_m\leq k$ with $k_j\geq 3$ for some $j$ such that
$$
\kappa^{\varphi}_{kr(\pi)}(X_{1,n},\dots,X_{n,n})=\kappa^{\varphi}_{k_1}\big(|g|^2(\frac{x_{i_1}}{\sqrt{n}})\big)\cdots \kappa^{\varphi}_{k_m}\big(|g|^2(\frac{x_{i_m}}{\sqrt{n}})\big)\prod_{r=1}^{kn-(k_1+\cdots+k_m)}\kappa^{\varphi}_1\big(|g|^2(\frac{x_{j_r}}{\sqrt{n}})\big),
$$
for some $i_1,\dots,i_m,j_1,\dots,j_{kn-(k_1+\dots+k_m)}\in [n]$. 
Applying Lemma \ref{lem: Multi-CLT}, one gets
\begin{equation*}
\kappa^{\varphi}_{{\rm Kr}(\pi)}\big(|g|^2(\frac{x_1}{\sqrt{n}}),\dots,|g|^2(\frac{x_n}{\sqrt{n}})\big)= O\big(n^{-\frac{1}{2}(k_1+\cdots +k_m)}\big).
\end{equation*}
The number of $n$-completing non-crossing partitions with type $(b_2,\ldots,b_k)$ (omitting the number of singletons) is equal to  
$$
n\frac{\big((n-1)k\big)!}{b_1!\cdots b_k!}.
$$
The number of singletons is minimal for pair or singleton partitions and is equal to 
$
nk-2(k-1)
$
and maximal for partitions with only one block of size $k$ and is equal to $nk - k$. Then, 
\begin{align}\label{eqn: estimate of S(n,k)}
S_2(n,k) &\lesssim n\sum_{\substack{k\geq b_j \geq 0, \,\exists  j\geq 3,\, b_j \geq 1, \nonumber \\ b_2+\cdots + b_k = 1+q-k, \\ 2b_2 + \cdots + kb_k = q}} \frac{1}{n^{\frac{2b_2+\cdots+kb_k}{2}}}\frac{\big((n-1)k\big)!}{\big(nk-(2b_2+\cdots+kb_k)\big)! b_2!\cdots b_k!} \nonumber \\%
&=\sum_{q=k}^{2(k-1)-1} \frac{1}{n^{\frac{q}{2}-1}}\frac{\big((n-1)k\big)!}{(nk-q)!}\sum_{\substack{\substack{k\geq b_j \geq 0, \,\exists  j\geq 3,\, b_j \geq 1, \\ b_2+\cdots + b_k = 1+q-k, \\ 2b_2 + \cdots +kb_k = q}}} \frac{1}{ b_2!\cdots b_k!} \nonumber \\%
&=\sum_{q=k}^{2(k-1)-1} \frac{1}{n^{\frac{q}{2}-1}}\frac{\big((n-1)k\big)!}{(nk-q)!}\frac{1}{(1+q-k)!}\sum_{\substack{\substack{k\geq b_j \geq 0, \,\exists  j\geq 3,\, b_j \geq 1, \\ b_2+\cdots + b_k = 1+q-k, \\ 2b_2 + \cdots +kb_k = q}}} \binom{1+q-k}{b_2,\ldots,b_k} \nonumber \\
&\leq 
\sum_{q=k}^{2(k-1)-1} \frac{1}{n^{\frac{q}{2}-1}}\frac{(k-1)^{1+q-k}}{(1+q-k)!}\frac{\big((n-1)k\big)!}{(nk-q)!} 
\end{align}
where the last inequality holds by noting that
$$
(k-1)^{1+q-k} = \sum_{\substack{b_2,\dots,b_k\geq 0, \\ b_2+\cdots + b_k= 1+q-k}} \binom{1+q-k}{b_2,\ldots,b_k}\geq  \sum_{\substack{\substack{k\geq b_j \geq 0, \,\exists  j\geq 3,\, b_j \geq 1, \\ b_2+\cdots + b_k = 1+q-k, \\ 2b_2 + \cdots +kb_k = q}}} \binom{1+q-k}{b_2,\ldots,b_k}.
$$
By substituting $s=q-k$, we can write 
\begin{align*}
\eqref{eqn: estimate of S(n,k)} 
&= \sum_{s=0}^{k-3} \frac{1}{n^{\frac{s+k}{2}-1}} \frac{(k-1)^{1+s}}{(1+s)!}\frac{\big((n-1)k\big)!}{\big((n-1)k-s\big)!} \leq \sum_{s=0}^{k-3} \frac{1}{n^{\frac{s+k}{2}-1}} \frac{(k-1)^{1+s}}{(1+s)!}{(nk)^s} 
\\
&= \sum_{s=0}^{k-3} \frac{1}{n^{\frac{k-s}{2}-1}}  \frac{k^s(k-1)^{1+s}}{(1+s)!}  \leq \frac{1}{n^{\frac{1}{2}}}\sum_{s=0}^k \frac{k^{2s+1}}{s!}\leq \frac{k}{n^{\frac{1}{2}}}e^{k^2}.
\end{align*}
\end{proof}

%%%%%%%%%%%%%%%%%%%%%%%%%%%%%%

\bibliographystyle{alpha}
\bibliography{AOP/AOP}

\newcommand{\etalchar}[1]{$^{#1}$}
\begin{thebibliography}{Kar07b}

\bibitem[ABT22]{Ariz-Banna-Tseng}
Octavio Arizmendi, Marwa Banna, and Pei-Lun Tseng.
\newblock Quantitative estimates for operator-valued and infinitesimal boolean
  and monotone limit theorems.
\newblock {\em To appear in: Indiana University Mathematics Journal. arXiv
  preprint arXiv:2211.08054}, 2022.

\bibitem[AFU24]{Arizmendi24}
Octavio Arizmendi, Katsunori Fujie, and Yuki Ueda.
\newblock New combinatorial identity for the set of partitions and limit
  theorems in finite free probability theory.
\newblock {\em International Mathematics Research Notices}, page rnae089, 2024.

\bibitem[AV12]{Arizmendi12}
Octavio Arizmendi and Carlos Vargas.
\newblock Products of free random variables and {$k$}-divisible non-crossing
  partitions.
\newblock {\em Electron. Commun. Probab.}, 17:no. 11, 13, 2012.

\bibitem[AV25]{AuerVoigt}
Martin Auer and Michael Voit.
\newblock An explicit formula for free multiplicative brownian motions via
  spherical functions.
\newblock {\em Indagationes Mathematicae}, 2025.

\bibitem[BDS17]{bogachev2017weighted}
V~Bogachev, A~Doledenok, and S~Shaposhnikov.
\newblock Weighted {Z}olotarev metrics and the {K}antorovich metric.
\newblock In {\em Doklady Mathematics}, volume~95, 2017.

\bibitem[BM23]{BannaMaiBerry}
Marwa Banna and Tobias Mai.
\newblock Berry-{E}sseen bounds for the multivariate {$\mathcal{B}$}-free {CLT}
  and operator-valued matrices.
\newblock {\em Trans. Amer. Math. Soc.}, 376(6):3761--3818, 2023.

\bibitem[Bog18]{Bogachev-book}
Vladimir~I. Bogachev.
\newblock {\em Weak convergence of measures}, volume 234 of {\em Math. Surv.
  Monogr.}
\newblock Providence, RI: American Mathematical Society (AMS), 2018.

\bibitem[BP00]{ber2000}
Hari Bercovici and Vittorino Pata.
\newblock Limit laws for products of free and independent random variables.
\newblock {\em Studia Mathematica}, 1(141):43--52, 2000.

\bibitem[BV92]{Ber92}
Hari Bercovici and Dan-Virgil Voiculescu.
\newblock L{\'e}vy-hin{\v{c}}in type theorems for multiplicative and additive
  free convolution.
\newblock {\em Pacific journal of mathematics}, 153(2):217--248, 1992.

\bibitem[BW08]{ber08}
Hari Bercovici and Jiun-Chau Wang.
\newblock Limit theorems for free multiplicative convolutions.
\newblock {\em Transactions of the American Mathematical Society},
  360(11):6089--6102, 2008.

\bibitem[CG08a]{ChGo-8}
G.~P. Chistyakov and F.~G{\"o}tze.
\newblock Limit theorems in free probability theory. {I}.
\newblock {\em Ann. Probab.}, 36(1):54--90, 2008.

\bibitem[CG08b]{chistyakov08}
Gennadiy~P Chistyakov and Friedrich G{\"o}tze.
\newblock Limit theorems in free probability theory {II}.
\newblock {\em Cent. Eur. J. Math}, 6(1):87--117, 2008.

\bibitem[CG13]{ChGo-13}
G.~P. Chistyakov and F.~G{\"o}tze.
\newblock Asymptotic expansions in the {CLT} in free probability.
\newblock {\em Probab. Theory Relat. Fields}, 157(1-2):107--156, 2013.

\bibitem[Ho11]{Ho2011}
Keang-Po Ho.
\newblock Central limit for the product of free random variables.
\newblock {\em arXiv preprint arXiv:1101.5220}, 2011.

\bibitem[HZ14]{huang14}
Hao-Wei Huang and Ping Zhong.
\newblock On the supports of measures in free multiplicative convolution
  semigroups.
\newblock {\em Mathematische Zeitschrift}, 278(1):321--345, 2014.

\bibitem[Kar07a]{kargin07free-additive}
Vladislav Kargin.
\newblock Berry-{E}sseen for free random variables.
\newblock {\em Journal of Theoretical Probability}, 20:381--395, 2007.

\bibitem[Kar07b]{kargin2007norm}
Vladislav Kargin.
\newblock The norm of products of free random variables.
\newblock {\em Probability theory and related fields}, 139(3):397--413, 2007.

\bibitem[Kar08]{kargin2008asymptotic}
Vladislav Kargin.
\newblock On asymptotic growth of the support of free multiplicative
  convolutions.
\newblock {\em Electronic Communications in Probability [electronic only]},
  13:415--421, 2008.

\bibitem[MS17]{mingospeicherbook}
James~A Mingo and Roland Speicher.
\newblock {\em Free probability and random matrices}, volume~35.
\newblock Springer, 2017.

\bibitem[MS23]{MaSakuma}
Makoto Maejima and Noriyoshi Sakuma.
\newblock Rates of convergence in the free central limit theorem.
\newblock {\em Stat. Probab. Lett.}, 197:6, 2023.
\newblock Id/No 109802.

\bibitem[Neu24]{neufeld24}
Leonie Neufeld.
\newblock Weighted sums and {B}erry-{E}sseen type estimates in free probability
  theory.
\newblock {\em Probability Theory and Related Fields}, 190(3):803--879, 2024.

\bibitem[NS06a]{NicaSpeicher}
Alexandru Nica and Roland Speicher.
\newblock {\em Lectures on the combinatorics of free probability}, volume 335
  of {\em Lond. Math. Soc. Lect. Note Ser.}
\newblock Cambridge: Cambridge University Press, 2006.

\bibitem[NS06b]{NicaSpeicher2006}
Alexandru Nica and Roland Speicher.
\newblock {\em Lectures on the combinatorics of free probability}, volume~13.
\newblock Cambridge University Press, 2006.

\bibitem[Rio98]{Rio-98}
Emmanuel Rio.
\newblock Minimal distances and ideal distances.
\newblock {\em C. R. Acad. Sci., Paris, S{\'e}r. I, Math.}, 326(9):1127--1130,
  1998.

\bibitem[Ros11]{ross11}
Nathan Ross.
\newblock Fundamentals of {S}tein's method.
\newblock {\em Probability Surveys}, 8:210--293, 2011.

\bibitem[Spe07]{speicher07rate}
Roland Speicher.
\newblock On the rate of convergence and {B}erry-{E}sseen type theorems for a
  multivariate free central limit theorem.
\newblock {\em arXiv preprint arXiv:0712.2974}, 2007.

\bibitem[V{\etalchar{+}}08]{villani08}
C{\'e}dric Villani et~al.
\newblock {\em Optimal transport: old and new}, volume 338.
\newblock Springer, 2008.

\bibitem[Zho14]{zhong14}
Ping Zhong.
\newblock Free {B}rownian motion and free convolution semigroups:
  multiplicative case.
\newblock {\em Pacific Journal of Mathematics}, 269(1):219--256, 2014.

\bibitem[Zho15]{zhong15}
Ping Zhong.
\newblock On the free convolution with a free multiplicative analogue of the
  normal distribution.
\newblock {\em Journal of Theoretical Probability}, 28(4):1354--1379, 2015.

\bibitem[Zol78]{Zol-78}
V.~M. Zolotarev.
\newblock Ideal metrics in the problem of approximating distributions of sums
  of independent random variables.
\newblock {\em Theory of Probability \& Its Applications}, 22(3):433--449,
  1978.

\bibitem[Zol81]{Zol-81}
V.~M. Zolotarev.
\newblock Properties of and relations among certain types of metrics.
\newblock {\em J. Sov. Math.}, 17:2218--2232, 1981.

\end{thebibliography}

\end{document}